\documentclass[reqno, 11pt, a4paper]{amsart}

\usepackage[utf8]{inputenc}
\usepackage[T1]{fontenc}

\usepackage{enumitem}
\usepackage{amsthm}
\usepackage{thmtools}
\usepackage{amssymb}
\usepackage{amsmath}
\usepackage{mathtools}
\usepackage{bm}
\usepackage{cite}
\usepackage{fnpct}
\usepackage{tikz-cd}
\usepackage{graphicx}
\usepackage[only, llbracket, rrbracket]{stmaryrd}
\usepackage{units} 
\usepackage{xcolor}

\usepackage[pdfusetitle, bookmarks, colorlinks, allcolors=blue]{hyperref}
\usepackage{doi}
\usepackage{cprotect}

\newtheorem{theorem}{Theorem}[section]

\newtheorem{lemma}[theorem]{Lemma}

\theoremstyle{definition}

\newtheorem{definition}[theorem]{Definition}
\newtheorem{example}[theorem]{Example}
\newtheorem{proposition}[theorem]{Proposition}
\newtheorem{remark}[theorem]{Remark}

\numberwithin{equation}{section}

\newcommand{\define}[1]{\emph{#1}}

\let\OldParagraph\paragraph
\renewcommand{\paragraph}[1]{\smallskip\OldParagraph{\textbf{#1}}}

\newcommand{\CommaSpace}{,\,}
\newcommand{\EndComma}{,}
\newcommand{\EndFullStop}{.}

\newcommand{\ComposedWith}{\circ}
\newcommand{\Inverse}{^{-1}}

\newcommand{\Prime}{^\prime}
\newcommand{\RestrictedTo}[1]{|_{ #1 }}
\renewcommand{\epsilon}{\varepsilon}
\renewcommand{\phi}{\varphi}

\newcommand{\EquivalenceClass}[1]{\left[ #1 \right]}
\newcommand{\Intersection}{\cap}
\newcommand{\Map}{\operatorname{Map}}

\newcommand{\Set}[1]{\left\{ #1 \right\}}
\newcommand{\SetCondition}[2]{\left\{ #1 \ : \  #2 \right\}}

\newcommand{\Union}{\cup}

\newcommand{\Isomorphism}{\cong}

\newcommand{\Conjugate}[1]{\overline{ #1 }}
\newcommand{\Complex}{\mathbb{C}}
\newcommand{\ComplexNumbers}{\mathbb{C}}

\newcommand{\ExponentialNumber}{e}
\newcommand{\Factorial}[1]{#1!}
\newcommand{\Floor}[1]{\left\lfloor #1 \right\rfloor}
\newcommand{\ImaginaryNumber}{i}

\newcommand{\Infinity}{\infty}
\newcommand{\Integers}{\mathbb{Z}}
\newcommand{\IntervalClosed}[2]{\left[ #1 , #2 \right]}
\newcommand{\IntervalClosedOpen}[2]{\left[ #1 , #2 \right)}
\newcommand{\IntervalOpen}[2]{\left( #1 , #2 \right)}
\newcommand{\IntervalOpenClosed}[2]{\left( #1 , #2 \right]}
\newcommand{\Modulus}[1]{\left| #1 \right|}
\newcommand{\NaturalNumbers}{\mathbb{N}}
\newcommand{\PiNumber}{\pi}

\newcommand{\RationalNumbers}{\mathbb{Q}}

\newcommand{\RealNumbers}{\mathbb{R}}

\newcommand{\SmallO}{o}
\newcommand{\LiebnitzDerivative}[2]{\frac{\mathrm{d} { #1 }}{\mathrm{d} { #2 }}}

\newcommand{\Norm}[1]{\left\| #1 \right\|}

\newcommand{\wrt}[1]{\, \mathrm{d} { #1 } }

\newcommand{\Commutator}[2]{\LieBracket{#1}{#2}}

\newcommand{\LieBracket}[2]{\left[ #1 , #2 \right]}
\newcommand{\ProjectiveSpecialLinearGroup}{\operatorname{PSL}}

\newcommand{\Automorphisms}{\operatorname{Aut}}

\newcommand{\Determinant}{\operatorname{det}}
\newcommand{\Dimension}{\dim}
\newcommand{\DirectSum}{\oplus}

\newcommand{\Identity}{\mathrm{Id}}

\newcommand{\Kernel}{\ker}

\newcommand{\Projective}{\mathbb{P}}

\newcommand{\ArbitraryIndex}{\ast}
\newcommand{\Argument}{\cdot}
\newcommand{\CochainComplex}{\operatorname{C}}
\newcommand{\Cohomology}{\operatorname{H}}

\newcommand{\FundamentalGroup}{\pi_1}

\newcommand{\Inclusion}{\hookrightarrow}
\newcommand{\PullBack}{^\ast}
\newcommand{\PushForward}{_\ast}
\newcommand{\Tensor}{\otimes}
\newcommand{\Units}{^{\times}}

\newcommand{\Boundary}{\partial}
\newcommand{\Circle}{{S^1}}

\newcommand{\ConnectedSum}{\#}
\newcommand{\ContractibleLoopSpace}{\mathcal{L}}

\newcommand{\CupProduct}{\smile}
\newcommand{\Disc}{D}
\newcommand{\FirstChernClass}{c_1}
\newcommand{\FundamentalClass}[1]{\left[#1\right]}

\newcommand{\HomotopyGroup}{\pi}
\newcommand{\InfiniteComplexProjectiveSpace}{\Complex \Projective ^\infty}
\newcommand{\InfiniteSphere}{S^\infty}

\newcommand{\CotangentSpace}{{T^{\ast}}}
\newcommand{\DifferentiableMaps}{C}

\newcommand{\Derivative}{D}
\newcommand{\Embedding}{\hookrightarrow}
\newcommand{\EvaluateAt}[1]{|_{ #1 }}
\newcommand{\ExteriorDerivative}{\mathrm{d}}
\newcommand{\InteriorDerivative}{\imath}
\newcommand{\SmoothFunctions}{C^\infty}
\newcommand{\SmoothMaps}{C^\infty}
\newcommand{\Sphere}{S^2}
\newcommand{\HighDimensionalSphere}[1]{S^{ #1 }}
\newcommand{\TangentSpace}{T}

\newcommand{\MorseIndex}{\operatorname{ind}}
\newcommand{\StableManifold}{W^s}
\newcommand{\UnstableManifold}{W^u}

\newcommand{\StandardLineBundle}{\mathcal{O}}

\newcommand{\ActionFunctional}{\mathcal{A}}
\newcommand{\Energy}{E}
\newcommand{\HamiltonianSymplectomorphismGroup}{\operatorname{Ham}}
\newcommand{\HamiltonianVectorField}[1]{X_{ #1 }}

\newcommand{\ContractibleLoopSpaceWithFillings}[1]{\widetilde{\mathcal{L}{ #1 }}}
\newcommand{\InterestingSpheres}{\Gamma}
\newcommand{\NovikovRing}{\Lambda}
\newcommand{\QuantumCohomology}{\operatorname{QH}}
\newcommand{\QuantumCochainComplex}{\operatorname{QC}}
\newcommand{\QuantumProduct}{\ast}

\newcommand{\ConleyZehnderIndex}{\mu}

\newcommand{\ModuliSpace}{\mathcal{M}}
\newcommand{\QuotientedModuliSpace}{\widetilde{\mathcal{M}}}

\newcommand{\WithFilling}[1]{{\widetilde{ #1 }}}

\newcommand{\VerticalTangentSpace}{T^{\text{vert}}}

\newcommand{\Manifold}{M}
\newcommand{\ManifoldElement}{m}
\newcommand{\SymplecticForm}{\omega}
\newcommand{\dimM}{n}

\newcommand{\ContactForm}{\alpha}
\newcommand{\ContactManifold}{\Sigma}
\newcommand{\Reebvf}{X_{\ContactForm}}
\newcommand{\ReebPeriods}{\mathcal{R}}

\newcommand{\AlmostComplexStructure}{J}
\newcommand{\ConvexCoordMap}{\psi}
\newcommand{\StandardACS}{j}
\newcommand{\AntilinearSection}{\overline{\partial}}
\newcommand{\FamilyACS}{\mathbf{\AlmostComplexStructure}}
\newcommand{\ChernBoundedSpheres}{V}
\newcommand{\Ham}{H}
\newcommand{\HamiltonianFlow}{\phi}
\newcommand{\slope}{\lambda}

\newcommand{\NovVariable}{q}
\newcommand{\HamOrbits}{\mathcal{P}}
\newcommand{\FilledHamOrbits}{\widetilde{\mathcal{P}}}
\newcommand{\Orientation}{\mathfrak{o}}
\newcommand{\ContinuationMap}{\varphi}

\newcommand{\CircleAction}{\sigma}
\newcommand{\CircleElement}{\theta}
\newcommand{\CircleHam}{K_\CircleAction}
\newcommand{\CircleHamSpecifyAction}[1]{K_{#1}}
\newcommand{\CircleSlope}{\kappa}
\newcommand{\CircleVectorField}[1]{X_{#1}}
\newcommand{\Lifted}[1]{\widetilde{#1}}

\renewcommand{\Cohomology}{H}
\newcommand{\ECochainComplex}{EC}
\newcommand{\ECohomology}{EH}
\newcommand{\FloerC}{FC}
\newcommand{\EFloerC}{EFC}
\newcommand{\FloerCohomology}{FH}
\newcommand{\EFloerCohomology}{EFH}
\newcommand{\SymplecticCohomology}{SH}
\newcommand{\ESymplecticCohomology}{ESH}
\newcommand{\FloerSeidel}{F\mathcal{S}}
\newcommand{\EFloerSeidel}{EF\mathcal{S}}
\renewcommand{\QuantumCohomology}{{QH}}
\newcommand{\EQuantumCohomology}{EQH}
\renewcommand{\QuantumCochainComplex}{{QC}}
\newcommand{\EQuantumCochainComplex}{{EQC}}
\newcommand{\QuantumSeidelMap}{Q\mathcal{S}}
\newcommand{\EQuantumSeidelMap}{EQ\mathcal{S}}
\newcommand{\PSSmap}{\text{PSS}}
\newcommand{\EPSSmap}{E \text{PSS}}

\newcommand{\MaslovIndex}{I}
\newcommand{\GradedCompletedTensorProduct}{\mathbin{\widehat{\otimes}}}

\newcommand{\sphereinclusion}{i}
\newcommand{\sphereshift}{U}
\newcommand{\spherecriticalpoint}{c}
\newcommand{\sphereprojection}{\pi}
\newcommand{\spheremorsefunction}{F}

\newcommand{\uformal}{\mathbf{u}}
\newcommand{\infsphereelement}{w}

\newcommand{\spheretrivialise}{\tau}

\newcommand{\MorseFunction}{f}
\newcommand{\CriticalPoints}{\operatorname{Crit}}

\newcommand{\EqntHam}{\Ham^\text{eq}}
\newcommand{\EqntACS}{\FamilyACS^\text{eq}}
\newcommand{\EqntMorseFunction}{f^\text{eq}}
\newcommand{\Eqnt}{^\text{eq}}
\newcommand{\eqnt}{\text{eq}}

\newcommand{\ClutchingBundle}{E}
\newcommand{\Hemisphere}{\mathbb{D}}
\newcommand{\ClutchingInclusions}{\iota}
\newcommand{\Pole}{z}
\newcommand{\ClutchingProjection}{\pi}
\newcommand{\ClutchingSymplecticBilinearForm}{\Omega}
\newcommand{\ClutchingTwoForm}{\widehat{\Omega}}
\newcommand{\SphereSymplecticForm}{\omega_{\Sphere}}
\newcommand{\ClutchingHamiltonian}{H^\ClutchingBundle}
\newcommand{\ClutchingBundleACS}{\widehat{\mathbf{J}}}
\newcommand{\ACSSpace}{\mathcal{J}}
\newcommand{\SectionClass}{S}
\newcommand{\LongitudeLine}{L}
\newcommand{\ClutchingAction}{\rho_\ClutchingBundle}

\newcommand{\FubiniStudy}{_{\text{FS}}}
\newcommand{\tautLB}[1]{\mathcal{O}_{\Projective^{#1}}(-1)}
\newcommand{\ZeroSection}{Z}

\renewcommand{\Check}{\vee}

\newcommand{\Connection}{\nabla}
\newcommand{\ShiftOperator}{\mathbb{S}}

\title[An intertwining relation for equivariant Seidel maps]{An intertwining relation for \\ equivariant Seidel maps}
\author{Todd Liebenschutz-Jones}
\date{\today}
\thanks{\textit{Correspondence}: \href{mailto:todd.liebenschutz-jones@maths.ox.ac.uk}{todd.liebenschutz-jones@maths.ox.ac.uk}}
\thanks{\textit{Institution}: University of Oxford, UK}

\begin{document}

\maketitle
\begin{abstract}
    The Seidel maps are two maps associated to a Hamiltonian circle action on a convex symplectic manifold, one on Floer cohomology and one on quantum cohomology.
    We extend their definitions to $S^1$-equivariant Floer cohomology and $S^1$-equivariant quantum cohomology based on a construction of Maulik and Okounkov.
    The $S^1$-action used to construct $S^1$-equivariant Floer cohomology changes after applying the equivariant Seidel map (a similar phenomenon occurs for $S^1$-equivariant quantum cohomology).
    We show the equivariant Seidel map on $S^1$-equivariant quantum cohomology does not commute with the $S^1$-equivariant quantum product, unlike the standard Seidel map.
    We prove an intertwining relation which completely describes the failure of this commutativity as a weighted version of the equivariant Seidel map.
    We will explore how this intertwining relationship may be interpreted using connections in an upcoming paper.
    We compute the equivariant Seidel map for rotation actions on the complex plane and on complex projective space, and for the action which rotates the fibres of the tautological line bundle over projective space.
    Through these examples, we demonstrate how equivariant Seidel maps may be used to compute the $S^1$-equivariant quantum product and $S^1$-equivariant symplectic cohomology.
\end{abstract}


\section{Introduction}

    \label{sec:introduction}
    
    For us, \define{equivariant} will always mean $\Circle$-equivariant.
    

    The Seidel maps on the quantum and Floer cohomology of a closed symplectic manifold $\Manifold$ are two maps associated to a Hamiltonian $\Circle$-action $\CircleAction$ on $\Manifold$ \cite{seidel_$_1997}.
    They are compatible with each other via the PSS isomorphisms, which are maps that identify quantum cohomology and Floer cohomology \cite[Theorem~8.2]{seidel_$_1997}.
    McDuff and Tolman used Seidel maps to recover the Batyrev presentation of the quantum cohomology of closed toric manifolds \cite{mcduff_topological_2006}.
    Ritter extended the definition of Seidel maps to (non-closed) convex symplectic manifolds \cite{ritter_floer_2014}.
    He used the Seidel maps to determine the quantum cohomology and the symplectic cohomology of convex toric manifolds \cite{ritter_circle-actions_2016}.


    Let $\rho$ be a Hamiltonian $\Circle$-action on a closed or convex symplectic manifold $\Manifold$.
    The equivariant quantum cohomology $\EQuantumCohomology_\rho^\ArbitraryIndex(\Manifold)$ has three important compatible algebraic structures:
    it is a ring equipped with the equivariant quantum product $\QuantumProduct_\rho$;
    it is a module over the Novikov ring $\NovikovRing$; and
    it has a geometric $\Integers [\uformal]$-module structure denoted $\CupProduct$.
    In this paper, we introduce an equivariant quantum Seidel map corresponding to an additional Hamiltonian $\Circle$-action $\CircleAction$ on $\Manifold$ which commutes with $\rho$.
    We could, for example, set $\rho = \Identity_\Manifold$, or we could let $\rho$ and $\CircleAction$ be two $\Circle$-actions which are part of a Hamiltonian torus action on $\Manifold$.
    
    \begin{theorem}
        There is an equivariant quantum Seidel map
        \begin{equation}
            \label{eqn:equivariant-quantum-seidel-map-introduction}
            \EQuantumSeidelMap_{\Lifted{\CircleAction}} : \EQuantumCohomology^\ArbitraryIndex_\rho (\Manifold) \to \EQuantumCohomology^{\ArbitraryIndex + 2 \MaslovIndex(\Lifted{\CircleAction})} _{\CircleAction \PullBack \rho} (\Manifold)
        \end{equation}
        which is a $\NovikovRing \Tensor \Integers [\uformal]$-module homorphism.
        The codomain of \eqref{eqn:equivariant-quantum-seidel-map-introduction} is the equivariant quantum cohomology of $\Manifold$ with the \define{pullback action} 
        \begin{equation}
        \label{eqn:pullback-action-introduction}
            \CircleAction \PullBack \rho = \CircleAction \Inverse \rho \EndFullStop
        \end{equation}
    \end{theorem}
    
    Our main theorem describes the relationship between $\EQuantumSeidelMap_{\Lifted{\CircleAction}}$ and equivariant quantum multiplication.
    Unlike for the (non-equivariant) quantum Seidel map, this relationship is not simply that the map commutes with the quantum product.    
    We can nonetheless exploit the relationship to derive the equivariant quantum product, which we demonstrate for a few examples in Sections \ref{sec:example-computations-introduction-summary} and \ref{sec:example-computations}.

    \begin{theorem}
        [Intertwining relation]
        \label{thm:interwining-relation-introduction}
        The equation
        \begin{equation}
            \label{eqn:quantum-intertwining-seidel-equivariant-introduction}
            \EQuantumSeidelMap_{\Lifted{\CircleAction}} (x \underset{\rho}{\QuantumProduct} \alpha^+)
            -
            \EQuantumSeidelMap_{\Lifted{\CircleAction}} (x) \underset{\CircleAction \PullBack \rho}{\QuantumProduct} \alpha^-
            =
            \uformal \CupProduct \EQuantumSeidelMap_{\Lifted{\CircleAction}, \alpha}(x)
        \end{equation}
        holds for all $x \in \EQuantumCohomology^\ArbitraryIndex_\rho(\Manifold)$.
        Here, $\alpha^+ \in \ECohomology^\ArbitraryIndex_\rho(\Manifold)$ and $\alpha^- \in \ECohomology^\ArbitraryIndex_{\CircleAction \PullBack \rho} (\Manifold)$ are two equivariant cohomology classes which are related via the clutching bundle, and 
        \begin{equation}
            \label{eqn:error-term-description-introduction}
            \EQuantumSeidelMap_{\Lifted{\CircleAction}, \alpha} : \EQuantumCohomology^\ArbitraryIndex_\rho (\Manifold) \to \EQuantumCohomology^{\ArbitraryIndex + 2 \MaslovIndex(\Lifted{\CircleAction}) + \Modulus{\alpha^\pm} - 2} _{\CircleAction \PullBack \rho} (\Manifold)
        \end{equation}
        is a map defined in \autoref{sec:error-term-definition}.
    \end{theorem}


    Maulik and Okounkov gave the first definition of equivariant quantum Seidel maps as part of their work on quiver varieties \cite[Section~8]{maulik_quantum_2012}.
    They also proved that the maps satisfied the intertwining relation \cite[Proposition~8.1]{maulik_quantum_2012}.
    Braverman, Maulik and Okounkov used equivariant quantum Seidel maps to derive the equivariant quantum product of the Springer resolution \cite{braverman_quantum_2011}.
    Iritani recovered Givental's mirror theorem using a new version of the equivariant quantum Seidel map which applied to big equivariant quantum cohomology \cite{iritani_shift_2017}.
    Givental's mirror theorem describes the equivariant genus-zero Gromov-Witten invariants of a toric variety.
    
    Maulik and Okounkov's work on the equivariant quantum Seidel map applies to smooth quasi-projective varieties $X$ with a holomorphic symplectic structure which are equipped with the action of a reductive group $G$ (such $X$ are (real) symplectic manifolds with $\FirstChernClass = 0$).
    Iritani's work applies to smooth toric varieties $X$ that admit a projective morphism to an affine variety and for which the action of the algebraic torus $G$ on $X$ satisfies some positivity condition.
    The definitions of the equivariant quantum Seidel map in this algebrogeometric context use \emph{equivariant virtual fundamental classes} to count stable maps.
    The proofs of the intertwining relation in this context make heavy use of an algebrogeometric technique called \emph{virtual localisation}, which reduces counting stable maps to counting only the $G$-fixed stable maps.
    
    In contrast, our results apply to closed or convex symplectic manifolds $\Manifold$ which satisfy a monotonicity condition.
    We use a Borel model for equivariant quantum cohomology, which combines the Morse cohomology of our manifold $\Manifold$ with the Morse cohomology of the classifying space of $\Circle$.
    This model is preferable in our context because it readily extends to equivariant Floer theory.
    In our Borel model, we perturb the data on $\Manifold$ using the classifying space of $\Circle$ to ensure that the moduli spaces are smooth manifolds.
    We therefore avoid using virtual fundamental classes to count stable maps, and instead just count the 0-dimensional moduli spaces.
    The $G$-fixed stable maps used by virtual localisation are not pseudoholomorphic curves for the perturbed data, however, so virtual localisation will not work in our context.
    To remedy this, we provide a new proof of the intertwining relation \autoref{thm:interwining-relation-introduction} which has a Morse-theoretic flavour: we construct an explicit 1-dimensional moduli space whose boundary gives the relation.


    We also introduce an equivariant Floer Seidel map, which is an isomorphism on equivariant Floer cohomology.
    We use a Borel model for equivariant Floer cohomology, which combines Morse theory on the classifying space of $\Circle$ with Floer theory on $\Manifold$.
    The $\Circle$-action $\rho : \Circle \times \Manifold \to \Manifold$ induces an $\Circle$-action on Hamiltonian orbits $x : \Circle \to \Manifold$, given by
    \begin{equation}
    \label{eqn:circle-action-induced-on-orbits}
        (\CircleElement \cdot x) (t) = \rho_\CircleElement ( x (t - \CircleElement) ),
    \end{equation}
    and it is with respect to this action that we define equivariant Floer cohomology.
    The equivariant Floer Seidel map is the identity map on the classifying space of $\Circle$ and it maps the Hamiltonian orbit $x : \Circle \to \Manifold$ to the Hamiltonian orbit $(\CircleAction \PullBack x)(t) = \CircleAction \Inverse _t (x(t))$.
    Much like the equivariant quantum Seidel map \eqref{eqn:equivariant-quantum-seidel-map-introduction}, the codomain of the equivariant Floer Seidel map is the equivariant Floer cohomology not for the action \eqref{eqn:circle-action-induced-on-orbits}, but instead for the corresponding action induced by the pullback action $\CircleAction \PullBack \rho$.
    
    The equivariant Floer Seidel map commutes with continuation maps, which means it induces a map on equivariant symplectic cohomology.
    Recall that the equivariant symplectic cohomology of $\Manifold$ is the direct limit of equivariant Floer cohomology as the slope of the Hamiltonian function increases (see \autoref{sec:equivariant-floer-cohomology-all}).

    \begin{theorem}
        There is an equivariant Floer Seidel map on equivariant symplectic cohomology
        \begin{equation}
            \EFloerSeidel_{\Lifted{\CircleAction}} : \ESymplecticCohomology^\ArbitraryIndex_\rho(\Manifold) \to \ESymplecticCohomology^{\ArbitraryIndex + 2 \MaslovIndex(\Lifted{\CircleAction})}_{\CircleAction \PullBack \rho} (\Manifold)
        \end{equation}
        which is a $\NovikovRing \GradedCompletedTensorProduct \Integers [\uformal]$-module isomorphism.
        The diagram
        \begin{equation}
            \label{eqn:commutative-diagram-equivariant-quantum-seidel-map-into-equivariant-symplectic-cohomology}
            \begin{tikzcd}[column sep=large]
                \EQuantumCohomology^\ArbitraryIndex_\rho (\Manifold) \arrow[r, "\EQuantumSeidelMap_{\Lifted{\CircleAction}}"] \arrow[d]
                & \EQuantumCohomology^{\ArbitraryIndex + 2 \MaslovIndex(\Lifted{\CircleAction})} _{\CircleAction \PullBack \rho} (\Manifold) \arrow[d]
                \\ \ESymplecticCohomology^\ArbitraryIndex_\rho(\Manifold) \arrow[r, "\EFloerSeidel_{\Lifted{\CircleAction}}"]
                & \ESymplecticCohomology^{\ArbitraryIndex + 2 \MaslovIndex(\Lifted{\CircleAction})}_{\CircleAction \PullBack \rho} (\Manifold) 
            \end{tikzcd}
        \end{equation}
        commutes, where the vertical arrows denote equivariant $c^\ast$ maps.
    \end{theorem}


    Equivariant symplectic cohomology has attracted attention in recent years because, while similar to (non-equivariant) symplectic cohomology, it possesses a number of different properties.
    Chiefly, it distinguishes Hamiltonian orbits with different stabilizer groups more readily than does its non-equivariant counterpart.
    This is useful because these different orbits have different geometric significance (for certain choices of Hamiltonians).
    For example, the constant orbits (with stabilizer group $\Circle$) capture the topology of $\Manifold$.
    Zhao restricted to the constant orbits by localising the ring $\Integers [\uformal]$, and obtained the isomorphism
    \begin{equation}
        \label{eqn:zhao-localisation-theorem}
        \RationalNumbers[\uformal, \uformal\Inverse] \Tensor_{\Integers [\uformal]} \ESymplecticCohomology^\ArbitraryIndex(\Manifold) \cong
        \RationalNumbers[\uformal, \uformal\Inverse] \Tensor_{\Integers} \Cohomology^\ArbitraryIndex(\Manifold)
    \end{equation}
    for completions of Liouville domains \cite[Theorem~1.1]{zhao_periodic_2014}.
    This is unlike (non-equivariant) symplectic cohomology, which vanishes for subcritical Stein manifolds (the vanishing follows from \cite[Theorem~1.1, part~1]{cieliebak_handle_2002} and $\SymplecticCohomology^\ArbitraryIndex(\ComplexNumbers^N) = 0$ \cite[Section~3]{oancea_survey_2004}).
    
    On the other hand, the nonconstant orbits (with finite stablizer groups) capture the Reeb dynamics.
    We can restrict to the nonconstant orbits by looking at the positive part of $\ESymplecticCohomology^\ArbitraryIndex(\Manifold)$.
    Bourgeois and Oancea showed that the positive part of equivariant symplectic homology is isomorphic to linearized contact homology \cite{bourgeois_s1-equivariant_2017} while Gutt used it to distinguish nonisomorphic contact structures on spheres \cite[Theorem~1.4]{gutt_positive_2017}.
    In fact, McLean and Ritter used $\ESymplecticCohomology^\ArbitraryIndex(\Manifold)$ to classify orbits with different finite stablizer groups in their new proof of the McKay correspondence \cite{mclean_mckay_2018}.
    
    One downside of equivariant symplectic cohomology is that it lacks an interesting algebraic structure.
    (Non-equivariant) symplectic cohomology has the pair-of-pants product, which equips it with a graded-commutative, associative and unital ring structure.
    This ring structure can in fact be upgraded to an entire TQFT structure on symplectic cohomology \cite[Section~(8a)]{ritter_topological_2013, seidel_biased_2007}.
    In contrast, equivariant symplectic cohomology only has the $\Integers [\uformal]$-module structure, which arises as the cohomology of the classifying space of $\Circle$.

    
    Seidel described a new algebraic structure on equivariant Floer cohomology, which he called the \define{$q$-connection} \cite{seidel_connections_2016}.
    This structure enhances the module structure, but is less rich than a full ring structure.
    Seidel's $q$-connection is based upon the \define{quantum connection}, which is a map $\Connection_\alpha :  \EQuantumCohomology^\ArbitraryIndex (\Manifold) \to \EQuantumCohomology^{\ArbitraryIndex + 2} (\Manifold)$ that combines multiplication by $\alpha$ with a differentiation-like operation applied to the Novikov variable.
    Throughout the algebrogeometric literature on equivariant quantum Seidel maps, the intertwining relation is often written as
        \begin{equation}
        \label{eqn:flatness-for-quantum-introduction}
            \Connection_\alpha \ComposedWith \ShiftOperator_\CircleAction = \ShiftOperator_\CircleAction \ComposedWith \Connection_\alpha,
        \end{equation}
    where $\ShiftOperator_\CircleAction$ is a variant of the equivariant quantum Seidel map \cite{maulik_quantum_2012, braverman_quantum_2011, iritani_shift_2017}.
    In our upcoming paper \cite{liebenschutz-jones_flatness_in-progress}, we will show that a variant of our equivariant Floer Seidel map commutes with the $q$-connection, just like \eqref{eqn:flatness-for-quantum-introduction}.
    This represents a further enhancement of the available algebraic structures on equivariant symplectic cohomology.
    The result is proved with a Floer theory version of our new proof of the intertwining relation \autoref{thm:interwining-relation-introduction}.

    \subsection{Outline}
    
        We give an overview of the background material for our work in the next section (\autoref{sec:overview}), as well as giving more information about and intuition for our constructions and results.
        \autoref{sec:example-computations-introduction-summary} contains an overview of our example calculations.
    
        In \autoref{sec:floer-theory}, we introduce the Floer Seidel maps in detail, clarifying our assumptions and conventions for symplectic cohomology.
        We introduce equivariant Floer theory and its associated module structures in \autoref{sec:equivariant-floer-theory} before defining the equivariant Floer Seidel map in \autoref{sec:equivariant-floer-seidel-map}.
        
        In \autoref{sec:quantum-theory}, we introduce the quantum Seidel map (\autoref{sec:quantum-seidel-map}) and equivariant quantum cohomology (\autoref{sec:equivariant-quantum-cohomology}).
        We define the equivariant quantum Seidel map in \autoref{sec:equivariant-quantum-seidel-map} and prove the intertwining relation in \autoref{sec:intertwining-relation-equivariant-quantum-seidel-map}.
        
        \autoref{sec:example-computations} contains the details of our three example calculations.
    
    \subsection{Acknowledgements}
        
        The author would like to thank his supervisor Alexander Ritter for his guidance, support and ideas.
        The author wishes to thank Paul Seidel for his ideas.
        The author thanks Nicholas Wilkins and Filip Živanović for many constructive discussions and their feedback.
        The author thanks Jack Smith for useful conversations.
        The author gratefully acknowledges support from the EPSRC grant EP/N509711/1.
        This work will form part of his DPhil thesis.
    
\section{Overview}

    \label{sec:overview}
    
    \subsection{Background}
    
        \subsubsection{Seidel maps on closed symplectic manifolds}
        
        \label{sec:seidel-maps-on-closed-manifolds}
    
        Let $\Manifold$ be a closed monotone symplectic manifold and let $\CircleAction$ be a Hamiltonian circle action on $\Manifold$.
        In \cite{seidel_$_1997}, Seidel defined a pair of automorphisms associated to $\CircleAction$, one on Floer cohomology and one on quantum cohomology.
        To distinguish between these maps, we call the former the \emph{Floer Seidel map} and the latter the \emph{quantum Seidel map}.
        
        \begin{definition}
            [Floer Seidel map]
            Recall that the Floer cochain complex $\FloerC^\ArbitraryIndex(\Manifold; \Ham)$ associated to the time-dependent Hamiltonian function $\Ham : \Circle \times \Manifold \to \RealNumbers$ is freely-generated by the 1-periodic Hamiltonian orbits of $\Ham$ over the Novikov ring $\NovikovRing$.
            Given a Hamiltonian orbit $x : \Circle \to \Manifold$ of $\Ham$, the pullback orbit $\CircleAction \PullBack x$ given by
            \begin{equation}
                \label{eqn:pullback-hamiltonian-orbit}
                (\CircleAction \PullBack x) (t) = \CircleAction\Inverse_t (x(t))
            \end{equation}
            is a Hamiltonian orbit of the pullback Hamiltonian $\CircleAction \PullBack \Ham$.
            Moreover, the assignment $x \mapsto \CircleAction \PullBack x$ is a bijection between the orbits of $\Ham$ and the orbits of $\CircleAction \PullBack \Ham$.
            This assignment can be upgraded to an isomorphism of Floer cochain complexes
            \begin{equation}
                \label{eqn:floer-seidel-map-on-floer-cochain-complex}
                \FloerSeidel_{\Lifted{\CircleAction}} : \FloerC^\ArbitraryIndex(\Manifold; \Ham) \to \FloerC^{\ArbitraryIndex + 2 \MaslovIndex(\Lifted{\CircleAction})} (\Manifold; \CircleAction \PullBack \Ham)
            \end{equation}
            by using a \emph{lift $\Lifted{\CircleAction}$} of the circle action $\CircleAction$ to keep track of the information recorded by the Novikov ring.
            The quantity $\MaslovIndex(\Lifted{\CircleAction})$ is a Maslov index associated to $\Lifted{\CircleAction}$ (see \autoref{sec:hamiltonian-circle-actions} for details).
            
            The \define{Floer Seidel map $\FloerSeidel_{\Lifted{\CircleAction}}$} is the map induced on Floer cohomology by \eqref{eqn:floer-seidel-map-on-floer-cochain-complex}.
            It satisfies $\FloerSeidel_{\Lifted{\CircleAction}}(a \ast b) = a \ast \FloerSeidel_{\Lifted{\CircleAction}}(b)$, where $\ast$ denotes the pair-of-pants product.
        \end{definition}
        
        \begin{definition}
            [Quantum Seidel map]
            Define a \define{clutching bundle $\ClutchingBundle$} over $\Sphere$ with fibre $\Manifold$ as follows.
            The sphere is the union of its northern hemisphere $\Hemisphere^-$ and its southern hemisphere $\Hemisphere^+$.
            Each hemisphere is isomorphic to a closed unit disc, and the two hemispheres are glued along the equator $\Circle = \Boundary \Hemisphere^- = \Boundary \Hemisphere^+$ to get the sphere.
            The clutching bundle is the union of the trivial bundles $\Hemisphere^- \times \Manifold$ and $\Hemisphere^+ \times \Manifold$, glued along the equator by the relation
            \begin{equation}
                \Boundary \Hemisphere^- \times \Manifold \ni (t, m) \leftrightarrow (t, \CircleAction_t(m)) \in \Boundary \Hemisphere^+ \times \Manifold \EndFullStop
            \end{equation}
            Thus we twist the bundle by $\CircleAction$ when passing from the northern hemisphere to the southern hemisphere.
            
            The \define{quantum Seidel map} is an isomorphism $\QuantumSeidelMap_{\Lifted{\CircleAction}} : \QuantumCohomology^\ArbitraryIndex(\Manifold) \to \QuantumCohomology^{\ArbitraryIndex + 2 \MaslovIndex(\Lifted{\CircleAction})} (\Manifold)$ which counts pseudoholomorphic sections of this clutching bundle.
            It satisfies
            \begin{equation}
                \label{eqn:non-equivariant-quantum-intertwining-relation}
                \QuantumSeidelMap_{\Lifted{\CircleAction}}(a \ast b) = a \ast \QuantumSeidelMap_{\Lifted{\CircleAction}}(b),
            \end{equation}
            where $\ast$ denotes the quantum product.
            It follows that the quantum Seidel map is given by quantum product multiplication by the invertible element $\QuantumSeidelMap_{\Lifted{\CircleAction}}(1)$.
            This element is the \define{Seidel element of $\Lifted{\CircleAction}$}.
        \end{definition}
        
        Using the PSS isomorphism to identify $\QuantumCohomology^\ArbitraryIndex(M)$ with Floer cohomology, the Floer Seidel map and the quantum Seidel map are identified.
        Roughly, each hemisphere in the clutching bundle corresponds to a PSS map, and the twisting along the equator corresponds to the Floer Seidel map.
        
        More generally, Seidel maps may be defined for loops in the Hamiltonian symplectomorphism group $\HamiltonianSymplectomorphismGroup(\Manifold)$ based at the identity $\Identity_\Manifold$.
        (Technically, we must use a cover $\widetilde{\HamiltonianSymplectomorphismGroup}(\Manifold)$, though we omit details here.)
        The Seidel maps are homotopy invariants, so that the assignment
        \begin{equation}
            \label{eqn:seidel-representation}
            \begin{aligned}
                \FundamentalGroup( \widetilde{\HamiltonianSymplectomorphismGroup}(\Manifold), \Identity_\Manifold) &\to \QuantumCohomology^\ArbitraryIndex(\Manifold) \Units \\
                \Lifted{\CircleAction} \mapsto \QuantumSeidelMap_{\Lifted{\CircleAction}}(1)
            \end{aligned}
        \end{equation}
        is a group homomorphism.
        The map \eqref{eqn:seidel-representation} is the \define{Seidel representation}.
        
        Here is a selection of results whose proofs use Seidel maps. It is by no means complete.
        
        \begin{itemize}
            \item
                An algorithmic and combinatorial computation of quantum cohomology from the moment polytope of a toric symplectic manifold \cite[Section~5]{mcduff_topological_2006}.
            \item
                An isomorphism $\QuantumCohomology^\ArbitraryIndex (\Manifold) \cong \QuantumCohomology^\ArbitraryIndex((\Projective^1)^n)$ whenever $\Manifold$ admits a semifree circle action with nonempty isolated fixed point set \cite{gonzalez_quantum_2006}.
            \item
                Computation of the Gromov width and Hofer-Zehnder capacity in terms of the values of the Hamiltonian of $\CircleAction$ on fixed components, under the assumption that $\CircleAction$ is semi-free (and its maximal fixed locus is a point) \cite{hwang_symplectic_2017}.
        \end{itemize}
        
        \subsubsection{Convex manifolds}
        
            \label{intro:convex-manifolds}
        
            The symplectic manifold $(\Manifold, \SymplecticForm)$ is \define{convex} (sometimes \define{conical} in the literature) if it is symplectomorphic to $\ContactManifold \times \IntervalClosedOpen{1}{\Infinity}$ with symplectic form $\SymplecticForm = \ExteriorDerivative( R \ContactForm )$ away from a relatively compact subset $\Manifold_0 \subseteq \Manifold$.
            Here, $\ContactManifold$ is a closed contact manifold with contact form $\ContactForm$ and $R$ is the coordinate of $\IntervalClosedOpen{1}{\Infinity}$.
            
            The cotangent bundle of a closed manifold is convex, and more generally so is the completion of a Liouville domain.
            In addition, there are many examples of convex symplectic manifolds whose symplectic forms are not globally exact, such as the total space of the line bundle $\StandardLineBundle_{\Projective^N}(-1)$.
            
            In the region $\ContactManifold \times \IntervalClosedOpen{1}{\Infinity}$, the \define{convex end}, the symplectic form is exact and it grows infinitely big as $R \to \Infinity$.
            Moreover the formula $\SymplecticForm = \ExteriorDerivative( R \ContactForm ) = \ExteriorDerivative R \wedge \ContactForm + R \ExteriorDerivative \ContactForm$ decomposes the tangent space into the two-dimensional subspace $\RealNumbers \partial_R + \RealNumbers \Reebvf$ and the contact structure $\Kernel \alpha$.
            (Here, $\Reebvf$ is the Reeb vector field, which is characterised by $\ContactForm(\Reebvf) = 1$ and $\ExteriorDerivative \ContactForm (\Reebvf, \Argument) = 0$.)
            These properties of $\SymplecticForm$ constrain the Hamiltonian dynamics of a Hamiltonian function $\Ham$ to a compact region of the form $\Manifold_0 \Union (\ContactManifold \times \IntervalClosed{1}{R_0})$ so long as $\Ham = \slope R + \mu$ is a linear function of $R$ outside this region (some further conditions are also required, see \autoref{sec:floer-cohomology-all}).
            
            Once the Hamiltonian dynamics are restricted to a compact region, the machinery of Floer cohomology applies, so we can define Floer cohomology for convex symplectic manifolds (which satisfy an appropriate monotonicity assumption) and Hamiltonian functions which are linear outside a compact region.
            Floer cohomology depends on the slope $\slope$ of the Hamiltonian, unlike for closed manifolds whose Floer cohomology is entirely independent of the Hamiltonian. 
            \define{Symplectic cohomology $\SymplecticCohomology^\ArbitraryIndex(\Manifold)$} is the direct limit of Floer cohomology as the slope $\slope$ tends to infinity.
            The PSS maps are isomorphisms between the quantum cohomology $\QuantumCohomology^\ArbitraryIndex(\Manifold)$ and the Floer cohomology of a Hamiltonian with sufficiently small (positive) slope.
        
        \subsubsection{Seidel maps on convex symplectic manifolds}
        
            \label{sec:seidel-maps-on-convex-manifolds}
            
            In \cite{ritter_floer_2014}, Ritter extended the construction of Seidel maps to Hamiltonian circle actions $\CircleAction$ on convex symplectic manifolds, so long as the Hamiltonian $\CircleHam$ of $\CircleAction$ is linear, with nonnegative slope, outside a compact region.
            The Floer Seidel map is an isomorphism between $\FloerCohomology^\ArbitraryIndex(\Manifold; \Ham)$ and $\FloerCohomology^{\ArbitraryIndex + 2 \MaslovIndex(\Lifted{\CircleAction})}(\Manifold; \CircleAction \PullBack \Ham)$ as for closed manifolds, however the pullback Hamiltonian $\CircleAction \PullBack \Ham$ may have a different slope to $\Ham$.
            The limit of the Floer Seidel maps as the slope $\slope$ of $\Ham$ tends to infinity induces an automorphism $\FloerSeidel_{\Lifted{\CircleAction}}$ of symplectic cohomology $\SymplecticCohomology^\ArbitraryIndex(\Manifold)$.
            
            The quantum Seidel map on a convex symplectic manifold is not necessarily an isomorphism, but instead is merely a module homomorphism.
            Ritter computed the quantum Seidel maps on the total spaces of the line bundles $\StandardLineBundle_{\Projective^N}(-1)$.
            His calculations demonstrate that, even in elementary examples, the quantum Seidel map may fail to be injective and surjective.
            
            The fact that the Floer Seidel map is an isomorphism even though the quantum Seidel map may fail to be injective and surjective is explained by the following commutative diagram.
            The non-bijectivity of the quantum Seidel map corresponds to the non-bijectivity of the continuation map.
            \begin{equation}
                \label{eqn:diagram-relating-quantum-and-floer-seidel-maps-introduction}
                \begin{tikzcd}[column sep=0, row sep=large]
                    \QuantumCohomology^\ArbitraryIndex (\Manifold)
                        \arrow[rr, "\QuantumSeidelMap_{\Lifted{\CircleAction}}"]
                        \arrow[d, "\text{PSS map}"', "\cong"] 
                        &
                        & \QuantumCohomology^{\ArbitraryIndex + 2\MaslovIndex(\Lifted{\CircleAction})} (\Manifold)
                    \\
                    \FloerCohomology^\ArbitraryIndex(\Manifold; \Ham^\text{small})
                        \arrow[rd, "\FloerSeidel_{\Lifted{\CircleAction}}"', "\cong"]
                        &
                        & \FloerCohomology^{\ArbitraryIndex + 2\MaslovIndex(\Lifted{\CircleAction})} (\Manifold; \Ham^\text{small})
                        \arrow[u, "\text{PSS map}"', "\cong"]
                    \\
                    & \FloerCohomology^{\ArbitraryIndex + 2\MaslovIndex(\Lifted{\CircleAction})} (\Manifold; \CircleAction\PullBack\Ham^\text{small})
                        \arrow[ru, "\text{continuation map}"']
                        &                  
                \end{tikzcd}
            \end{equation}
            The Hamiltonian $\Ham^\text{small}$ has small (positive) slope, but the Hamiltonian $\CircleAction \PullBack \Ham^\text{small}$ typically has negative slope.
            
            When the slope of $\CircleHam$ is positive, the quantum Seidel map may be used to compute symplectic cohomology \cite[Theorem~22]{ritter_floer_2014}.
            This approach explicitly computes symplectic cohomology for various line bundles of the form $\StandardLineBundle_{\Projective^N}(-k)$ \cite[Theorem~5]{ritter_floer_2014} and provides an algorithm to compute the symplectic cohomology of a Fano toric negative line bundle using its moment polytope \cite[Theorem~1.5]{ritter_circle-actions_2016}.
        
        \subsubsection{Equivariant Floer cohomology}
        
            Recall that Floer cohomology is inspired by the Morse cohomology of the loop space $\ContractibleLoopSpace \Manifold = \Set{\text{contractible $x : \Circle \to \Manifold$}}$.
            The loop space $\ContractibleLoopSpace \Manifold$ has a canonical circle action which rotates the loops.
            Equivariant Floer cohomology is analogously inspired by the equivariant Morse cohomology of $\ContractibleLoopSpace \Manifold$ with this rotation action.
            
            Viterbo introduced the first version of equivariant Floer cohomology \cite[Section~5]{viterbo_functors_1996}, and later Seidel introduced a second version \cite[Section~8b]{seidel_biased_2007}.
            Bourgeois and Oancea showed these different versions are equivalent \cite[Proposition~2.5]{bourgeois_s1-equivariant_2017}.
            We use Seidel's approach, since its analysis is far simpler.
            For more information on the history, see \cite[Section~2]{bourgeois_s1-equivariant_2017}.
            Our conventions differ from those in the literature in three important ways: we use a different relation for the Borel homotopy quotient (\autoref{rem:diagonal-action-in-literature-for-equivariant-floer-theory}), we use a geometric module structure (\autoref{sec:y-shape-module-structure-on-equivariant-floer}) and we incorporate actions on the manifold $\Manifold$.
            
            Denote by $\InfiniteSphere$ the limit of the inclusions $\Circle \Inclusion \HighDimensionalSphere{3} \Inclusion \HighDimensionalSphere{5} \cdots$, where $\HighDimensionalSphere{2k-1}$ is thought of as the unit sphere in $\ComplexNumbers^{k}$.
            It is a contractible space with a free circle action.
            Given any space $X$ with a circle action, its \define{Borel homotopy quotient}, denoted $\InfiniteSphere \times_\Circle X$, is the quotient of $\InfiniteSphere \times X$ by the relation $(\infsphereelement, \CircleElement \cdot x) \sim (\CircleElement \cdot \infsphereelement, x)$ for all $\CircleElement \in \Circle$.
            \define{Equivariant cohomology} is the cohomology of $\InfiniteSphere \times_\Circle X$.
            
            Informally, the equivariant Morse cohomology of the Borel homotopy quotient may be obtained by doing Morse theory on $\InfiniteSphere$ and on $X$, and quotienting the moduli spaces by the induced relation.
            Similarly, the equivariant Floer cohomology of $\Manifold$ is obtained by doing Morse theory on $\InfiniteSphere$ and Floer theory on $\Manifold$, and quotienting the moduli spaces by the induced relation.
            
            The \define{equivariant Floer cochain complex} is generated by $\sim$-equivalence classes $[(\infsphereelement, x)]$, where $\infsphereelement$ is a critical point in $\InfiniteSphere$ and $x$ is a 1-periodic Hamiltonian orbit of a time-dependent Hamiltonian $\EqntHam_{\infsphereelement}$.
            The function $\EqntHam : \InfiniteSphere \times \Circle \times \Manifold \to \RealNumbers$ must satisfy the identity 
            \begin{equation}
                \label{eqn:equivariant-condition-on-hamiltonian-no-action}
                \EqntHam_{\CircleElement \cdot \infsphereelement, t} (\ManifoldElement) = \EqntHam_{\infsphereelement, t + \CircleElement} (\ManifoldElement )
            \end{equation}        
            in order that the relation $\sim$ make sense on the pairs $(\infsphereelement, x) \in \InfiniteSphere \times \ContractibleLoopSpace{\Manifold}$.
            Such $\sim$-equivalence classes are \define{equivariant Hamiltonian orbits}.
            
            Similarly, consider $\sim$-equivalence classes $[(v, u)]$ where $v : \RealNumbers \to \InfiniteSphere$ is a Morse trajectory and $u : \RealNumbers \times \Circle \to \Manifold$ is a Floer solution of the $s$-dependent Hamiltonian $\EqntHam_{v(s)}$.
            The differential on the equivariant Floer cochain complex counts these equivalence classes modulo the free $\RealNumbers$-action of translation.
            
            The cohomology of $\InfiniteSphere / \Circle$ is $\Integers [\uformal]$, where $\uformal$ is a formal variable in degree 2.
            For certain choices of Floer data, the equivariant Floer cochain complex\footnote{
                We may need infinitely-many powers of $\uformal$ so we must use some kind of completed tensor product.
                Some authors use $\Integers \llbracket \uformal \rrbracket$, however we use a slightly smaller cochain complex \eqref{eqn:equivariant-floer-complex}.
                With our approach, $\EFloerCohomology^\ArbitraryIndex(\Manifold)$ is graded in the conventional sense.
            } is $\FloerC^\ArbitraryIndex(\Manifold) \GradedCompletedTensorProduct \Integers [\uformal]$ with differential $d\Eqnt = d + \SmallO(\uformal)$, where $d$ is the (non-equivariant) Floer differential.
            
            The resulting \define{equivariant Floer cohomology $\EFloerCohomology^\ArbitraryIndex(\Manifold; \EqntHam)$} is a graded module over the Novikov ring $\NovikovRing$.
            Moreover, it has a $\Integers [\uformal]$-module structure coming from the description of the cochain complex above.
            In fact, it has another $\Integers [\uformal]$-module structure given by an equivariant cup product type construction, which we call the \define{geometric module structure} and denote $\CupProduct$.
            
            Much like in the non-equivariant setup in \autoref{intro:convex-manifolds}, $\EFloerCohomology^\ArbitraryIndex(\Manifold; \EqntHam)$ only depends on the slope of $\EqntHam$, and the limit of $\EFloerCohomology^\ArbitraryIndex(\Manifold; \EqntHam)$ as the slopes increase to infinity is the \define{equivariant symplectic cohomology $\ESymplecticCohomology^\ArbitraryIndex(\Manifold)$}.
            Equivariant PSS maps identify $\EFloerCohomology^\ArbitraryIndex(\Manifold; \EqntHam)$ and $\EQuantumCohomology^\ArbitraryIndex(\Manifold)$ when $\EqntHam$ has sufficiently small (positive) slope.
            Here, $\EQuantumCohomology^\ArbitraryIndex(\Manifold)$ is the equivariant quantum cohomology of $\Manifold$ for the trivial identity circle action on $\Manifold$.
            
            For us, we have a Hamiltonian action $\CircleAction$ acting on $\Manifold$.
            This means that the loop space $\ContractibleLoopSpace(\Manifold)$ has another canonical circle action given by 
            \begin{equation}
                \label{eqn:circle-action-on-loop-space-introduction}
                \CircleElement \cdot (\ t \mapsto x(t)\ ) = (\ t \mapsto \CircleAction_\CircleElement (x(t - \CircleElement))\ ) \EndFullStop
            \end{equation}
            This action combines the rotation action on the domain of the loops with the action $\CircleAction$ on the target space $\Manifold$.
            All of the equivariant constructions above generalise to the new action, and we denote these versions with a subscript $\CircleAction$.
            For example, $\ESymplecticCohomology_\CircleAction^\ArbitraryIndex(\Manifold)$ is the equivariant symplectic cohomology corresponding to \eqref{eqn:circle-action-on-loop-space-introduction}.
    
    \subsection{Equivariant Seidel maps}
    
        In this paper, we define new variants of the Floer Seidel map and the quantum Seidel map on equivariant Floer cohomology and equivariant quantum cohomology respectively.
        We prove a number of initial properties of these maps, which for the most part are just like the non-equivariant maps.
        The main exception is that the equivariant quantum Seidel map and the equivariant quantum product do not commute (\autoref{thm:interwining-relation-introduction}).
        
        The extension of the Floer Seidel map to the equivariant setting is nontrivial because the equivariant Floer Seidel map pulls back the action \eqref{eqn:circle-action-on-loop-space-introduction}.
        A similar phenomenon occurs with the equivariant quantum Seidel map.
        No new analysis is required, however, since our constructions only use an $\InfiniteSphere$-parameterised version of the analysis used to define the Seidel maps in \cite{ritter_floer_2014}.
        
        \subsubsection{Definitions}
        
            Let $\Manifold$ be a convex\footnote{
                Our construction works for closed manifolds as well, but we only discuss the convex case here to simplify the discussion.
                We treat both cases in the rest of the paper (see \autoref{rem:closed-manifold-is-trivially-convex}).
            } symplectic manifold which is either monotone or whose first Chern class vanishes.
            Let $\CircleAction$ and $\rho$ be two commuting Hamiltonian circle actions on $\Manifold$ whose Hamiltonian functions are linear outside a compact subset of $\Manifold$.
            Assume the Hamiltonian $\CircleHam$ of $\CircleAction$ has nonnegative slope.
            
            Recall the cochain complex for equivariant Floer cohomology $\EFloerCohomology^\ArbitraryIndex_\rho (\Manifold; \EqntHam)$ is generated by equivariant Hamiltonian orbits, which are certain equivalence classes of pairs $(w, x) \in \InfiniteSphere \times \ContractibleLoopSpace \Manifold$ under the equivalence relation $(\CircleElement w, x(t)) \sim_\rho (w, \rho_\CircleElement(x(t - \CircleElement)))$.
            
            In order for the map $[(w, x)] \mapsto [(w, \CircleAction \PullBack x)]$ to be well-defined, the equivalence class $[(w, \CircleAction \PullBack x)]$ must be considered with the relation $\sim_{\CircleAction \PullBack \rho}$, which corresponds to the pullback circle action $\CircleAction \PullBack \rho = \CircleAction \Inverse \rho$.
            Once we account for the change in the action, the definition of the Floer Seidel map extends naturally.
        
            \begin{definition}
                [Equivariant Floer Seidel map]
                The \define{equivariant Floer Seidel map} is the map 
                \begin{equation}
                    \EFloerSeidel_{\Lifted{\CircleAction}} : \EFloerCohomology^\ArbitraryIndex_\rho (\Manifold; \EqntHam) \to \EFloerCohomology^{\ArbitraryIndex + 2 \MaslovIndex(\Lifted{\CircleAction})}_{\CircleAction \PullBack \rho} (\Manifold; \CircleAction \PullBack \EqntHam)
                \end{equation}
                given by $[(w, x)] \mapsto [(w, \CircleAction \PullBack x)]$ on equivariant Hamiltonian orbits.
                It is a $\NovikovRing$-module isomorphism on the cochain complex which is compatible with algebraic $\Integers [\uformal]$-module structure.
                On cohomology, it is compatible with the geometric $\Integers [\uformal]$-module structure and with continuation maps.
                Under the limit as the slope of the equivariant Hamiltonians tends to infinity, the maps induce a well-defined isomorphism
                \begin{equation}
                    \EFloerSeidel_{\Lifted{\CircleAction}} : \ESymplecticCohomology^\ArbitraryIndex_\rho(\Manifold) \to \ESymplecticCohomology^{\ArbitraryIndex + 2 \MaslovIndex(\Lifted{\CircleAction})}_{\CircleAction \PullBack \rho} (\Manifold) \EndFullStop
                \end{equation}
            \end{definition}
            
            We discuss how the equivariant Floer Seidel map is compatible with filtrations and the Gysin sequence in \autoref{sec:equivariant-floer-seidel-map-properties}.
            
            For the quantum Seidel map, we put an action on the clutching bundle which lifts the natural rotation action of the sphere and which restricts to the action $\rho$ on the fibre above the south pole.
            The ``twisting by $\CircleAction$'' across the equator forces the action on the fibre above the north pole to be $\CircleAction \PullBack \rho$.
            With this action on the clutching bundle, the quantum Seidel map extends naturally to the equivariant setup.
            
            \begin{definition}
                [Equivariant quantum Seidel map]
                \label{def:equivariant-quantum-seidel-map-introduction}
                The \define{equivariant quantum Seidel map} is the map
                \begin{equation}
                    \EQuantumSeidelMap_{\Lifted{\CircleAction}} : \EQuantumCohomology^\ArbitraryIndex_\rho (\Manifold) \to \EQuantumCohomology^{\ArbitraryIndex + 2 \MaslovIndex(\Lifted{\CircleAction})} _{\CircleAction \PullBack \rho} (\Manifold)
                \end{equation}
                which counts equivalence classes $[(w, u)]$, where  $w \in \InfiniteSphere$ and $u$ is a pseudoholomorphic section of the clutching bundle.
                It is a $\NovikovRing$-module homomorphism which is compatible with the algebraic and geometric $\Integers [\uformal]$-module structures.
            \end{definition}
            
        \subsubsection{Properties}
        
            Equivariant PSS maps identify equivariant quantum cohomology with the equivariant Floer cohomology of an equivariant Hamiltonian of sufficiently small (positive) slope.
            
            \newcommand{\SmallEqntHam}{\Ham^{\eqnt\text{, small}}}
            \begin{proposition}
                [Compatibility with PSS maps]
                An analogous commutative diagram to \eqref{eqn:diagram-relating-quantum-and-floer-seidel-maps-introduction} holds for the equivariant maps:
                \begin{equation}
                    \label{eqn:diagram-relating-quantum-and-floer-seidel-maps-equivariant-case-introduction}
                    \begin{tikzcd}[row sep=3.5em, column sep=-4em]
                        \EQuantumCohomology^\ArbitraryIndex_\rho (\Manifold)
                            \arrow[rr, "\EQuantumSeidelMap_{\Lifted{\CircleAction}}"]
                            \arrow[d, "\substack{\text{equivariant}\\\text{PSS map}}"', "\cong"] 
                            &
                            & \EQuantumCohomology^{\ArbitraryIndex + 2\MaslovIndex(\Lifted{\CircleAction})}_{\CircleAction \PullBack \rho} (\Manifold)
                        \\
                        \EFloerCohomology^\ArbitraryIndex_\rho(\Manifold; \SmallEqntHam_0)
                            \arrow[rd, "\EFloerSeidel_{\Lifted{\CircleAction}}"', "\cong"]
                            &
                            & \EFloerCohomology ^{\ArbitraryIndex + 2\MaslovIndex(\Lifted{\CircleAction})} _{\CircleAction \PullBack \rho} (\Manifold; \SmallEqntHam_1)
                            \arrow[u, "\substack{\text{equivariant}\\\text{PSS map}}"', "\cong"]
                        \\
                        & \EFloerCohomology ^{\ArbitraryIndex + 2\MaslovIndex(\Lifted{\CircleAction})} _{\CircleAction \PullBack \rho} (\Manifold; \CircleAction\PullBack\SmallEqntHam_0)
                            \arrow[ru, "\substack{\text{equivariant}\\\text{continuation map}}"', pos=0.6]
                            &                  
                    \end{tikzcd}
                \end{equation}
                Here, $\SmallEqntHam_0$ satisfies an equivariance condition analogous to \eqref{eqn:equivariant-condition-on-hamiltonian-no-action} which corresponds to the action $\rho$.
                In contrast, $\CircleAction\PullBack\SmallEqntHam_0$ and $\SmallEqntHam_1$ satisfy the equivariance condition which corresponds to the action $\CircleAction \PullBack \rho$.
                Both $\SmallEqntHam_0$ and $\SmallEqntHam_1$ have small positive slope, but $\CircleAction \PullBack \SmallEqntHam_0$ has negative slope in general.
            \end{proposition}
            
            This proposition implies \eqref{eqn:commutative-diagram-equivariant-quantum-seidel-map-into-equivariant-symplectic-cohomology}.
            
            The maps for different actions compose exactly like the non-equivariant case as follows.
            
            \begin{proposition}
                [Composition of multiple actions]
                \label{prop:composition-of-multiple-actions-introduction}
                Let $\rho$, $\CircleAction_1$ and $\CircleAction_2$ be commuting Hamiltonian circle actions whose Hamiltonians are linear outside a compact subset of $\Manifold$.
                Suppose the Hamiltonians of $\CircleAction_1$ and $\CircleAction_2$ have nonnegative slope.
                The following diagrams commute. (We have omitted grading, the Hamiltonians and $\Manifold$ from the notation.)
                \begin{equation}
                    \begin{tikzcd}[column sep=1em]
                        & \EQuantumCohomology_{\CircleAction_1 \PullBack \rho} \arrow{dd}{\EQuantumSeidelMap_{\Lifted{\CircleAction_2}}} && & \EFloerCohomology_{\CircleAction_1 \PullBack \rho} \arrow{dd}{\EFloerSeidel_{\Lifted{\CircleAction_2}}} \\
                        \EQuantumCohomology_\rho \arrow{ur}{\EQuantumSeidelMap_{\Lifted{\CircleAction_1}}} \arrow[rd, "\EQuantumSeidelMap_{\Lifted{\CircleAction_2}\Lifted{\CircleAction_1}}"'] &&  & \EFloerCohomology_\rho \arrow{ur}{\EFloerSeidel_{\Lifted{\CircleAction_1}}} \arrow[rd, "\EFloerSeidel_{\Lifted{\CircleAction_2}\Lifted{\CircleAction_1}}"'] &  \\
                        & \EQuantumCohomology_{(\CircleAction_2 \CircleAction_1) \PullBack \rho} && & \EFloerCohomology_{(\CircleAction_2 \CircleAction_1) \PullBack \rho}
                    \end{tikzcd}
                \end{equation}
            \end{proposition}
            
            The quantum Seidel map intertwines the quantum product.
            More precisely, given $\alpha \in \QuantumCohomology^\ArbitraryIndex(\Manifold)$, the operations $\QuantumSeidelMap_{\Lifted{\CircleAction}}$ and  $\QuantumProduct \alpha$ commute, where $\QuantumProduct \alpha$ is the operation given by right-multiplication by $\alpha$.
            Thus the relation $\Commutator{\QuantumSeidelMap_{\Lifted{\CircleAction}}}{\QuantumProduct \alpha} = 0$ holds.
            
            The first complication for extending this relation to the equivariant setup is that the domain and codomain of $\EQuantumSeidelMap_{\Lifted{\CircleAction}}$ are different.
            We solve this by using the clutching bundle to relate the domain and codomain.
            Let $\alpha \in \ECohomology^\ArbitraryIndex(\ClutchingBundle)$ be an equivariant cohomology class of the clutching bundle $\ClutchingBundle$ with respect to the action we put on it.
            The restriction of $\alpha$ to the fibre over the south pole is a class $\alpha^+ \in \ECohomology^\ArbitraryIndex_\rho(\Manifold)$, while the restriction to the fibre over the north pole is a class $\alpha^- \in \ECohomology^\ArbitraryIndex_{\CircleAction \PullBack \rho}(\Manifold)$.
            
            Having established this, we can ask whether the analogue of the intertwining relation $\Commutator{\QuantumSeidelMap_{\Lifted{\CircleAction}}}{\QuantumProduct \alpha} = 0$ holds in the equivariant setup.
            From the non-equivariant relation, we can immediately deduce any failure to commute is $\SmallO(\uformal)$.
            The following theorem characterises the failure precisely.
            
            \begin{theorem}
                [Intertwining relation, \autoref{thm:interwining-relation-introduction}]
                The equation
                \begin{equation}
                    \label{eqn:quantum-intertwining-seidel-equivariant-overview}
                    \EQuantumSeidelMap_{\Lifted{\CircleAction}} (x \underset{\rho}{\QuantumProduct} \alpha^+)
                    -
                    \EQuantumSeidelMap_{\Lifted{\CircleAction}} (x) \underset{\CircleAction \PullBack \rho}{\QuantumProduct} \alpha^-
                    =
                    \uformal \CupProduct \EQuantumSeidelMap_{\Lifted{\CircleAction}, \alpha}(x)
                \end{equation}
                holds for all $\alpha \in \ECohomology^\ArbitraryIndex(\ClutchingBundle)$ and $x \in \EQuantumCohomology^\ArbitraryIndex_\rho(\Manifold)$, where
                \begin{equation}
                    \EQuantumSeidelMap_{\Lifted{\CircleAction}, \alpha} : \EQuantumCohomology^\ArbitraryIndex_\rho (\Manifold) \to \EQuantumCohomology^{\ArbitraryIndex + 2 \MaslovIndex(\Lifted{\CircleAction}) + \Modulus{\alpha} - 2} _{\CircleAction \PullBack \rho} (\Manifold)
                \end{equation}
                is a map defined in \autoref{sec:error-term-definition}.
            \end{theorem}
            
            The product $\QuantumProduct_\rho$ is the equivariant quantum product for the action $\rho$, and the symbol $\CupProduct$ denotes the action of the geometric $\Integers [\uformal]$-module structure.
            The map $\EQuantumSeidelMap_{\Lifted{\CircleAction}, \alpha}$ counts (equivariant) pseudoholomorphic sections of the clutching bundle which are weighted by the (equivariant) class $\alpha$.
            This weighting is easiest to understand when $\alpha \in \ECohomology^2(\ClutchingBundle)$ is degree 2.
            In this case, the map $\EQuantumSeidelMap_{\Lifted{\CircleAction}, \alpha}$ counts exactly the same pseudoholomorphic sections $u$ as the map $\EQuantumSeidelMap_{\Lifted{\CircleAction}}$, but with the weight $\alpha(u\PushForward(\FundamentalClass{\Sphere}))$.
            For the full definition, see \autoref{sec:error-term-definition}.
            
            In the realm of algebraic geometry, Maulik and Okounkov proved an analogous\footnote{
                The sections of the bundle must intersect the input $x \in \EQuantumCohomology^\ArbitraryIndex_\rho(\Manifold)$ over the south pole.
                In Maulik and Okounkov's conventions, the sections intersect $x$ over the north pole, and this is why their intertwining relation has different signs to ours.
            } intertwining relation for quiver varieties \cite[Proposition~8.2.1]{maulik_quantum_2012}, and our formula resembles theirs when $\alpha$ has degree 2.
            They prove the relation using \emph{virtual localization} \cite[Chapter~27]{hori_mirror_2003}.
            This technique converts counting sections into counting only the sections which are invariant under the $\Circle$-action on the clutching bundle.
            Any invariant section must be a constant section at a fixed point of the action $\CircleAction$ on $\Manifold$, however invariant sections are allowed to bubble over the poles.
            The result is a decomposition of $\EQuantumSeidelMap$ into three maps: $B^- \ComposedWith F \ComposedWith B^+$.
            The map $B^-$ counts bubbles over the north pole and $B^+$ counts bubbles over the south pole, both appropriately modified according to the virtual localization.
            The map $F$ corresponds to the constant section at a fixed point of the action $\CircleAction$.
            
            The intertwining relation is then proven for each of the three maps.
            For $B^-$ and $B^+$, it is a consequence of standard relations corresponding to gravitational descendant invariants (namely, the divisor equation and the topological recursion relation \cite[Chapter~26]{hori_mirror_2003}).
            For $F$, the intertwining relation is a topological result which relates $\alpha^-$ and $\alpha^+$ when both classes are restricted to the fixed locus of $\CircleAction$ on $\Manifold$.
            
            In contrast, our proof has a Floer-theoretic flavour: we construct an explicit 1-dimensional moduli space whose boundary gives \eqref{eqn:quantum-intertwining-seidel-equivariant-overview} on cohomology.
            
            To motivate our proof, consider first the following proof of the intertwining of the non-equivariant quantum Seidel map.
            Define a 1-dimensional moduli space of pseudoholomorphic sections which intersect a fixed Poincar\'{e} dual $\alpha^\Check$ of $\alpha$ along a fixed line of longitude $\LongitudeLine \subset \Sphere$.
            The boundary of this moduli space occurs when the intersection point is at either pole.
            When it is at the south pole, we recover the term $\QuantumSeidelMap (x \QuantumProduct \alpha)$, while at the north pole we get the term $- \QuantumSeidelMap (x) \QuantumProduct \alpha$.
            Summing these boundary components gives the equation $\QuantumSeidelMap (x \QuantumProduct \alpha) - \QuantumSeidelMap (x) \QuantumProduct \alpha = 0$ as desired.
            
            In the equivariant case, we must allow the line of longitude to vary with $\infsphereelement \in \InfiniteSphere$.
            This is because the intersection condition must be preserved by the equivalence relation $\sim$.
            We fix an equivariant assignment of lines of longitude $\infsphereelement \mapsto \LongitudeLine_\infsphereelement$ for an invariant dense open set of $\infsphereelement \in \InfiniteSphere$.
            Note that a global assignment is not possible: the set of lines of longitude is isomorphic to $\Circle$, however there are no equivariant maps $\InfiniteSphere \to \Circle$.
            
            For the equivariant 1-moduli space, we ask that the equivariant section $[(w, u)]$ satisfies $u(z) \in \alpha^\Check$ for some $z \in \LongitudeLine_\infsphereelement$.
            As per the non-equivariant case, the two poles yield the two terms on the left-hand side of \eqref{eqn:quantum-intertwining-seidel-equivariant-overview}.
            The remaining term in \eqref{eqn:quantum-intertwining-seidel-equivariant-overview} comes from a limit in which $\infsphereelement$ exits the dense open set.
            
            The computation behind \autoref{eg:projective-space} verifies that the right-hand side of \eqref{eqn:quantum-intertwining-seidel-equivariant-overview} is nonzero even in straightforward cases.
    
    \subsection{Examples}
    \label{sec:example-computations-introduction-summary}
    
        We compute the equivariant quantum Seidel map for the following three spaces: the complex plane, complex projective space and the total space of the tautological line bundle $\StandardLineBundle_{\Projective^N}(-1)$.
        Through these examples, we demonstrate how the map may be used to compute equivariant quantum cohomology and equivariant symplectic cohomology.
        
        For the complex plane, we deduce the equivariant quantum Seidel map from the equivariant Floer complex which was computed in \cite[Section~8.1]{zhao_periodic_2014}.
        In both other examples, we find the map by directly computing some coefficients and deducing the rest by repeated application of the intertwining relation \eqref{eqn:quantum-intertwining-seidel-equivariant-introduction}.
        
        We use the parameterisation $\Circle = \RealNumbers / \Integers$ throughout.
        
        \begin{example}
            [Complex plane]
            \label{eg:complex-plane-introduction}
            The complex plane has a Hamiltonian circle action $\CircleAction_\CircleElement(z) = \ExponentialNumber^{2 \PiNumber \ImaginaryNumber \CircleElement} z$.
            The origin $0_\ComplexNumbers \in \ComplexNumbers$ is the unique fixed point of $\CircleAction$.
            Thus $\ComplexNumbers$ is equivariantly contractible and its symplectic form is globally exact, so for any nonnegative $r$, we have $\EQuantumCohomology^\ArbitraryIndex_{\CircleAction^{-r}}(\ComplexNumbers) \cong \Integers [\uformal]$.
            The equivariant quantum Seidel map is
            \begin{equation}
                \begin{aligned}
                    \EQuantumSeidelMap_{\CircleAction} : \EQuantumCohomology^\ArbitraryIndex_{\CircleAction^{-r}}(\ComplexNumbers) &\to \EQuantumCohomology^{\ArbitraryIndex + 2} _{\CircleAction^{-(r+1)}}(\ComplexNumbers)
                    \\
                    1 &\mapsto (r+1) \uformal\EndFullStop
                \end{aligned}
            \end{equation}
        \end{example}
        
        \begin{example}
            [Projective space]
            \label{eg:projective-space}
            The complex projective space $\Projective^\dimM$ with its Fubini-Study symplectic form has a Hamiltonian action $\CircleAction$ given by 
            \begin{equation}
                \CircleElement \cdot [z_0: z_1:\cdots:z_\dimM] = [z_0: \ExponentialNumber^{2 \PiNumber \ImaginaryNumber \CircleElement} z_1 : \cdots : \ExponentialNumber^{2 \PiNumber \ImaginaryNumber \CircleElement} z_\dimM]\EndComma
            \end{equation}
            for any $(z_0, z_1,\cdots,z_\dimM) \in \HighDimensionalSphere{2 \dimM + 1} \subset \ComplexNumbers^{\dimM + 1}$.
            The $\CircleAction$-invariant Morse function $\MorseFunction_{\Projective^\dimM} ([z_0:\cdots:z_\dimM]) = \sum_{k=0}^\dimM k \Modulus{z_k}^2$ has critical points $e_0, \ldots, e_\dimM$, where $e_k$ has Morse index $2k$.
            For any nonnegative integer $r$, we have
            \begin{equation}
                \EQuantumCohomology^\ArbitraryIndex_{\CircleAction^{-r}}(\Projective^\dimM) \cong \Integers[\NovVariable^{\pm 1}] \GradedCompletedTensorProduct \Integers [\uformal]\langle e_0, \ldots, e_\dimM \rangle \EndComma
            \end{equation}
            where the Novikov variable $\NovVariable$ is a formal variable of degree $2(\dimM + 1)$.
            The equivariant quantum Seidel map is the map
            \begin{equation}
                \EQuantumSeidelMap_{\Lifted{\CircleAction}} : \EQuantumCohomology^\ArbitraryIndex_{\CircleAction^{-r}}(\Projective^\dimM) \to \EQuantumCohomology^{\ArbitraryIndex + 2\dimM} _{\CircleAction^{-(r+1)}}(\Projective^\dimM)
            \end{equation}
            given by
            \begin{equation}
                \left\{
                \begin{aligned}
                e_0 & \mapsto \sum_{l=0}^\dimM (r+1)^{\dimM - l} \uformal^{\dimM - l} \ e_l, & \\
                e_k & \mapsto \sum_{l=0}^{k - 1} (r+1)^{k-1-l} \NovVariable \uformal^{k-1-l} \ e_l, & k = 1, \ldots, \dimM\EndFullStop
                \end{aligned}
                \right.
            \end{equation}
        \end{example}
        
        Computing the quantum Seidel map and using  \eqref{eqn:non-equivariant-quantum-intertwining-relation} is one way to compute the (non-equivariant) quantum product on $\Projective^\dimM$ \cite[Section~5]{mcduff_topological_2006}.
        In our computations behind \autoref{eg:projective-space}, we demonstrate this extends to the equivariant case: we use the equivariant quantum Seidel map and \autoref{thm:interwining-relation-introduction} to derive the equivariant quantum product on $\Projective^\dimM$.
        Equivariant quantum Seidel maps have already been used to compute the equivariant quantum product on $\Projective^\dimM$ \cite[Section~4.4]{iritani_shift_2017}, though our direct method is new.
        The product on $\Projective^\dimM$ was previously known; it had been derived using equivariant quantum Littlewood–Richardson coefficients \cite{mihalcea_equivariant_2006}.
        
        \begin{example}
            [Tautological line bundle]
            The total space $\tautLB{n}$ of the tautological line bundle over projective space $\Projective^n$ is a monotone convex symplectic manifold \cite[Section~7]{ritter_floer_2014}.
            The fibres are symplectic submanifolds, and the circle action $\CircleAction$ which rotates the fibres is a linear Hamiltonian circle action.
            The action $\CircleAction$ fixes the image of the zero section $\ZeroSection$, and, like \autoref{eg:complex-plane-introduction}, $\tautLB{n}$ equivariantly contracts onto $\ZeroSection$ with the trivial circle action.
            We use the same Morse function on $\ZeroSection \cong \Projective^n$ as for \autoref{eg:projective-space}, so we have
            \begin{equation}
                \EQuantumCohomology^\ArbitraryIndex _{\CircleAction^{-r}} (\tautLB{n}) \cong \Integers [\NovVariable^{\pm 1}] \GradedCompletedTensorProduct \Integers [\uformal] \langle e_0, \ldots, e_n \rangle \EndComma
            \end{equation}
            where the Novikov variable $\NovVariable$ now has degree $2n$ (so $\NovikovRing = \Integers [\NovVariable^{\pm 1}]$ is the Novikov ring).
            The equivariant quantum Seidel map is the map
            \begin{equation}
                \EQuantumSeidelMap_{\Lifted{\CircleAction}} : \EQuantumCohomology^\ArbitraryIndex_{\CircleAction^{-r}}(\tautLB{n}) \to \EQuantumCohomology^{\ArbitraryIndex + 2} _{\CircleAction^{-(r+1)}}(\tautLB{n})
            \end{equation}
            given by
            \begin{equation}
                \label{eqn:equivariant-quantum-seidel-map-tautological-line-bundle-introduction}
                e_k \mapsto \begin{cases}
                -e_{k+1} + (r+1) \uformal e_k & k < n \EndComma \\
                \NovVariable e_1 + (r+1) \uformal e_n - (r+1) \uformal \NovVariable e_0 & k = n \EndFullStop
            \end{cases}
            \end{equation}
            Unlike the non-equivariant quantum Seidel map on $\tautLB{n}$, this is an injective map.
            
            From \eqref{eqn:equivariant-quantum-seidel-map-tautological-line-bundle-introduction}, we derive the equivariant quantum product as given by
            \begin{equation}
                e_1 \underset{\CircleAction^{-r}}{\QuantumProduct} e_k = \begin{cases}
                e_{k+1} & k < n \EndComma \\
                - \NovVariable e_1 + r \uformal \NovVariable e_0 & k = n \EndFullStop
            \end{cases}
            \end{equation}
            This product has a term which is not detected by either quantum cohomology or equivariant cohomology: it exists only in equivariant quantum cohomology.
            
            We can use \eqref{eqn:equivariant-quantum-seidel-map-tautological-line-bundle-introduction} to find the equivariant symplectic cohomology of $\tautLB{n}$ using an argument of Ritter \cite[Theorem~22]{ritter_floer_2014}.
            We deduce that the equivariant symplectic cohomology $\ESymplecticCohomology^\ArbitraryIndex_{\CircleAction^{-r}} (\tautLB{n})$ is a $\NovikovRing \GradedCompletedTensorProduct \Integers [\uformal]$-module which is not finitely generated and which satisfies
            \begin{equation}
                \EQuantumCohomology^\ArbitraryIndex_{\CircleAction^{-r}} (\tautLB{n}) \subsetneq \ESymplecticCohomology^\ArbitraryIndex_{\CircleAction^{-r}} (\tautLB{n}) \subsetneq \big( \EQuantumCohomology^\ArbitraryIndex_{\CircleAction^{-r}} (\tautLB{n}) \big) _{\Integers [\uformal] \setminus \Set{0}} \EndComma
            \end{equation}
            where the left inclusion is the equivariant $c^\ArbitraryIndex$ map and the module on the right is the localisation by $\Integers [\uformal] \setminus \Set{0}$.
            If we perform this localisation to all three modules, we find that the localised equivariant symplectic cohomology is isomorphic to the localised equivariant quantum cohomology, which is equivalent to a version of Zhao's localisation theorem \eqref{eqn:zhao-localisation-theorem}.
            
            We express $\ESymplecticCohomology^\ArbitraryIndex_{\CircleAction^{-r}} (\tautLB{n})$ as a $\NovikovRing \GradedCompletedTensorProduct \Integers [\uformal]$-submodule of the localised equivariant quantum cohomology with an explicit set of generators.
            The generators are defined by a recurrence relation induced by \eqref{eqn:equivariant-quantum-seidel-map-tautological-line-bundle-introduction}.
        \end{example}

\section{Floer theory}

    \label{sec:floer-theory}

    In \autoref{sec:symplectic-manifold-definitions-and-assumptions}, we clarify the assumptions we place on our symplectic manifold.
    We proceed in \autoref{sec:floer-cohomology-all} by defining Floer cohomology and symplectic cohomology using the same conventions as \cite{seidel_$_1997, ritter_floer_2014}.
    A more complete explanation of the construction of Floer cohomology may be found in \cite{salamon_lectures_1997}.
    Next, we introduce the Hamiltonian circle actions which yield the Seidel map in \autoref{def:linear-ham-circle-action}, and define the Floer Seidel map in \autoref{sec:floer-seidel-map}.

    \subsection{Symplectic manifolds}
    
        \label{sec:symplectic-manifold-definitions-and-assumptions}
    
        Let $\Manifold$ be a $2 \dimM$-dimensional smooth manifold with a symplectic form $\SymplecticForm$.
        For convenience, we assume throughout that $\Manifold$ is nonempty and connected.
        There are two additional conditions that we will impose on $\Manifold$.
        The first establishes a relationship between the cohomology class of $\SymplecticForm$ and the first Chern class, and the second controls the behaviour of the symplectic form when the manifold is open.
        
        \begin{definition}
            \label{def:nonnegatively-monotone}
            Denote by $\FirstChernClass \in \Cohomology^2(\Manifold)$ the first Chern class of the symplectic vector bundle $(\TangentSpace \Manifold, \SymplecticForm)$.
            The symplectic manifold $\Manifold$ is \define{nonnegatively monotone} if either:
            \begin{itemize}
                \item there is $\lambda \ge 0$ such that $\SymplecticForm(A) = \lambda \FirstChernClass(A)$ for all $A \in \HomotopyGroup_2(\Manifold)$; or
                
                \item $\FirstChernClass(A) = 0$ for all $A \in \HomotopyGroup_2(\Manifold)$.
            \end{itemize}
        \end{definition}
        
        By an analogous argument to the proof of Lemma~1.1 in \cite{hofer_floer_1995}, a symplectic manifold is nonnegatively monotone if and only if the implication
        \begin{equation}
            \label{eqn:nonnegatively-monotone-negative-spheres}
            \FirstChernClass(A) < 0 \implies \SymplecticForm(A) \le 0
        \end{equation}
        holds for all $A \in \HomotopyGroup_2(M)$.
        Such a manifold has the property that, for any compatible almost complex structure, all pseudoholomorphic curves have nonnegative first Chern number.
        
        \begin{definition}
            \label{def:convex-symplectic-manifold}
            A \define{convex symplectic manifold} is a symplectic manifold $\Manifold$ that is equipped with a closed $(2\dimM - 1)$-dimensional manifold $\ContactManifold$, a contact form $\ContactForm$ on $\ContactManifold$ and a map
            \begin{equation}
                \label{eqn:convex-end-parameterisation}
                \ConvexCoordMap : \ContactManifold \times \IntervalClosedOpen{1}{\Infinity} \to \Manifold
            \end{equation}
            such that $\ConvexCoordMap$ is a diffeomorphism onto its image, the set $\Manifold \setminus \ConvexCoordMap(\ContactManifold \times \IntervalClosedOpen{1}{\Infinity})$ is relatively compact\footnote{
                A subset of a topological space is \define{relatively compact} if its closure is compact.
            } and
            \begin{equation}
                \label{eqn:symplectic-form-on-convex-end}
                \ConvexCoordMap\PullBack \SymplecticForm = \ExteriorDerivative(R \ContactForm)
            \end{equation}
            holds on $\ContactManifold \times \IntervalClosedOpen{1}{\Infinity}$.
            Here, $R \in \IntervalClosedOpen{1}{\Infinity}$ is the \define{radial coordinate}\footnote{
                The set $\Set{R \le R_0}$ is defined to be $\Manifold \setminus \ConvexCoordMap(\ContactManifold \times \IntervalOpen{R_0}{\Infinity})$ and is compact for all $R_0 \ge 1$.
            } and the image of $\ConvexCoordMap$ is the \define{convex end} of $\Manifold$.
        \end{definition}
        
        \remark{
            \label{rem:closed-manifold-is-trivially-convex}
            We emphasise three features of this definition.
            \begin{itemize}
                \item The manifold $\ContactManifold$, the contact form $\ContactForm$ and the diffeomorphism $\ConvexCoordMap$ are all part of the data of a convex symplectic manifold.
                Consequently, we can use the coordinates provided by $\ConvexCoordMap$ without worrying about whether our constructions are independent of this choice (c.f. \autoref{rem:dependence-of-symplectic-cohomology-on-convex-end}).
                
                \item A closed symplectic manifold is convex since we allow the manifold $\ContactManifold$ to be empty\footnote{
                    We adopt the convention that the \define{empty manifold} is a disconnected closed oriented manifold of every dimension.
                    This means that the dimension of a manifold $X$ is not well-defined; the statement $\Dimension X = k$ is to be interpreted as `$X$ is $k$-dimensional'.
                }.
                With this convention, we are able to prove all of our results for closed manifolds and (non-trivially) convex symplectic manifolds simultaneously.
                
                \item The symplectic form $\SymplecticForm$ does not have to be exact on all of $\Manifold$, since \eqref{eqn:symplectic-form-on-convex-end} applies only on the convex end of $\Manifold$.
                Indeed, symplectic forms on closed manifolds are never globally exact.
            \end{itemize} 
        }
        
        Henceforth, $\Manifold$ will be a nonnegatively monotone convex symplectic manifold.
        
        Finally, we assume that $\RealNumbers \setminus \ReebPeriods$ is unbounded so that symplectic cohomology is a meaningful direct limit (see \autoref{def:symplectic-cohomology}).
        Here $\ReebPeriods$ is the set of Reeb periods as defined in \eqref{eqn:reeb-periods-notation-with-zero}.
        
        \begin{remark}
            [Orientations]
            This paper uses integral coefficients, and thus we require orientations on all moduli spaces.
            For this purpose, let $\Orientation$ be a coherent orientation, defined as in \cite[Appendix~B]{ritter_topological_2013}.
            The proof of all orientation signs in this paper is omitted.
        \end{remark}
        
    \subsection{Floer cohomology}
    
        \label{sec:floer-cohomology-all}
            
    \subsubsection{Hamiltonian dynamics}
    
        Let $\Circle = \RealNumbers / \Integers$ and $\Disc = \SetCondition{z \in \ComplexNumbers}{\Modulus{z} \le 1}$.
        
        A \define{(time-dependent\footnote{
            Throughout, a \define{time-dependent} object is a smooth $\Circle$-family of objects, and a \define{$B$-dependent} object is a smooth $B$-family of objects for any manifold $B$.
        }) Hamiltonian (function)} $\Ham$ is a smooth function $\Circle \times \Manifold \to \RealNumbers$.
        Its \define{(time-dependent) Hamiltonian vector field $\HamiltonianVectorField{\Ham}$} is the unique $\Circle$-family of vector fields by $\SymplecticForm( \Argument \CommaSpace \HamiltonianVectorField{\Ham_t} ) = \ExteriorDerivative \Ham_t$.
        The \define{Hamiltonian flow $\HamiltonianFlow_\Ham$} is the flow along the vector field $\HamiltonianVectorField{\Ham}$, so $\HamiltonianFlow_\Ham ^t$ satisfies
        \begin{equation}
            \partial_t \left(\HamiltonianFlow_\Ham ^t (\ManifoldElement)\right) = \HamiltonianVectorField{\Ham_t}\left(\HamiltonianFlow_\Ham ^t (\ManifoldElement) \right) \EndFullStop
        \end{equation}
        A \define{Hamiltonian orbit} is a loop $x : \Circle \to \Manifold$ which satisfies $\partial_t x (t) = \HamiltonianVectorField{\Ham_t}(x(t))$ for all $t \in \Circle$.
        It is \define{nondegenerate} if the linear map $\Derivative \HamiltonianFlow_\Ham ^1 : \TangentSpace_{x(0)} \Manifold \to \TangentSpace_{x(1)} \Manifold$ has no eigenvalue equal to 1.
        Denote by $\HamOrbits(\Ham)$ the set of Hamiltonian orbits of $\Ham$.
        
        The Hamiltonian $\Ham$ is \define{linear of slope $\slope$} if the identity $\Ham_t(\ConvexCoordMap(y, R)) = \slope R + \mu$ holds at infinity\footnote{
            In the context of convex manifolds, a condition holds \define{at infinity} if there exists $R_0 \ge 1$ such that the condition holds on $\ConvexCoordMap(\ContactManifold \times \IntervalClosedOpen{R_0}{\Infinity})$.
            Notice that if finitely-many conditions each hold individually at infinity, then their conjunction holds at infinity (i.e., they all hold in a common region at infinity).
            All statements hold tautologically at infinity on a closed symplectic manifold (see \autoref{rem:closed-manifold-is-trivially-convex}).
        } for some constant $\mu$.
        The vector field of such a Hamiltonian is a multiple of the Reeb vector field\footnote{
            The \define{Reeb vector field} $\Reebvf$ associated to $\ContactForm$ is the vector field on $\ContactManifold$ uniquely defined by $\InteriorDerivative_{\Reebvf}(\ExteriorDerivative \ContactForm) = 0$ and $\InteriorDerivative_{\Reebvf}\ContactForm = 1$.
        } $\Reebvf$, so, at infinity, we have
        \begin{equation}
            \label{eqn:hamiltonian-vector-field-at-infinity-is-multiple-of-reeb-vector-field}
            \HamiltonianVectorField{\Ham_t} = \slope \Reebvf \oplus 0 \in \TangentSpace \ContactManifold \DirectSum \TangentSpace \IntervalClosedOpen{1}{\Infinity} \cong \TangentSpace \Manifold \EndFullStop
        \end{equation}
        As such, any Hamiltonian orbit in this region corresponds to a $\slope$-periodic flow along the Reeb vector field.
        
        A \define{Reeb period} is a nonzero number $\lambda \in \RealNumbers \setminus \Set{0}$ such that there exists a point $y \in \ContactManifold$ such that the flow of the Reeb vector field from $y$ is $\lambda$-periodic.
        If $\ContactManifold$ is empty, set $\ReebPeriods = \emptyset$, and otherwise set 
        \begin{equation}
            \label{eqn:reeb-periods-notation-with-zero}
            \ReebPeriods = \Set{\text{Reeb periods}} \Union \Set{0} \EndFullStop
        \end{equation}
        Notice $k \lambda \in \ReebPeriods$ for any $\lambda \in \ReebPeriods$ and any $k \in \Integers$.
        
        Thus, if $\Ham$ is linear of slope $\slope$ with $\slope \notin \ReebPeriods$, then all Hamiltonian orbits of $\Ham$ lie in a compact region of $\Manifold$.
        If moreover all Hamiltonian orbits are nondegenerate, then $\HamOrbits(\Ham)$ is finite.
            
    \subsubsection{Almost complex structures}
    
        Let $m$ be a point in the convex end of $\Manifold$.
        The $\SymplecticForm$-compatible almost complex structure $\AlmostComplexStructure$ is \define{convex at the point $\ManifoldElement$} if $- \ExteriorDerivative R \ComposedWith \AlmostComplexStructure = R \ContactForm$ holds at $\ManifoldElement$.
        This means that the almost complex structure respects the direct sum decomposition 
        \begin{equation}
            \label{eqn:decomposition-of-tangent-space-at-infinity}
            \TangentSpace_y \Manifold \Isomorphism (\RealNumbers \partial_R \DirectSum \RealNumbers \Reebvf) \DirectSum \Kernel \ContactForm
        \end{equation}
        and satisfies $\AlmostComplexStructure (R\partial_R) = \Reebvf$ at $\ManifoldElement$.
        
        Let $B$ be a manifold and let $\FamilyACS = (\FamilyACS_b)_{b \in B}$ be a smooth family of $\SymplecticForm$-compatible almost complex structures.
        The family $\FamilyACS$ is \define{convex} if, at infinity, the almost complex structure $\FamilyACS_b$ is convex for all $b \in B$.
        
        \definition{
            A choice of \define{Floer data} is a pair $(\Ham, \FamilyACS)$, where $\Ham$ is a linear time-dependent Hamiltonian with slope not in $\ReebPeriods$ and $\FamilyACS$ is a convex time-dependent $\SymplecticForm$-compatible almost complex structure.
        }
        
    \subsubsection{Pseudoholomorphic spheres}
        
        Let $\StandardACS$ be the standard almost complex structure on the sphere $\Projective^1$.
        Let
        \begin{equation}
            \label{eqn:interesting-spheres-definition}
            \InterestingSpheres = \frac{\HomotopyGroup_2(M)}{\Kernel \FirstChernClass \Intersection \Kernel \SymplecticForm} \EndFullStop
        \end{equation}
        A curve $u : \Projective^1 \to \Manifold$ is \define{$\AlmostComplexStructure$-holomorphic} if $(\Derivative u) \ComposedWith \StandardACS = \AlmostComplexStructure \ComposedWith (\Derivative u)$ and $u$ \define{represents the class $A \in \InterestingSpheres$} if $\EquivalenceClass{u} = A$.
        
        Let  $B$ be a manifold and $\FamilyACS$ be a convex $B$-dependent family of $\SymplecticForm$-compatible almost complex structures.
        The moduli space $\ModuliSpace(A, \FamilyACS)$ is the space of pairs $(b, u)$, where $u$ is a simple\footnote{
            The holomorphic curve $u : \Projective^1 \to \Manifold$ is \define{multiply-covered} if it is a composition of a holomorphic branched covering map $\Projective^1 \to \Projective^1$ with degree strictly greater than 1 and a second holomorphic curve.
            The curve $u$ is \define{simple} otherwise.
        } $\FamilyACS_b$-holomorphic curve representing $A \in \InterestingSpheres$.
        
        Denote by $\ChernBoundedSpheres_k(\FamilyACS)$ the set of pairs $(b, \ManifoldElement) \in B \times \Manifold$ such that $\ManifoldElement$ lies in the image of a nonconstant $\FamilyACS_b$-holomorphic sphere $u$ with $\FirstChernClass(u) \le k$.
            
    \subsubsection{Floer solutions}
        
        Denote by $\ContractibleLoopSpace \Manifold$ the space of contractible smooth maps $\Circle \to \Manifold$.
        Define a cover of this space $\ContractibleLoopSpaceWithFillings{\Manifold}$ as the space of pairs $(x \CommaSpace u)$ of loops $x \in \ContractibleLoopSpace \Manifold$ and smooth maps $u : \Disc \to \Manifold$ satisfying $x(t) = u(\ExponentialNumber^{2 \PiNumber \ImaginaryNumber t})$, considered up to the equivalence
        \begin{equation}
            \label{eqn:equivalence-on-contractible-loop-space-with-fillings}
            (x \CommaSpace u) \sim (x \CommaSpace u \Prime) \iff \int_\Disc u \PullBack \widetilde{\FirstChernClass} = \int_\Disc {u \Prime} \PullBack \widetilde{\FirstChernClass} \text{ and } \int_\Disc u \PullBack \SymplecticForm = \int_\Disc {u \Prime} \PullBack \SymplecticForm \EndComma
        \end{equation}
        where $\widetilde{\FirstChernClass}$ is a differential 2-form representing the first Chern class $\FirstChernClass$.
        The deck transformation group of $\ContractibleLoopSpaceWithFillings{\Manifold}$ is $\InterestingSpheres$, which acts on $\ContractibleLoopSpaceWithFillings{\Manifold}$ by `adding $A \in \InterestingSpheres$ to the filling $u$'.
        This action is described explicitly in \cite[Section~5]{hofer_floer_1995}.
        Let
        \begin{equation}
            \FilledHamOrbits(\Ham) = \SetCondition{(x \CommaSpace u) \in \ContractibleLoopSpaceWithFillings{\Manifold}}{x \in \HamOrbits(\Ham)} \EndFullStop
        \end{equation}
        
        Let $(\Ham, \FamilyACS)$ be a choice of Floer data.
        The \define{action functional} associated to $\Ham$ is the map $\ActionFunctional_\Ham : \ContractibleLoopSpaceWithFillings{\Manifold} \to \RealNumbers$ given by
        \begin{equation}
            \label{eqn:action-functional}
            \ActionFunctional_\Ham (x, u) = - \int_\Disc u \PullBack \SymplecticForm + \int_{t = 0}^1 \Ham_t (x (t)) \wrt{t} \EndFullStop
        \end{equation}
        The set of critical points of $\ActionFunctional_\Ham$ equals $\FilledHamOrbits(\Ham)$.
        A smooth map $u : \RealNumbers \times \Circle \to \Manifold$ is a \define{Floer solution} if it satisfies
        \begin{equation}
            \label{eqn:floer-equation}
            \partial_s u + \FamilyACS_t (\partial_t u - \HamiltonianVectorField{H \CommaSpace t}) = 0
        \end{equation}
        for all $(s, t) \in \RealNumbers \times \Circle$.
        The left side of \eqref{eqn:floer-equation} is abbreviated by $\AntilinearSection_{\Ham \CommaSpace \FamilyACS}(u)$.
        The \define{energy} $\Energy(u)$ of a map $u$ is
        \begin{equation}
            \label{eqn:energy-of-floer-solution}
            \Energy(u) = \int_{\RealNumbers \times \Circle} \Norm{\partial_s u}_{\FamilyACS_t}^2 \ExteriorDerivative s \wedge \ExteriorDerivative t \EndComma
        \end{equation}
        where $\Norm{\Argument}_\AlmostComplexStructure$ is the norm associated to the $\AlmostComplexStructure$-invariant metric $\SymplecticForm(\Argument, \AlmostComplexStructure \Argument)$.
        Suppose the Hamiltonian orbits of $\Ham$ are nondegenerate.
        If the energy of the Floer solution $u$ is finite, then there exist two Hamiltonian orbits $x^\pm$ such that
        \begin{align}
            \label{eqn:Floer-solution-convergence-to-limits}
            \lim_{s \to \pm \Infinity} u(s, t) &= x^\pm (t) & \lim_{s \to \pm \Infinity} \partial_s u(s, t) &= 0 \EndComma
        \end{align}
        where the limits denote uniform convergence\footnote{
            The uniform convergence is a priori with respect to the $t$-dependent metric $\Norm{\Argument}_{\FamilyACS_t}$, but the property holds for any $t$-dependent Riemannian metric.
        } in $t$, and moreover
        \begin{equation}
            \label{eqn:energy-is-difference-of-action-at-limits}
            \Energy(u) = \ActionFunctional_H (x^-, u^-) - \ActionFunctional_H (x^+, u^- \ConnectedSum u)
        \end{equation}
        for any filling $u^-$ of $x^-$.
        
        By a maximum principle \cite[Appendix~D]{ritter_topological_2013}, there is a compact region of $\Manifold$ such that all Floer solutions lie completely within the region.
        
        Let $\WithFilling{x}^\pm = (x^\pm, u^\pm) \in \FilledHamOrbits(\Ham)$.
        Denote by $\ModuliSpace(\WithFilling{x}^- \CommaSpace \WithFilling{x}^+)$ the moduli space of Floer solutions $u$ which satisfy \eqref{eqn:Floer-solution-convergence-to-limits} and $u^+ = u^- \ConnectedSum u$.
            
    \subsubsection{Moduli space of Floer trajectories}
    
        \label{sec:moduli-space-of-floer-trajectories-properties}
        
        In order to ensure the moduli space $\ModuliSpace(\WithFilling{x}^-,\allowbreak \WithFilling{x}^+)$ is a smooth finite-dimensional manifold and has other desired behaviour, we impose a number of \define{regularity conditions} on the Floer data.
        For \define{regular} Floer data, the Hamiltonian orbits are nondegenerate, the moduli space $\ModuliSpace(A; \FamilyACS)$ is a smooth canonically-oriented\footnote{
            In order to orient all other moduli spaces in this paper, we have had to choose orientation data, such as coherent orientations (or orientations of the unstable manifolds in the case of Morse theory).
            The orientation of $\ModuliSpace(A; \FamilyACS)$ is intrinsic, however it does rely on an orientation of the parameter space $\Circle$, which we fix.
        } manifold of dimension\footnote{
            The 1 in this formula corresponds to $\Dimension \Circle$, and will be $\Dimension B$ for $B$-dependent almost complex structures.
        } $2 \dimM + 2 \FirstChernClass(A) + 1$ and the moduli space $\ModuliSpace(\WithFilling{x}^- \CommaSpace \WithFilling{x}^+)$ of Floer solutions is a smooth oriented manifold of dimension 
        \begin{equation}
            \Dimension\ModuliSpace(\WithFilling{x}^- \CommaSpace \WithFilling{x}^+) = \ConleyZehnderIndex (\WithFilling{x}^-) - \ConleyZehnderIndex (\WithFilling{x}^+) \EndFullStop
        \end{equation}
        Here, we denote by $\ConleyZehnderIndex (\WithFilling{x})$ the Conley-Zehnder index for $\WithFilling{x} \in \FilledHamOrbits(\Ham)$.
        In addition, we have $(t, x(t)) \notin \ChernBoundedSpheres_1(\FamilyACS)$ and $(t, u(s, t)) \notin \ChernBoundedSpheres_0(\FamilyACS)$ for all Hamiltonian orbits $x$ and Floer solutions $u \in \ModuliSpace(\WithFilling{x}^- \CommaSpace \WithFilling{x}^+)$ when $\Dimension \ModuliSpace(\WithFilling{x}^- \CommaSpace \WithFilling{x}^+) \le 2$.
            
        \begin{remark}
            [Regularity]
            In this paper, we will not always list all of the regularity conditions, however we will mention those that are more uncommon\footnote{
                In order to find that moduli spaces are manifolds of a given dimension, the standard method is to find a Fredholm operator whose kernel describes the tangent space of the moduli space.
                The corresponding regularity condition is that the operator is onto, so that some version of the implicit function theorem may be applied.
                A version of the Sard-Smale theorem shows this to be generic \cite{floer_transversality_1995}.
                The nondegeneracy of the orbits follows by \cite[Remark~5.4.8]{audin_morse_2014}, and the avoidence of bubbling here is due to \cite{hofer_floer_1995, seidel_$_1997}.
            }.
            The conditions are always motivated by improving the behaviour of the moduli spaces.
            Here, we have used them to ensure the moduli spaces of Floer solutions have the structure of manifolds without bubbling in dimensions 1 and 2.
            In our notation, regularity conditions will always be satisfied by generic\footnote{
                In a topological space, a subset is \define{generic} or \define{of second category} if it is a countable intersection of open dense subsets.
                A condition on (Floer) data is \define{generic} if the set of data satisfying the condition forms a generic subset of all data.
            } data, and hence will exist.
        \end{remark}
        
        Assume $(\Ham, \FamilyACS)$ is regular.
        When $\WithFilling{x}^- \ne \WithFilling{x}^+$, the moduli space $\ModuliSpace(\WithFilling{x}^- \CommaSpace \WithFilling{x}^+)$ admits a smooth free $\RealNumbers$-action given by $s$-translation.
        The quotient $\QuotientedModuliSpace(\WithFilling{x}^- \CommaSpace \WithFilling{x}^+)$ is a smooth manifold.
        The 0-dimensional moduli spaces $\QuotientedModuliSpace(\WithFilling{x}^- \CommaSpace \WithFilling{x}^+)$ are all compact by the so-called compactification argument (for example, see \cite[Section~3.1]{salamon_lectures_1997}).
        
        The 1-dimensional moduli spaces $\QuotientedModuliSpace(\WithFilling{x}^- \CommaSpace \WithFilling{x}^+)$ may be given the structure of an oriented compact 1-manifold with boundary by attaching the endpoints
        \begin{equation}
            \label{eqn:boundary-of-1-dimensional-moduli-space-of-unparameterised-floer-trajectories}
            \bigcup_{
                \substack{
                    \WithFilling{x}^0 \in \FilledHamOrbits(\Ham) \\ \ConleyZehnderIndex(x^-) = \ConleyZehnderIndex(x^0) + 1
                }
            } \QuotientedModuliSpace(\WithFilling{x}^- \CommaSpace \WithFilling{x}^0) \times \QuotientedModuliSpace(\WithFilling{x}^0 \CommaSpace \WithFilling{x}^+)
        \end{equation}
        using the so-called gluing maps (for example, see \cite[Section~3.3]{salamon_lectures_1997}).
        Any element of \eqref{eqn:boundary-of-1-dimensional-moduli-space-of-unparameterised-floer-trajectories}, and more generally any chain of Floer trajectories, is a \define{broken} Floer trajectory.
            
    \subsubsection{Floer cochain complex}
        
        Let $(\Ham, \FamilyACS)$ be a regular choice of Floer data.
        
        Let $\NovVariable$ be a formal variable.
        A \define{formal power series with coefficients in $\Integers$ and exponents in $\InterestingSpheres$} is a formal sum $\sum_{A \in \InterestingSpheres} \alpha_A \NovVariable^A$.
        The monomial $\NovVariable^A$ is given the grading $2 \FirstChernClass(A) \in \Integers$.
        Set $\NovikovRing^k$ to be the group of formal power series which are supported by only monomials of grading $k$ and satisfy the condition  that the set $$\SetCondition{A \in \InterestingSpheres}{\SymplecticForm(A) \le c \CommaSpace \alpha_A \ne 0}$$ is finite for all $c \in \RealNumbers$.
        The \define{Novikov ring} $\NovikovRing$ is the $\Integers$-graded ring $\DirectSum_{k \in \Integers} \NovikovRing^k$.
        
        The \define{degree-$k$ Floer cochains} are the formal sums
        \begin{equation*}
            \sum_{\substack{
                \WithFilling{x} \in \FilledHamOrbits(\Ham) \\
                \ConleyZehnderIndex(\WithFilling{x}) = k
            }} \alpha_\WithFilling{x} \, \WithFilling{x}
        \end{equation*}
        with integer coefficients such that the set $
            \SetCondition{\WithFilling{x} \in \FilledHamOrbits(\Ham)}{\SymplecticForm(\WithFilling{x}) \le c \CommaSpace \alpha_\WithFilling{x} \ne 0}
        $ is finite for all $c \in \RealNumbers$.
        Denote by $\FloerC^k(\Manifold; \Ham, \FamilyACS)$ the set of degree-$k$ Floer cochains.
        The \define{Floer cochain complex $\FloerC^\ArbitraryIndex(\Manifold; \Ham, \FamilyACS)$ associated to $\FamilyACS$ and $\Ham$} is the finitely-generated free $\Integers$-graded $\NovikovRing$-module $\DirectSum_{k \in \Integers} \FloerC^k(\Manifold; \Ham, \FamilyACS)$.
        The $\NovikovRing$-module structure is induced by the inverse of the $\InterestingSpheres$-action on $\FilledHamOrbits(\Ham)$, giving $\NovVariable^A \cdot [x, u] = [x, (-A) \ConnectedSum u]$.
        
        The Floer cochain differential $d : \FloerC^\ArbitraryIndex(\Manifold; \Ham, \FamilyACS) \to \FloerC^{\ArbitraryIndex + 1}(\Manifold; \Ham, \FamilyACS)$ is the degree-1 $\NovikovRing$-module endomorphism given by\footnote{
            A map of the form \eqref{eqn:floer-differential} \define{counts} the moduli spaces $\QuotientedModuliSpace(\WithFilling{x}^-, \WithFilling{x}^+)$.
        }
        \begin{equation}
            \label{eqn:floer-differential}
            d(\WithFilling{x}^+)  = \sum_{\substack{
                \WithFilling{x}^- \in \FilledHamOrbits(\Ham) \\ 
                \ConleyZehnderIndex(\WithFilling{x}^-) - \ConleyZehnderIndex(\WithFilling{x}^+) = 1
            }} \sum_{
                [u] \in \QuotientedModuliSpace(\WithFilling{x}^-, \WithFilling{x}^+)
            } \Orientation([u]) \, \WithFilling{x}^- \EndComma
        \end{equation}
        where $\Orientation([u]) \in \Set{\pm 1}$ denotes the orientation\footnote{
            Recall an orientation of a 0-dimensional manifold is a choice of $\pm 1$ for each point of the manifold.
        } of the point $[u] \in \QuotientedModuliSpace(\WithFilling{x}^-, \WithFilling{x}^+)$ induced by the coherent orientation $\Orientation$.
        
        \begin{lemma}
            \label{lem:differential-squares-to-zero}
            The differential $d$ satisfies $d^2 = 0$.
        \end{lemma}
        
        The \define{Floer cohomology $\FloerCohomology^\ArbitraryIndex(\Manifold; \Ham, \FamilyACS)$ of $\Manifold$ with Floer data $(\Ham, \FamilyACS)$} is the cohomology of $(\FloerC^\ArbitraryIndex(\Manifold; \Ham, \FamilyACS), d)$.
        It is a $\Integers$-graded $\NovikovRing$-module.
        
        \begin{remark}
            \label{rem:invariance-of-floer-cohomology}
            Floer cohomology only depends on the slope of the Hamiltonian in the Floer data, and is otherwise independent of the choice of Floer data.
            This follows from the fact that, given two choices of Floer data whose Hamiltonians have the same slope, any monotone homotopy between them induces a continuation map which is an isomorphism (continuation maps are defined in \autoref{sec:contination-maps}).
            Floer cohomology is also independent of the coherent orientation $\Orientation$.
            For closed manifolds $\Manifold$, there is a $\NovikovRing$-module isomorphism $\FloerCohomology^\ArbitraryIndex(\Manifold; \Ham, \FamilyACS) \cong \NovikovRing \Tensor \Cohomology^\ArbitraryIndex (\Manifold)$ (the PSS maps of \eqref{eqn:pss-map-domain-codomain} are isomorphisms).
        \end{remark}
            
    \subsubsection{Continuation maps}
    
        \label{sec:contination-maps}
        
        Given any two regular choices of Floer data $(\Ham^-, \FamilyACS^-)$ and $(\Ham^+, \FamilyACS^+)$, a \define{homotopy} between them is a pair $\left(\Ham_{s,t}, \FamilyACS_{s,t}\right)$, where $\Ham_{s,t}$ is a $\RealNumbers \times \Circle$-dependent Hamiltonian and $\FamilyACS_{s,t}$ is a convex $\RealNumbers \times \Circle$-dependent ($\SymplecticForm$-compatible) almost complex structure, such that both are each $s$-dependent only on a compact region of $\RealNumbers \times \Circle$ and, respectively, equal $\Ham^\pm$ and $\FamilyACS^\pm$ for $\pm s \gg 0$.
        The homotopy is \define{monotone} if, at infinity, $\Ham_{s, t} = h_s(R)$ and $\partial_s h_s^\prime (R) \le 0$.
        A monotone homotopy exists only if $\Ham^\pm$ have slopes $\slope^\pm$ that satisfy $\slope^- \ge \slope^+$.
        
        Given any monotone homotopy\footnote{
            The space of such regular monotone homotopies is nonempty whenever $\Ham^\pm$ have slopes $\slope^\pm$ that satisfy $\slope^- \ge \slope^+$.
        } $(\Ham_{s, t}, \FamilyACS_{s, t})$ which is suitably regular, the \define{continuation map $\ContinuationMap : \FloerCohomology^\ArbitraryIndex (\Manifold; \Ham^+, \FamilyACS^+) \to \FloerCohomology^\ArbitraryIndex (\Manifold; \Ham^-, \FamilyACS^-)$} is defined by
        \begin{equation}
            \label{eqn:continuation-map}
            \ContinuationMap(\WithFilling{x}^+)  = \sum_{\substack{
                \WithFilling{x}^- \in \FilledHamOrbits(\Ham^-) \\ 
                \ConleyZehnderIndex(\WithFilling{x}^-) = \ConleyZehnderIndex(\WithFilling{x}^+)
            }} \sum_{
                u \in \ModuliSpace(\WithFilling{x}^-, \WithFilling{x}^+)
            } \Orientation(u) \, \WithFilling{x}^- \EndComma
        \end{equation}
        where $\ModuliSpace(\WithFilling{x}^-, \WithFilling{x}^+)$ is the moduli space of solutions to a parameterised version of the Floer equation \eqref{eqn:floer-equation} for the homotopy.
        Continuation maps are independent of the choice of homotopy.
        Moreover, the composition of two continuation maps is itself a continuation map.
        
        \definition{
            \label{def:symplectic-cohomology}
            The \define{symplectic cohomology $\SymplecticCohomology^\ArbitraryIndex(\Manifold)$ of $\Manifold$} is the direct limit 
            \begin{equation}
                \lim_{\longrightarrow} \FloerCohomology^\ArbitraryIndex(\Manifold; \Ham, \FamilyACS)\EndComma
            \end{equation}
            where the limit is over all choices of regular Floer data $(\Ham, \FamilyACS)$ ordered by slope, and the maps between the Floer cohomologies are the continuation maps.
        }
        
        \begin{remark}
            \label{rem:dependence-of-symplectic-cohomology-on-convex-end}
            The symplectic cohomology $\SymplecticCohomology^\ArbitraryIndex(\Manifold)$ a priori depends on the parameterisation of the convex end of $\Manifold$.
            Perturbations of this parameterisation do not affect $\SymplecticCohomology^\ArbitraryIndex(\Manifold)$ (see \cite[Theorem~1.9]{benedetti_invariance_2019}).
            With our convention in \autoref{rem:closed-manifold-is-trivially-convex}, symplectic cohomology is isomorphic to Floer cohomology for closed manifolds.
        \end{remark}
            
    \subsection{Hamiltonian circle actions}
    
        \label{sec:hamiltonian-circle-actions}
        
        A \define{Hamiltonian circle action} on $\Manifold$ is a smooth circle action $\CircleAction : \Circle \times \Manifold \to \Manifold$ which flows along the Hamiltonian vector field of some Hamiltonian function $\CircleHam : \Manifold \to \RealNumbers$.
        Such a circle action automatically preserves the symplectic structure.
        The action is \define{linear} if $\CircleHam$ is linear.
        \label{def:linear-ham-circle-action}
        
        The \define{vector field $\CircleVectorField{\CircleAction}$} of a circle action $\CircleAction$ is the vector field along which the action flows; it is given by $\CircleVectorField{\CircleAction} = \partial_t \CircleAction_t \EvaluateAt{t=0}$.
        Thus for our \emph{Hamiltonian} circle action $\CircleAction$, the vector field $\CircleVectorField{\CircleAction}$ equals the Hamiltonian vector field $\HamiltonianVectorField{\CircleHam}$ of $\CircleHam$.
        
        \begin{lemma}
            \label{lem:reeb-periods-discrete}
            \label{thm:circle-action-has-fixed-point}
            Let $\CircleAction$ be a linear Hamiltonian circle action of nonzero slope $\CircleSlope$.
            The positive Reeb periods form a discrete subset of $\IntervalOpen{0}{\Infinity}$.
        \end{lemma}
        
        \begin{proof}
            Without loss of generality, suppose $\CircleSlope > 0$ by using the action $\CircleAction_t\Inverse$ if necessary.
            Recall that, by \eqref{eqn:hamiltonian-vector-field-at-infinity-is-multiple-of-reeb-vector-field}, the vector field of a linear Hamiltonian of slope $\CircleSlope$ is $\CircleSlope \Reebvf$ at infinity, where $\Reebvf$ is the Reeb vector field on $(\ContactManifold, \ContactForm)$.
            The time-1 flow along this vector field is the identity map.
            Therefore every point of $\ContactManifold$ is on a closed Reeb orbit of period $\CircleSlope$.
            
            For each $y \in \ContactManifold$, set $P_y$ to be the minimal positive period of the flow along $\Reebvf$ starting at $y$.
            It is immediate that $P_y$ divides $\CircleSlope$.
            Set $l_y = \CircleSlope / P_y$.
            There is no sequence $y_r \in \ContactManifold$ such that $l_{y_r} \to \Infinity$ as $r \to \Infinity$.
            This follows by an application of the Arzel\`{a}-Ascoli theorem and the nonvanishing of the Reeb vector field.
            Thus, the positive Reeb periods form a subset of
            \begin{equation*}
                \frac{\CircleSlope}{\Factorial{\left(\max l_y\right)}} \cdot \Integers^{>0} \EndComma
            \end{equation*}
            which is a discrete set as required.
        \end{proof}
        
        \begin{remark}
            A contact manifold is \define{Besse} if every point is on a closed Reeb orbit (see \cite{cristofaro-gardiner_action_2019}).
            The above proof shows that $\ContactManifold$ is Besse for any convex symplectic manifold which admits a linear Hamiltonian circle action of nonzero slope.
        \end{remark}
        
        A Hamiltonian circle action $\CircleAction$ induces a diffeomorphism on the free loop space of $\Manifold$ which is given by
        \begin{equation}
            \label{eqn:circle-action-induces-map-on-curves}
            ( \CircleAction (x) ) (t) = \CircleAction_t \cdot x(t)
        \end{equation}
        for any loop $x : \Circle \to \Manifold$.
        
        \begin{lemma}
            \label{lem:action-preserves-contractibility-of-loops}
            Let $\CircleAction$ be a linear Hamiltonian circle action on a convex symplectic manifold $\Manifold$.
            The action $\CircleAction$ has a fixed point, and hence the map \eqref{eqn:circle-action-induces-map-on-curves} takes contractible loops to contractible loops.
        \end{lemma}
        
        \begin{proof}
            Let $\CircleHam$ have slope $\CircleSlope$.
            If $\CircleSlope \ge 0$, then the function $\CircleHam : \Manifold \to \RealNumbers$ has a minumum, so that it has at least one critical point.
            If $\CircleSlope < 0$, then $\CircleHam$ has a maximum and hence a critical point.
            Critical points of $\CircleHam$ are fixed points of $\CircleAction$, so $\CircleAction$ has a fixed point.
            
            Let $\ManifoldElement_0 \in \Manifold$ be a fixed point of $\CircleAction$.
            Let $x : \Circle \to \Manifold$ be any contractible loop.
            Let $u : \Disc \to \Manifold$ be a smooth filling of $x$ such that $u(z) = \ManifoldElement_0$ for all $\Modulus{z} < \frac{1}{2}$.
            Such a filling exists because $\Manifold$ is connected.
            Define $\CircleAction \cdot u : \Disc \to \Manifold$ by
            \begin{equation}
                \label{eqn:action-on-special-fillings}
                (\CircleAction \cdot u) (r \ExponentialNumber^{2 \PiNumber \ImaginaryNumber t}) = \CircleAction_t \cdot u(r \ExponentialNumber^{2 \PiNumber \ImaginaryNumber t})
            \end{equation}
            for all $t \in \Circle$ and $r \in \IntervalClosed{0}{1}$.
            Since $u$ is constantly a fixed point of $\CircleAction$ in a neighbourhood of $0 \in \Disc$, the map $\CircleAction \cdot u$ is well-defined.
            The map $\CircleAction \cdot u$ is a filling of $\CircleAction \cdot x$.
        \end{proof}
        
        \remark{
            For convex manifolds, the linearity hypothesis is vital.
            The Hamiltonian action on $\CotangentSpace \Circle$ induced by rotation of $\Circle$ is not linear, and the induced action on loops does not preserve contractibility.
            Seidel proved a more general result for closed manifolds which applies to loops in $\HamiltonianSymplectomorphismGroup(\Manifold)$ based at $\Identity_\Manifold$ \cite[Lemma~2.2]{seidel_$_1997}.
            His proof used the Arnold conjecture for closed manifolds.
            For convex manifolds, Ritter observed (in a Technical Remark \cite[page~14]{ritter_circle-actions_2016}) that if the map \eqref{eqn:circle-action-induces-map-on-curves} did not preserve contractibility, then symplectic cohomology vanished as did the quantum Seidel map of \autoref{sec:quantum-seidel-map-definition}, which renders this case uninteresting from the point of view of the Seidel map, so it was discarded.
        }
        
        \autoref{lem:action-preserves-contractibility-of-loops} means \eqref{eqn:circle-action-induces-map-on-curves} restricts to a map $\CircleAction : \ContractibleLoopSpace \Manifold \to \ContractibleLoopSpace \Manifold$.
        This map may be lifted to the cover $\ContractibleLoopSpaceWithFillings{\Manifold}$ by the argument of \cite[Lemma~2.4]{seidel_$_1997}.
        Denote the choice of a lift of $\CircleAction$ by $\Lifted{\CircleAction}$.
        
        \begin{definition}
            [Maslov index of $\Lifted{\CircleAction}$]
            Given a lift $\Lifted{\CircleAction}$, and a point $(x, u) \in \ContractibleLoopSpaceWithFillings{\Manifold}$, let $\Lifted{\CircleAction} (x, u) = (\CircleAction x, v)$.
            Let $\tau_x : (x \PullBack \TangentSpace \Manifold, x \PullBack \SymplecticForm) \to (\RealNumbers^{2n}, \Omega)$ be the restriction of a trivialisation of $(\TangentSpace\Manifold, \SymplecticForm)$ on $u$ and let $\tau_{\CircleAction x}$ be the restriction of a trivialisation on $v$.
            Here, $\Omega$ is the standard symplectic bilinear form on $\RealNumbers^{2n}$.
            Define the loop of symplectic matrices $l(t)$ by
            \begin{equation}
                \label{eqn:loop-of-symplectic-matrices-associated-to-lift-of-hamiltonian-circle-action}
                l(t) = \tau_{\CircleAction x} (t) \Derivative \CircleAction_t (x(t)) \tau_{x} (t) \Inverse \EndFullStop
            \end{equation}
            The Maslov index $\MaslovIndex(\Lifted{\CircleAction})$ associated to this loop does not depend on the choice of the point $(x, u)$ or on the choice of trivialisations, but it does depend on the choice of lift $\Lifted{\CircleAction}$ of $\CircleAction$.
        \end{definition}
        
    \subsection{Floer Seidel map}
        \label{sec:floer-seidel-map}
        
        Let $\Lifted{\CircleAction}$ be a lift of a linear Hamiltonian circle action of slope $\CircleSlope$.
        In \cite{seidel_$_1997}, Seidel defined a natural automorphism on Floer cohomology associated to $\Lifted{\CircleAction}$ for closed symplectic manifolds, which was extended to convex symplectic manifolds in \cite{ritter_floer_2014}.
        
        Let $(\Ham, \FamilyACS)$ be a regular choice of Floer data.
        The \define{pullback\footnote{
            This definition conforms to the conventions for pullbacks of tensor fields.
        } $\CircleAction \PullBack \FamilyACS$} of $\FamilyACS$ by $\CircleAction$ is
        \begin{equation}
            \label{eqn:pullback-almost-complex-structure}
            (\CircleAction \PullBack \FamilyACS)_t = \left( \Derivative \CircleAction_t \right)\Inverse \FamilyACS_t \Derivative \CircleAction_t
        \end{equation}
        and the \define{pullback $\CircleAction \PullBack \Ham$} of $\Ham$ by $\CircleAction$ is given by
        \begin{equation}
            \label{eqn:pullback-hamiltonian}
            (\CircleAction \PullBack \Ham)_t (\ManifoldElement) = \Ham_t(\CircleAction_t(\ManifoldElement)) - \CircleHam(\CircleAction_t(\ManifoldElement))
        \end{equation}
        for all $\ManifoldElement \in \Manifold$.
        By the elementary calculations in \cite[Section~1.4]{polterovich_geometry_2001}, the Hamiltonian flows satisfy $\HamiltonianFlow_{\CircleAction \PullBack \Ham}^t = \CircleAction_t \Inverse \HamiltonianFlow_{\Ham} ^t$.
        The \define{pullback Floer data} $(\CircleAction \PullBack \Ham, \CircleAction \PullBack \FamilyACS)$ is a regular choice of Floer data, with $\CircleAction \PullBack \Ham$ of slope $\slope - \CircleSlope$ if $\Ham$ is of slope $\slope$.
        
        The map \eqref{eqn:circle-action-induces-map-on-curves} induces isomorphisms between the moduli spaces of Floer solutions of the pullback Floer data $(\CircleAction \PullBack \Ham, \CircleAction \PullBack \FamilyACS)$ and of $(\Ham, \FamilyACS)$.
        This isomorphism is orientation-preserving with respect to the orientations induced by $\Orientation$.
        
        These pullback constructions yield a degree-$2 \MaslovIndex(\Lifted{\CircleAction})$ $\NovikovRing$-module isomorphism $\FloerSeidel_{\Lifted{\CircleAction}}$ on Floer cohomology.
        This map
        \begin{align}
            \label{eqn:floer-seidel-map}
            \FloerSeidel_{\Lifted{\CircleAction}} : \FloerC^\ArbitraryIndex(M; \FamilyACS, \Ham,) \to
            \FloerC^{\ArbitraryIndex + 2 \MaslovIndex(\Lifted{\CircleAction})}(M; \CircleAction \PullBack \FamilyACS, \CircleAction \PullBack \Ham)
        \end{align}
        is given by the $\NovikovRing$-linear extension of $\WithFilling{x} \mapsto \Lifted{\CircleAction}\PullBack \cdot \WithFilling{x}$, where $\Lifted{\CircleAction}\PullBack \cdot \WithFilling{x}$ is the preimage of $\WithFilling{x}$ under the map $\Lifted{\CircleAction}$.
        By pulling back regular monotone homotopies, it is possible to establish that $\FloerSeidel_{\Lifted{\CircleAction}}$ commutes with continuation maps, so that
        \begin{equation}
            \label{eqn:seidel-map-commutes-with-continuation}
            \begin{tikzcd}
            \FloerCohomology^\ArbitraryIndex(M; \FamilyACS^+, \Ham^+)
            \arrow[d, "\ContinuationMap"'] 
            \arrow[r, "\FloerSeidel_{\Lifted{\CircleAction}}"] 
            & \FloerCohomology^{\ArbitraryIndex + 2 \MaslovIndex(\Lifted{\CircleAction})}(M; \CircleAction \PullBack \FamilyACS^+, \CircleAction \PullBack \Ham^+)
            \arrow[d, "\Lifted{\CircleAction} \PullBack \ContinuationMap"] 
            \\
            \FloerCohomology^\ArbitraryIndex(M; \FamilyACS^-, \Ham^-) 
            \arrow[r, "\FloerSeidel_{\Lifted{\CircleAction}}"']           
            & \FloerCohomology^{\ArbitraryIndex + 2 \MaslovIndex(\Lifted{\CircleAction})}(M; \CircleAction \PullBack \FamilyACS^-, \CircleAction \PullBack \Ham^-)
            \end{tikzcd}
        \end{equation}
        commutes.
        Hence $\FloerSeidel_{\Lifted{\CircleAction}}$ induces a degree-$2 \MaslovIndex(\Lifted{\CircleAction})$ automorphism of symplectic cohomology.
        
        Seidel showed that $\FloerSeidel_{\Lifted{\CircleAction}}$ respects the module structure induced by the pair-of-pants product for closed manifolds \cite[Proposition~6.3]{seidel_$_1997}, and this was extended to convex manifolds by Ritter \cite[Theorem~23]{ritter_floer_2014}.
        The pair-of-pants product does not extend to the equivariant setup, so there is no analogue for this result.
            
\section{Equivariant Floer theory}

    \label{sec:equivariant-floer-theory}

    In this section, we define equivariant Floer cohomology.
    It should be considered in analogy to the Borel homotopy-quotient model for the equivariant cohomology of a topological space (see \autoref{sec:equivariant-cohomology}).
    We outline in \autoref{sec:infinite-sphere-conventions} our conventions for the universal bundle of $\Circle$, and proceed to explain our definition in \autoref{sec:equivariant-floer-cohomology-all}.
    
    In the literature, the $\Circle$-action used in the definition of $\Circle$-equivariant Floer cohomology is the action which rotates the domains of loops.
    In this paper, we incorporate an additional $\Circle$-action on $\Manifold$ into the definition.
    Using the trivial action on $\Manifold$ in our definition recovers the usual definition, except we use a non-standard and more geometric construction for the $\Integers [\uformal]$-module action (see \autoref{sec:y-shape-module-structure-on-equivariant-floer}).
    For completeness, we give the standard construction of the  $\Integers [\uformal]$-module action in \autoref{sec:u-multiplication-shift-invariance}, though we do not use it in this paper.
    
    Our definition is based on the standard definition of the Borel homotopy-quotient rather than the variant which is common in the equivariant symplectic cohomology literature (see \autoref{rem:diagonal-action-in-literature-for-equivariant-floer-theory}).
    
    Aside from these three key differences, our construction strongly resembles those already in the literature.
    In the following remark, we compare our conventions to those in other papers, however all these remaining differences are cosmetic.
    
    \begin{remark}
        [Conventions]
        Our definition is close to those of \cite[Sections~2.2--2.3]{bourgeois_s1-equivariant_2017} and \cite[Section~2.3]{gutt_lecture_2018}, except we use cohomological conventions and a Novikov ring.
        Our conventions are almost identical to Zhao's \emph{periodic symplectic cohomology} in \cite{zhao_periodic_2014, zhao_periodic_2016} and Seidel's definition in \cite{seidel_connections_2016}, except that we use a direct sum convention for cohomology rather than a direct product (and we use a Novikov ring).
        Importantly, we do not use $\Integers [\uformal, \uformal \Inverse] / (\uformal \Integers [\uformal] )$ in our coefficient ring unlike \cite[Section~8b]{seidel_biased_2007} and \cite[Appendix~B]{mclean_mckay_2018}; indeed this is not possible with our module operation.
    \end{remark}
        
    \subsection{Infinite sphere}
    
        \label{sec:infinite-sphere-conventions}
        
        The group $\Circle$ acts freely on the odd-dimensional sphere $\HighDimensionalSphere{2k-1}$, which we consider as the subset of $\ComplexNumbers^k$ of norm 1, by multiplication 
        \begin{equation}
            \label{eqn:circle-action-on-infinite-sphere}
            \CircleElement \cdot \infsphereelement = \ExponentialNumber^{2 \PiNumber \ImaginaryNumber \CircleElement} \infsphereelement\EndFullStop
        \end{equation}
        With this action, the manifold $\HighDimensionalSphere{2k-1}$ is a principal $\Circle$-bundle over $\ComplexNumbers \Projective^{k-1}$.
        These spheres are equipped with canonical inclusion maps $\sphereinclusion_k : \HighDimensionalSphere{2k-1} \xhookrightarrow{} \HighDimensionalSphere{2k+1}$.
        Denote by $\InfiniteSphere$ the direct limit of the odd-dimensional spheres under these inclusion maps.
        
        \begin{remark}
            \label{rem:use-of-direct-limit-and-infinite-sphere-notation}
            The topological space $\InfiniteSphere$ is actually not very nice.
            It is not compact, it is not first countable\footnote{
                Recall a topological space $X$ is \define{first countable} if, for every point $x \in X$, there is a sequence $(U_k)_{k \in \NaturalNumbers}$ of open subsets of $X$ which contain $x$ such that, for every open subset $U \subseteq X$ with $x \in U$, there is an inclusion $U_k \subseteq U$ for some $k \in \NaturalNumbers$.
                All metric spaces are first countable.
            } and in particular it is not metrisable.
            We use the notation of $\InfiniteSphere$ to simplify otherwise cumbersome statements which involve the limit of spheres.
            For example, by a \define{smooth} map $f : \InfiniteSphere \to \RealNumbers$, we mean a sequence of smooth maps $f_k : \HighDimensionalSphere{2k-1} \to \RealNumbers$, which are compatible with the inclusions in the sense that $f_{k+1} \ComposedWith \sphereinclusion_k = f_k$ and $\lim f_k = f$.
            In this way, the space $\InfiniteSphere$ is a principal $\Circle$-bundle, with projection map $\sphereprojection : \InfiniteSphere \to \InfiniteComplexProjectiveSpace$.
        \end{remark}
        
        Define the smooth function $\spheremorsefunction : \InfiniteSphere \to \RealNumbers$ by $(\infsphereelement_0, \ldots) \mapsto \sum_k k \Modulus{\infsphereelement_k}^2 $.
        The function $\spheremorsefunction$ descends to a Morse-Smale function on $\InfiniteComplexProjectiveSpace$ whose unique critical point of index $2k$ is the standard basis vector $\spherecriticalpoint_k = [0:\cdots:0:1:0:\cdots]$ for all $k \ge 0$.
        Recall that the cohomology of $\InfiniteComplexProjectiveSpace$ is isomorphic to $\Integers[\uformal]$, with $\uformal$ a formal variable of degree 2, where the critical point $\spherecriticalpoint_k$ corresponds to $\uformal^k$.
        
        The space $\InfiniteSphere$ is equipped with the round metric.
        For each $k$, fix an identification $\Circle \leftrightarrow \spherecriticalpoint_k \subset \InfiniteSphere$.
        Extend this identification to the unstable and stable manifolds to get an equivariant map 
        \begin{equation}
            \label{eqn:trivialisations-of-unstable-and-stable-manifolds-in-infinite-sphere}
            \spheretrivialise_k : \UnstableManifold(\spherecriticalpoint_k) \Union \StableManifold(\spherecriticalpoint_k) \to \Circle \EndComma
        \end{equation}
        defined by where the negative gradient flowline of $\spheremorsefunction$ converges to along $\spherecriticalpoint_k$.
        
        The right-shift map $\ComplexNumbers^k \to \ComplexNumbers^{k+1}$ given by $(\infsphereelement_0, \ldots , \infsphereelement_{k-1}) \mapsto (0, \infsphereelement_0, \ldots , \infsphereelement_{k-1})$ induces an injective smooth map $\sphereshift : \InfiniteSphere \to \InfiniteSphere$.
        The gradient vector field of $\spheremorsefunction$ is $\sphereshift$-invariant.
        
    \subsection{Equivariant cohomology}
    
        \label{sec:equivariant-cohomology}
        
        Let $X$ be a topological space with a continuous circle action $\rho : \Circle \times X \to X$.
        The \define{Borel homotopy quotient}, denoted $\InfiniteSphere \times_{\Circle} X$, is the quotient of the product $\InfiniteSphere \times X$ by the relation $(\CircleElement \cdot \infsphereelement, x) \sim (\infsphereelement, \CircleElement \cdot x)$.
        Equivalently, it is the quotient of the product $\InfiniteSphere \times X$ by the free circle action
        \begin{equation}
            \label{eqn:circle-action-for-borel-homotopy-quotient}
            \CircleElement \cdot (\infsphereelement, x) = (\CircleElement \Inverse \cdot \infsphereelement, \rho_\CircleElement (x)) \EndFullStop
        \end{equation}
        The \define{equivariant cohomology} of the pair $(X, \rho)$ is $\ECohomology^\ArbitraryIndex_\rho (X) = \Cohomology^\ArbitraryIndex (\InfiniteSphere \times_{\Circle} X)$.
        The projection $\InfiniteSphere \times_{\Circle} X \to \InfiniteSphere / \Circle = \InfiniteComplexProjectiveSpace$ induces a map $\Integers [\uformal] \cong \Cohomology^\ArbitraryIndex (\InfiniteComplexProjectiveSpace) \to \ECohomology^\ArbitraryIndex_\rho (X)$.
        Together with the cup product, this map gives equivariant cohomology the structure of a unital, associative and graded-commutative $\Integers [\uformal]$-algebra.
        
        \begin{remark}
            [Diagonal action in literature]
            \label{rem:diagonal-action-in-literature-for-equivariant-floer-theory}
            The literature for equivariant symplectic cohomology uses an alternative convention for the Borel homotopy quotient whereby the free diagonal action on $\InfiniteSphere \times X$ is used.
            The automorphism of $\InfiniteSphere$ given by complex conjugation takes this diagonal action back to the standard action \eqref{eqn:circle-action-for-borel-homotopy-quotient}, however it is \emph{not orientation-preserving} on the quotient $\InfiniteComplexProjectiveSpace$, so this different convention results in a different $\Integers [\uformal]$-module structure.
            To correct for this, the transformation $\uformal \mapsto - \uformal$ must be used when changing convention.
        \end{remark}
        
        Let $(X, \rho_X)$ and $(Y, \rho_Y)$ be two topological spaces with circle actions as above.
        The continuous map $f : X \to Y$ is \define{equivariant} if it intertwines the two actions, that is the identity $\rho_{Y,\CircleElement}(f(x)) = f(\rho_{X,\CircleElement}(x))$ holds.
        Such an equivariant map induces a well-defined map on the Borel homotopy quotients, and hence induces a natural \define{pullback map} $f\PullBack : \ECohomology^\ArbitraryIndex_{\rho_Y} (Y) \to \ECohomology^\ArbitraryIndex_{\rho_X} (X)$.
        
    \subsection{Equivariant Floer cohomology}
    
        \label{sec:equivariant-floer-cohomology-all}
            
    \subsubsection{Equivariant Hamiltonian orbits}
        \label{sec:linear-symplectic-circle-action}
        Let $\rho$ be a symplectic circle action on $\Manifold$, and assume that the action flows along the Hamiltonian vector field of some linear Hamiltonian $\CircleHamSpecifyAction{\rho}$ at infinity.
        Such actions are \define{linear at infinity}.
        Notice that $\rho$ preserves the radial coordinate $R$ at infinity.
        The linear Hamiltonian circle actions of \autoref{def:linear-ham-circle-action} satisfy this assumption by definition; the difference is that we do not assume the action is Hamiltonian on the entire manifold here.
        
        The loop space $\ContractibleLoopSpace \Manifold$, being a space of maps $\Circle \to \Manifold$, naturally inherits\footnote{
            \label{footnote:action-conventions-on-function-spaces}
            Let the group $G$ act on sets $X$ and $Y$ on the left.
            Let $\Map(X, Y)$ denote the space of maps from $X$ to $Y$.
            The action on $Y$ induces an action on $\Map(X, Y)$ by post-composition.
            The action on $\Map(X, Y)$ given by $g \cdot f = f \ComposedWith g$ is a right action, instead of a left action.
            The action given by $g \cdot f = f \ComposedWith g \Inverse$ is still a left action, so this is the action \define{naturally inherited} by $\Map(X, Y)$.
            Here, $g \in G$ and $f \in \Map(X, Y)$.
        } two circle actions, one from the rotation action on the domain $\Circle$ and the other from the circle action on the codomain $\Manifold$.
        We combine these to get the circle action on $\ContractibleLoopSpace \Manifold$ which is given, for all $t, \CircleElement \in \Circle$ and all $x \in \ContractibleLoopSpace \Manifold$, by
        \begin{equation}
            \label{eqn:circle-action-on-loop-space}
            \CircleElement \cdot (\ t \mapsto x(t)\ ) = (\ t \mapsto \rho_\CircleElement (x(t - \CircleElement))\ ) \EndFullStop
        \end{equation}
        
        An \define{equivariant Hamiltonian} is a smooth function $\EqntHam : \InfiniteSphere \times (\Circle \times \Manifold) \to \RealNumbers$ which satisfies
        \begin{equation}
            \label{eqn:equivariant-condition-on-hamiltonian}
            \EqntHam_{\infsphereelement, t} (\ManifoldElement) = \EqntHam_{\CircleElement \Inverse \cdot \infsphereelement, t + \CircleElement} (\rho_\CircleElement ( \ManifoldElement ))
        \end{equation}
        for all $t, \CircleElement \in \Circle$, all $\infsphereelement \in \InfiniteSphere$ and all $\ManifoldElement \in \Manifold$.
        Notice that a function satisfying \eqref{eqn:equivariant-condition-on-hamiltonian} is equivariant, in the sense of \autoref{sec:equivariant-cohomology}, with respect to the natural action on the domain, according to our convention in \eqref{eqn:circle-action-for-borel-homotopy-quotient}, and the trivial action on the codomain $\RealNumbers$.
        
        The equivariant Hamiltonian $\EqntHam$ is \define{linear of slope $\slope$} if there is $R_0$ such that the equation $\EqntHam_{\infsphereelement, t} (\ConvexCoordMap(y, R)) = \slope R$ holds when $R \ge R_0$.
        
        An \define{equivariant Hamiltonian orbit} is an equivalence class $[\infsphereelement, x] \in \InfiniteSphere \times_\Circle \ContractibleLoopSpace \Manifold$ such that $\infsphereelement$ is a critical point of $\spheremorsefunction : \InfiniteSphere \to \RealNumbers$ and $x$ is a Hamiltonian orbit of $\EqntHam_{\infsphereelement, t} (\Argument)$.
        The equivariance of the Hamiltonian $\EqntHam$ guarantees this definition is independent of the choice of representative $(\infsphereelement, x)$.
        
        The action \eqref{eqn:circle-action-on-loop-space} on $\ContractibleLoopSpace \Manifold$ lifts canonically to an action on $\ContractibleLoopSpaceWithFillings{\Manifold}$.
        The action functional $\ActionFunctional_{\EqntHam} : \InfiniteSphere \times \ContractibleLoopSpaceWithFillings{\Manifold} \to \RealNumbers$ given by
        \begin{equation}
            \label{eqn:equivariant-action-functional}
            \ActionFunctional_{\EqntHam} (\infsphereelement, (x, u)) = - \int_\Disc u \PullBack \SymplecticForm + \int_{t = 0}^1 \EqntHam_{\infsphereelement, t} (x (t)) \wrt{t}
        \end{equation}
        is invariant under the action combining the lift of \eqref{eqn:circle-action-on-loop-space} and \eqref{eqn:circle-action-for-borel-homotopy-quotient}, much like the equivariant Morse functions of \autoref{sec:equivariant-morse-cohomology}.
            
    \subsubsection{Equivariant Floer data}
    
        An \define{equivariant\footnote{
            An almost-complex structure is a section of the bundle $\Automorphisms(\TangentSpace \Manifold) \to \Manifold$.
            Let $p : \InfiniteSphere \times \Circle \times \Manifold \to \Manifold$ be the natural projection map.
            An equivariant almost-complex structure is a map $\InfiniteSphere \times \Circle \times \Manifold \to p \PullBack \Automorphisms(\TangentSpace \Manifold)$ which is equivariant in the usual sense.
        } almost-complex structure $\EqntACS$} is a $\InfiniteSphere \times \Circle$-family of almost-complex structures $\FamilyACS\Eqnt_{\infsphereelement, t}$ which makes the diagram
        \begin{equation}
            \label{eqn:equivariant-condition-on-almost-complex-structure}
            \begin{tikzcd}[column sep=huge]
                \TangentSpace_\ManifoldElement \Manifold \arrow[d, "\Derivative \rho_\CircleElement"] \arrow[r, "\FamilyACS\Eqnt_{\infsphereelement, t}"] & \TangentSpace_\ManifoldElement\Manifold \arrow[d, "\Derivative \rho_\CircleElement"] \\
                \TangentSpace_{\rho_\CircleElement(\ManifoldElement)}\Manifold \arrow[r, "\FamilyACS\Eqnt_{ \CircleElement\Inverse \cdot \infsphereelement, t + \CircleElement}"]        & \TangentSpace_{\rho_\CircleElement (\ManifoldElement)} \Manifold     
            \end{tikzcd}
        \end{equation}
        commute for all $\ManifoldElement \in \Manifold$.
        
        \begin{definition}
            \label{def:extending-floer-data-equivariantly}
            The equivariant Hamiltonian $\EqntHam$ \define{extends the sequence of (non-equivariant) Hamiltonians $\Ham^k$} if
            \begin{equation}
                \label{eqn:local-behaviour-of-equivariant-hamiltonians}
                \Ham\Eqnt_{\infsphereelement, t} (\ManifoldElement) = \Ham^k_{t + \spheretrivialise_k(\infsphereelement)}( \rho_{\spheretrivialise_k(\infsphereelement)} (\ManifoldElement) )
            \end{equation}
            in a neighbourhood of $\spherecriticalpoint_k$ in $\UnstableManifold(\spherecriticalpoint_k) \Union \StableManifold(\spherecriticalpoint_k)$ for all $k \ge 0$.
            Likewise, the equivariant almost complex structure $\FamilyACS\Eqnt$ \define{extends the sequence of almost complex structures $\FamilyACS^k$} if
            \begin{equation}
                \label{eqn:local-behaviour-of-equivariant-acs}
                \FamilyACS\Eqnt_{t, \infsphereelement} = \Derivative \rho_{\spheretrivialise_k(\infsphereelement)} \Inverse\ \FamilyACS^k_{t + \spheretrivialise_k(\infsphereelement)} \ \Derivative \rho_{\spheretrivialise_k(\infsphereelement)}
            \end{equation}
            in a neighbourhood of $\spherecriticalpoint_k$ in $\UnstableManifold(\spherecriticalpoint_k) \Union \StableManifold(\spherecriticalpoint_k)$ for all $k \ge 0$.
        \end{definition}
        
        For such equivariant data, the equivariant Hamiltonian orbits are equivalence classes $[\infsphereelement, x]$ where $\infsphereelement \in \spherecriticalpoint_k$ satisfies $\spheretrivialise_k(\infsphereelement) = 0$ and $x \in \HamOrbits(\Ham^k)$.
        We use the shorthand $(\spherecriticalpoint_k, x)$ for such equivariant Hamiltonian orbits.
        
        \begin{definition}
            A choice of \define{equivariant Floer data} is a pair $(\EqntHam, \EqntACS)$ consisting of a linear equivariant Hamiltonian function $\EqntHam$ and a convex equivariant $\SymplecticForm$-compatible almost complex structure $\EqntACS$ which together extend a sequence of Floer data.
        \end{definition}
        
        \remark{
            Typically, one chooses the equivariant Floer data so that it is the same at the critical points, however we relax this requirement here, allowing any sequence of Floer data.
            We need \eqref{eqn:local-behaviour-of-equivariant-hamiltonians} and \eqref{eqn:local-behaviour-of-equivariant-acs} to hold for a sequence of non-equivariant data $(\Ham^k, \FamilyACS^k)_{k \ge 0}$ in order to apply the standard continuation map techniques that guarantee the desired moduli space behaviour.
        }
        
        \begin{proposition}
            [Existence of data]
            Let $(\Ham^k, \FamilyACS^k)$ be any sequence of Floer data whose Hamiltonians are all linear of the same slope $\slope$ and whose \emph{at infinity conditions\footnote{That is, the linearity of the Hamiltonians and the convexity of the almost complex structures.}} are all satisfied in a common region at infinity.
            There is equivariant Floer data which extends this sequence, and moreover the space of such data is contractible.
        \end{proposition}
        \begin{proof}
            We can construct an equivariant Hamiltonian by an analogous argument to \cite[Example~2.4]{bourgeois_s1-equivariant_2017}.
            First, we fix an invariant time-independent Hamiltonian $\Ham^\slope$ on $\Manifold$ which is linear at infinity of slope $\slope$.
            To do this, simply set $\Ham^\slope = 0$ on $\Set{R < R_0}$ and $\Ham_\slope(\ConvexCoordMap(y, R)) = h(R)$ on $\Set{R \ge R_0}$ for an appropriate function $h$.
            Next, use a cut-off function near each $\spherecriticalpoint_k$ to interpolate between each $\Ham^k$ (appropriately interpreted via \eqref{eqn:local-behaviour-of-equivariant-hamiltonians}) and the fixed $\Ham^\slope$.
            The result is an equivariant Hamiltonian of slope $\slope$ which extends the sequence as desired.
            This shows the desired existence.
            The space of all such equivariant Hamiltonians is convex and hence contractible.
            
            For the equivariant almost complex structure, consider the symplectic vector bundle
            \begin{equation}
                \label{eqn:equivariant-cotangent-space-as-symplectic-vector-bundle}
                \begin{tikzcd}
                    E \arrow[d]
                    = \nicefrac{\InfiniteSphere \times \Circle \times \CotangentSpace \Manifold}{\Circle}
                    \\
                    B = \nicefrac{\InfiniteSphere \times \Circle \times \Manifold}{\Circle} \EndFullStop
                \end{tikzcd}
            \end{equation}
            The space of (compatible) almost complex structures on $E$ is nonempty and contractible \cite[Proposition~2.63]{mcduff_introduction_1998}.
            The proof of this constructs a retraction $r$ from the space of inner products on $E$ (which is convex and hence contractible) to the space of compatible almost complex structures on $E$.
            By restricting this map $r$ to an appropriate subspace of inner products, we will get a retraction to the space of convex almost complex structures which extend the given sequence $\FamilyACS^k$.
            It is sufficient for the subspace to be nonempty and contractible to complete the proof.
            
            To guarantee convexity, we restrict to inner products $g$ which satisfy
            \begin{equation}
                \label{eqn:conditions-on-inner-products-to-force-convexity}
                \left\{
                        \begin{array}{l}
                            g(\partial_R, R \partial_R) = 1 \\
                            g(\partial_R, \Reebvf) = 0 \\
                            g(\Reebvf, \Reebvf) = R \\
                            \text{$\Kernel(\ContactForm)$ is $g$-orthogonal to $\partial_R$ and $\Reebvf$}
                        \end{array}
                    \right.
            \end{equation}
            at infinity.
            To ensure the almost complex structures will extend $\FamilyACS^k$, we further restrict to inner products which take appropriate fixed values over neighbourhoods of $\spherecriticalpoint_k \times \Circle \times \Manifold / \Circle$ for each $k$.
            The space of inner products which satisfy these two conditions remains convex and nonempty, as desired.
        \end{proof}
            
    \subsubsection{Equivariant Floer solutions}
    
        The map $(v, u) : \RealNumbers \to \InfiniteSphere \times \ContractibleLoopSpace \Manifold$ \define{satisfies the equivariant Floer equation} if $v$ is a negative gradient flowline of $\spheremorsefunction : \InfiniteSphere \to \RealNumbers$ and $u$ satisfies the equation
        \begin{equation}
            \label{eqn:equivariant-floer-equation}
            \partial_s u + \FamilyACS\Eqnt_{v(s), t} \left(\partial_t u - \HamiltonianVectorField{H_{v(s), t}(\Argument)}\right) = 0 \EndFullStop
        \end{equation}
        An \define{equivariant Floer solution} is an equivalence class $[v, u] \in \nicefrac{\SmoothMaps(\RealNumbers, \InfiniteSphere \times \ContractibleLoopSpace \Manifold)}{\Circle}$.
        The equivariant Floer solution $[v, u]$ \define{converges to equivariant Hamiltonian orbits $(\spherecriticalpoint_{k^\pm}, x^\pm)$} if the limits $\lim_{s \to \pm \Infinity} v(s) = \infsphereelement ^ \pm$ and \eqref{eqn:Floer-solution-convergence-to-limits} hold for some choices of representatives.
        
        For regular equivariant Floer data, the moduli space of equivariant Floer solutions $\ModuliSpace((\spherecriticalpoint_{k^-}, \WithFilling{x}^-), (\spherecriticalpoint_{k^+}, \WithFilling{x}^+))$ is a smooth oriented manifold of dimension $2k^- - 2k^+ + \ConleyZehnderIndex(\WithFilling{x}^-) - \ConleyZehnderIndex(\WithFilling{x}^-)$.
        The moduli space is empty unless $2k^- - 2k^+ \ge 0$.
        It admits a smooth $\RealNumbers$-action via $s$-translation that is free if the equivariant Hamiltonian orbits are distinct.
        (The case when $k^- = k^+$ is canonically identical to the non-equivariant case using the maps $\spheretrivialise_{k^+}$.)
        The dimension-0 moduli spaces are compact and the dimension-1 moduli spaces admit a compactification via gluing broken trajectories as per \autoref{sec:moduli-space-of-floer-trajectories-properties}.
        The phenomenon of bubbling is avoided using regularity conditions\footnote{
            Explicitly, the regularity conditions ensure $((w, t), x(t)) \notin \ChernBoundedSpheres_1(\EqntACS)$ for all equivariant Hamiltonian orbits $(\infsphereelement, x)$ and $((v(s), t), u(s, t)) \notin \ChernBoundedSpheres_0(\EqntACS)$ for all equivariant Floer solutions $(v, u)$ occuring in moduli spaces of dimension 1.
            These conditions are the equivariant analogues of those used by \cite{hofer_floer_1995}.
        }.
        
        \begin{remark}
            [Weak\textsuperscript{+} monotonicity]
            A symplectic manifold satisfies \define{weak\textsuperscript{+} monotonicity} if the implication
            \begin{equation}
                2-\dimM \le \FirstChernClass(A) < 0 \implies \SymplecticForm(A) \le 0
            \end{equation}
            holds for all $A \in \HomotopyGroup_2(\Manifold)$.
            Weak\textsuperscript{+} monotonicity is insufficient to exclude bubbling by standard arguments for the equivariant definitions, even though it is sufficient for the non-equivariant definitions.
            For regular time-dependent almost complex structures $\FamilyACS$, the moduli space of $\FamilyACS$-holomorphic spheres of class $A$ has dimension $2\dimM + 2 \FirstChernClass(A) + 1$.
            Therefore the moduli space is empty when $\FirstChernClass(A) \ll 0$ is large and negative, because the dimension is negative.
            In the equivariant setup, however, the moduli space of $\EqntACS$-holomorphic spheres has `dimension' $2\dimM + 2 \FirstChernClass(A) + 1 + \Dimension(\InfiniteSphere/\Circle)$.
            Since this `dimension' is positive for any $\FirstChernClass(A)$, the same argument does not apply.
            Instead, we require nonnegative monotonicity which prohibits any $\EqntACS$-holomorphic spheres with negative first Chern class via \eqref{eqn:nonnegatively-monotone-negative-spheres}.
        \end{remark}
        
    \subsubsection{Equivariant Floer cochain complex}
    
        \label{sec:equivariant-floer-cohomology-definition}
    
        Let $(\EqntHam, \EqntACS)$ be a regular choice of equivariant Floer data.
        The \define{degree-$l$ equivariant Floer cochains} are the formal sums 
        \begin{equation}
            \label{eqn:equivariant-floer-complex}
            \sum_{k \ge 0}
            \sum_{\substack{
                \WithFilling{x} \in \FilledHamOrbits(\Ham^k) \\
                \ConleyZehnderIndex(\WithFilling{x}) = l - 2k
            }} \alpha_{k, \WithFilling{x}} \, (\spherecriticalpoint_k, \WithFilling{x})
        \end{equation}
        with integer coefficients such that the set $$
            \SetCondition{\WithFilling{x} \in \FilledHamOrbits(\Ham^k)}{\SymplecticForm(\WithFilling{x}) \le c \CommaSpace \alpha_{k ,\WithFilling{x}} \ne 0}
        $$ is finite for all $c \in \RealNumbers$ and $k \ge 0$.
        Denote by $\EFloerC_{\rho}^l(\Manifold; \EqntHam)$ the $\Integers$-module of such degree-$l$ equivariant Floer cochains.
        The \define{equivariant Floer cochain complex} is the $\Integers$-graded $\NovikovRing$-module $\DirectSum_{k \in \Integers} \EFloerC_{\rho}^k (\Manifold; \EqntHam)$.
        
        The \define{equivariant Floer cochain differential} is the $\NovikovRing$-module degree-1 endomorphism $d : \EFloerC_{\rho}^\ArbitraryIndex (\Manifold; \EqntHam) \to \EFloerC_{\rho}^{\ArbitraryIndex + 1} (\Manifold; \EqntHam)$ given by
        \begin{equation}
            \label{eqn:equivariant-floer-differential}
            d(\spherecriticalpoint_{k^+}, \WithFilling{x}^+) = 
            \sum_{\substack{
                {k^-} \ge {k^+} \\
                \WithFilling{x}^- \in \FilledHamOrbits(\Ham^{k^-}) \\ 
                2{k^-} - 2{k^+} + \ConleyZehnderIndex(\WithFilling{x}^-) - \ConleyZehnderIndex(\WithFilling{x}^+) = 1
            }} \sum_{
                [v, u] \in \QuotientedModuliSpace((\spherecriticalpoint_{k^\pm}, \WithFilling{x}^\pm))
            } \Orientation([v, u]) \, (\spherecriticalpoint_{k^-}, \WithFilling{x}^-) \EndFullStop
        \end{equation}
        The differential indeed satisfies $d^2 = 0$, and the \define{equivariant Floer cohomology}, denoted $\EFloerCohomology_{\rho}^\ArbitraryIndex (\Manifold; \EqntHam)$, is the cohomology of $(\EFloerC_{\rho}^\ArbitraryIndex (\Manifold; \EqntHam), d)$. It is a $\Integers$-graded $\NovikovRing$-module.
        
        Standard homotopy techniques ensure $\EFloerCohomology_{\rho}^\ArbitraryIndex (\Manifold; \EqntHam)$ is dependent only on the slope of the Hamiltonian.
        Moreover, equivariant monotone homotopies induce equivariant continuation maps.
        The \define{equivariant symplectic cohomology $\ESymplecticCohomology_{\rho}^\ArbitraryIndex (\Manifold)$} is the direct limit of the resulting system, just as in the non-equivariant case.
        
    \subsection{Module structures}
        
    \subsubsection{Algebraic module structure}
        \label{sec:u-multiplication-shift-invariance}
        
        This section describes the $\Integers [\uformal]$-module action used in the literature.
    
        Recall the right-shift operator on $\InfiniteSphere$ denoted $\sphereshift$.
        Assume that the maps $\spheretrivialise_k$ are $\sphereshift$-compatible and that the equivariant Hamiltonian $\Ham\Eqnt$ and the equivariant almost-complex structure $\FamilyACS\Eqnt$ are $\sphereshift$-invariant.
        These assumptions yield isomorphisms
        \begin{equation*}
            \ModuliSpace((\spherecriticalpoint_{k^-}, \WithFilling{x}^-), (\spherecriticalpoint_{k^+}, \WithFilling{x}^+)) \cong \ModuliSpace((\spherecriticalpoint_{k^- + r}, \WithFilling{x}^-), (\spherecriticalpoint_{k^+ + r}, \WithFilling{x}^+))
        \end{equation*}
        for all $r \ge 0$.
        With respect to the $\Integers [\uformal]$-module structure on $\EFloerC^\ArbitraryIndex_\rho (\Manifold; \EqntHam)$ given by $\uformal \cdot (\spherecriticalpoint_{k}, \WithFilling{x}) = (\spherecriticalpoint_{k + 1}, \WithFilling{x})$, the differential $d$ is a $\NovikovRing [\uformal]$-module endomorphism.
        Hence under this assumption, the equivariant Floer cohomology is a $\NovikovRing [\uformal]$-module.
            
    \subsubsection{Geometric module structure}
        \label{sec:y-shape-module-structure-on-equivariant-floer}

        The Morse cup product counts `Y'-shaped flowlines.
        We adapt this product to get a $\Integers [\uformal]$-module structure on equivariant Floer cohomology which reflects the geometric behaviour of equivariant flowlines in $\InfiniteSphere$.
        The additional conditions placed on data for this construction are generic, in contrast to the invariance assumptions of \autoref{sec:u-multiplication-shift-invariance} which are not generic.
        
        Take an $\Circle$-invariant $s$-dependent perturbation\footnote{
            Throughout, any $s$-dependent perturbation is $s$-dependent only on a bounded interval.
        } of the function $\spheremorsefunction : \InfiniteSphere \to \RealNumbers$.
        This is a smooth function $\IntervalClosedOpen{0}{\Infinity} \to \SmoothFunctions(\InfiniteSphere)$, $s \mapsto \spheremorsefunction_s$.
        
        Given two equivariant Hamiltonian orbits $(\spherecriticalpoint_{k^-}, \WithFilling{x}^-)$ and $(\spherecriticalpoint_{k^+}, \WithFilling{x}^+)$, consider the moduli space of $\Circle$-equivalence classes of triples $(v, u, v^0)$, where $[v, u]$ is an equivariant Floer solution converging to the two equivariant Hamiltonian orbits and $v^0 : \IntervalClosedOpen{0}{\Infinity} \to \InfiniteSphere$ is a negative gradient flowline of $\spheremorsefunction_s$ which converges to a point on $\spherecriticalpoint_k$ and satisfies $v^0(0) = v(0)$.
        
        For a regular\footnote{
            In order to avoid bubbling, we moreover assume that the map
            \begin{equation}
                \label{eqn:intersection-of-spheres-and-flowlines-equivariant}
                \nicefrac{
                    \ModuliSpace(A, \FamilyACS\Eqnt \RestrictedTo{F \times \Circle})
                }{\Circle} \times_{\ProjectiveSpecialLinearGroup} \Projective^1 \times \ModuliSpace \left( (\spherecriticalpoint_{k^-}, \WithFilling{x}^-), (\spherecriticalpoint_{k^+}, \WithFilling{x}^+) \right)  \to \left( F \times_\Circle \Manifold \right) \times \left( F \times_\Circle \Manifold \right)
            \end{equation}
            given by $([(\infsphereelement, t, u_{\Projective^1}), p], (v, u)) \mapsto (\infsphereelement, u_{\Projective^1}(p), v(0), u(0, t))$ is transversal to the diagonal for all $A \in \InterestingSpheres$.
            Here, $F$ denotes the intersection $\UnstableManifold(\uformal^{k^-}) \Intersection \StableManifold(\uformal^{k^+}) \subset \InfiniteSphere$.
            The domain of the map has dimension $2k^- - 2k^+ + 2 \dimM + 2 \FirstChernClass(A) - 3 + (2k^- + \ConleyZehnderIndex(\WithFilling{x}^-) - 2k^+ - \ConleyZehnderIndex(\WithFilling{x}^+))$.
            When the moduli space of equivariant Floer solutions has dimension 1 or 2 and $\FirstChernClass(A) = 0$, the map \eqref{eqn:intersection-of-spheres-and-flowlines-equivariant} is transversal to the diagonal only if the intersection of the image with the diagonal is empty.
            This recovers one of the conditions for regular equivariant data.
            Data will generically satisfy this condition by an argument analogous to the proof of \cite[Theorem~3.2]{hofer_floer_1995}.
        } choice of the perturbed Morse function $\spheremorsefunction_s$, this moduli space is an oriented smooth manifold of dimension
        \begin{equation}
            2k^- - 2k^+ - 2k + \ConleyZehnderIndex(\WithFilling{x}^-) - \ConleyZehnderIndex(\WithFilling{x}^-) \EndFullStop
        \end{equation}
        The 0-dimensional moduli spaces are compact and the 1-dimensional moduli spaces admit a compactification by broken solutions.
        The map 
        \begin{equation}
            \uformal^k : \EFloerC_\rho^\ArbitraryIndex (\Manifold; \EqntHam) \to \EFloerC_\rho^{\ArbitraryIndex + 2k} (\Manifold; \EqntHam)    
        \end{equation}
        which counts the 0-dimensional moduli spaces commutes with the equivariant differential.
        Moreover, a standard argument yields a chain homotopy between $\uformal^k \ComposedWith \uformal^{k \Prime}$ and $\uformal^{k + k \Prime}$.
        This gives equivariant Floer cohomology the structure of a $\Integers [\uformal]$-module.
        To distinguish this module structure from the one in \autoref{sec:u-multiplication-shift-invariance}, we denote this new action by $\uformal \CupProduct \Argument$.
        
        \begin{remark}[Comparison of $\Integers \lbrack \uformal \rbrack$-module structures] 
            \label{rem:comparison-of-algebraic-and-geometric-module-structures}
            Suppose the equivariant Floer data satisfies the conditions for the module structures in both \autoref{sec:u-multiplication-shift-invariance} and this section.
            Let us compare the algebraic product $\uformal^k \cdot (\spherecriticalpoint_l, \WithFilling{x}) = (\spherecriticalpoint_{l + k}, \WithFilling{x})$ with the geometric product  $\uformal^k \CupProduct (\spherecriticalpoint_l, \WithFilling{x})$.
            Suppose $[(v, u, v^0)]$ is a solution to the geometric product $\uformal^k \CupProduct (\spherecriticalpoint_l, \WithFilling{x})$ with end point $(\spherecriticalpoint_{k^-}, \WithFilling{x}^-)$.
            The `Y'-shaped flowline $[v, v^0]$ is a solution to the cup product on $\InfiniteSphere / \Circle$, though it may not be isolated.
            Therefore we have $k^- \ge k + l$ since there are no solutions to the cup product otherwise.
            Moreover for $k^- = k + l$, the `Y'-shaped flowline is isolated, so $u$ is a continuation map between $\WithFilling{x}$ and $\WithFilling{x}^-$ (and this continuation map is an isomorphism because the Hamiltonian has not changed).
            As such, we can informally say that the two products agree on the $\spherecriticalpoint_{k + l}$ term.
            That said, the geometric product may have other terms on the $\spherecriticalpoint_{k^-}$ terms with $k^- > k + l$ unlike the algebraic product.
            The author has not determined whether the two products are chain homotopic, however anticipates that the moduli spaces in \cite[Sections~3 and 5]{seidel_connections_2016} can be extended to get a chain homotopy.
        \end{remark}
        
\section{Equivariant Floer Seidel map}

    \label{sec:equivariant-floer-seidel-map}
    
    In this section, we extend the definition of the Floer Seidel map of \autoref{sec:floer-seidel-map} to the equivariant setup introduced in \autoref{sec:equivariant-floer-theory}.
    
    Let $\Lifted{\CircleAction}$ be a lift of a linear Hamiltonian circle action and let $\rho$ be a symplectic circle action which is linear at infinity, as per \autoref{sec:linear-symplectic-circle-action}.
    Assume that the $\CircleAction$ and $\rho$ commute.

    \subsection{Equivariant Floer Seidel map definition}
    
        \label{sec:equivariant-floer-seidel-map-definition}
        
        Let $(\EqntHam, \EqntACS)$ be a regular choice of equivariant Floer data for the action $\rho$.
        The \define{pullback equivariant Floer data $(\CircleAction \PullBack \EqntHam, \allowbreak \CircleAction \PullBack \EqntACS)$} are given by the same formulae as in the non-equivariant case (equations \eqref{eqn:pullback-almost-complex-structure} and \eqref{eqn:pullback-hamiltonian}), so we have
        \begin{equation}
            \label{eqn:pullback-equivariant almost-complex-structure}
            (\CircleAction \PullBack \EqntACS)_{\infsphereelement, t} = \left( \Derivative \CircleAction_t \right)\Inverse \EqntACS_{\infsphereelement, t} \Derivative \CircleAction_t
        \end{equation}
        and
        \begin{equation}
            \label{eqn:pullback-equivariant-hamiltonian}
            (\CircleAction \PullBack \EqntHam)_{\infsphereelement, t} (\ManifoldElement) = \EqntHam_{\infsphereelement, t} (\CircleAction_t(\ManifoldElement)) - \CircleHam(\CircleAction_t(\ManifoldElement))
        \end{equation}
        for all $\infsphereelement \in \InfiniteSphere$ and $\ManifoldElement \in \Manifold$.
        
        The pullback equivariant Floer data are regular equivariant Floer data for the pullback action $\CircleAction \PullBack \rho$.
        We show this for the pullback Hamiltonian as follows.
        
        \begin{proof}[Proof of equivariance of \eqref{eqn:pullback-equivariant-hamiltonian}]
            Since $\CircleAction$ and $\rho$ commute, the Lie bracket of their vector fields vanishes.
            The function $\SymplecticForm(\CircleVectorField{\rho}, \CircleVectorField{\CircleAction}) : \Manifold \to \RealNumbers$ has Hamiltonian vector field equal to this bracket $\LieBracket{\CircleVectorField{\CircleAction}}{ \CircleVectorField{\rho}}$ \cite[Proposition~18.3]{silva_lectures_2001}, and is therefore constant.
            The value of the constant is 0 because $\CircleAction$ has a fixed point by \autoref{thm:circle-action-has-fixed-point}.
            This yields
            \begin{equation}
                \label{eqn:hamiltonian-unchanging-in-complementary-direction}
                \LiebnitzDerivative{}{\CircleElement} (\CircleHam(\rho_\CircleElement(\ManifoldElement))) = (\ExteriorDerivative \CircleHam)_{\rho_\CircleElement(\ManifoldElement)} ((\CircleVectorField{\rho})_{\rho_\CircleElement (\ManifoldElement)} ) = \SymplecticForm (\CircleVectorField{\rho}, \CircleVectorField{\CircleAction}) \EvaluateAt{\rho_\CircleElement(\ManifoldElement)} = 0 \EndComma
            \end{equation}
            from which we deduce that $\CircleHam$ is constant along $\rho$.
            We use this in line \eqref{eqn:using-hamiltonian-constant-along-commuting-flow} to deduce the desired equivariance condition:
            \begin{subequations}
                \begin{align}
                    &(\CircleAction \PullBack \EqntHam)_{\CircleElement \Inverse \cdot \infsphereelement, t + \CircleElement} \big((\CircleAction \PullBack \rho)_\CircleElement(\ManifoldElement) \big) \nonumber\\
                    & \qquad = \EqntHam_{\CircleElement \Inverse \cdot \infsphereelement, t + \CircleElement} \big(\CircleAction_{t + \CircleElement} \big((\CircleAction \PullBack \rho)_\CircleElement (\ManifoldElement) \big)\big) - \CircleHam \big(\CircleAction_{t + \CircleElement} \big((\CircleAction \PullBack \rho)_\CircleElement(\ManifoldElement) \big) \big) \\
                    & \qquad = \EqntHam_{\CircleElement \Inverse \cdot \infsphereelement, t + \CircleElement} \big( \rho_{\CircleElement} (\CircleAction _t (\ManifoldElement))\big) - \CircleHam \big( \rho_{\CircleElement} (\CircleAction _t (\ManifoldElement))\big)
                    \label{eqn:using-hamiltonian-constant-along-commuting-flow}\\
                    & \qquad = \EqntHam_{\infsphereelement, t} \big(\CircleAction_t(\ManifoldElement) \big) - \CircleHam \big( \CircleAction_t(\ManifoldElement) \big)\\
                    & \qquad = (\CircleAction \PullBack \EqntHam)_{\infsphereelement, t} (\ManifoldElement) \EndFullStop
                \end{align}
            \end{subequations}
        \end{proof}
        
        If $\EqntHam$ has slope $\slope$ and the Hamiltonian $\CircleHam$ of $\CircleAction$ has slope $\CircleSlope$, then the pullback $\CircleAction \PullBack \EqntHam$ has slope $\slope - \CircleSlope$.
        
        Recall that, by definition, $\Lifted{\CircleAction}$ is a lift of the automorphism $\CircleAction : \ContractibleLoopSpace \Manifold \to \ContractibleLoopSpace \Manifold$, given by $(\CircleAction(x))(t) = \CircleAction_t \cdot x(t)$, to an automorphism of $\ContractibleLoopSpaceWithFillings{\Manifold}$.
        We have a commutative diagram of automorphisms on $\InfiniteSphere \times \ContractibleLoopSpace \Manifold$
        \begin{equation}
            \label{eqn:diagram-of-equivariant-action-on-orbits-commuting-with-seidel-action}
            \begin{tikzcd}[column sep=huge]
                (\infsphereelement, t \mapsto x(t) \,) \arrow[r, "\Identity_{\InfiniteSphere} \times \Lifted{\CircleAction}"] \arrow[d, "\cdot \CircleElement \ \text{for $\CircleAction \PullBack \rho$}"] & (\infsphereelement, t \mapsto \CircleAction_t (x(t)) \,) \arrow[d, "\cdot \CircleElement\ \text{for $\rho$}"] \\
                (\CircleElement \Inverse \cdot \infsphereelement, t \mapsto (\CircleAction \PullBack \rho)_\CircleElement (x(t - \CircleElement)) \,) \arrow[r, "\Identity_{\InfiniteSphere} \times \Lifted{\CircleAction}"]                 & (\CircleElement \Inverse \cdot \infsphereelement, t \mapsto \rho_\CircleElement ( \CircleAction_{t-\CircleElement} ( x(t - \CircleElement)))\,)            
            \end{tikzcd}
        \end{equation}
        which lifts to $\InfiniteSphere \times \ContractibleLoopSpaceWithFillings{\Manifold}$.
        Just like the non-equivariant case, the map $(\Identity_{\InfiniteSphere} \times \Lifted{\CircleAction}) \Inverse$ takes equivariant Hamiltonian orbits of $\EqntHam$ to equivariant orbits of $\CircleAction \PullBack \EqntHam$, and induces similar isomorphisms on the moduli spaces of equivariant Floer solutions.
        As such, we get an isomorphism of cochain complexes 
        \begin{equation}
            \label{eqn:equivariant-floer-seidel-map-definition}
            \begin{aligned}
                \EFloerSeidel_{\Lifted{\CircleAction}} : \EFloerC^\ArbitraryIndex_\rho(\Manifold; \EqntACS, \EqntHam) &\to \EFloerC^{\ArbitraryIndex + 2 \MaslovIndex(\Lifted{\CircleAction})}_{\CircleAction \PullBack \rho}(\Manifold; \CircleAction \PullBack \EqntACS, \CircleAction \PullBack \EqntHam) \\
                (\spherecriticalpoint_k, \WithFilling{x}) &\mapsto (\spherecriticalpoint_k, \Lifted{\CircleAction} \PullBack \WithFilling{x}) \EndFullStop
            \end{aligned}
        \end{equation}
        This \define{equivariant Floer Seidel map} preserves both the geometric and algebraic module structures on cohomology because it induces isomorphisms between the relevant moduli spaces.
        
        \begin{remark}
            \label{rem:requirement-for-circle-action-in-equivariant-case}
            The diagram \eqref{eqn:diagram-of-equivariant-action-on-orbits-commuting-with-seidel-action} commutes if and only if $\CircleAction$ is indeed an action, and fails to commute if $\CircleAction$ is merely a based loop in $\HamiltonianSymplectomorphismGroup(\Manifold, \SymplecticForm)$.
            This diagram is the reason that our equivariant construction requires this stronger assumption, unlike the non-equivariant case.
        \end{remark}
        
        As per the non-equivariant case, we can pullback regular equivariant monotone homotopies.
        Thus for any equivariant continuation map $\ContinuationMap$ we get the following commutative diagram.
        \begin{equation}
            \label{eqn:equivariant-floer-seidel-map-commutes-with-equivariant-continuation-map}
            \begin{tikzcd}
                \EFloerC^\ArbitraryIndex_{\rho}(M; \FamilyACS^{\eqnt,+}, \Ham^{\eqnt,+})
                \arrow[d, "\ContinuationMap"'] 
                \arrow[r, "\EFloerSeidel_{\Lifted{\CircleAction}}"] 
                & \EFloerC^{\ArbitraryIndex + 2 \MaslovIndex(\Lifted{\CircleAction})}_{\CircleAction \PullBack \rho}(M; \CircleAction \PullBack \FamilyACS^{\eqnt,+}, \CircleAction \PullBack \Ham^{\eqnt,+})
                \arrow[d, "\CircleAction \PullBack \ContinuationMap"] 
                \\
                \EFloerC^\ArbitraryIndex_{\rho}(M; \FamilyACS^{\eqnt,-}, \Ham^{\eqnt,-}) 
                \arrow[r, "\EFloerSeidel_{\Lifted{\CircleAction}}"']
                & \EFloerC^{\ArbitraryIndex + 2 \MaslovIndex(\Lifted{\CircleAction})}_{\CircleAction \PullBack \rho}(M; \CircleAction \PullBack \FamilyACS^{\eqnt,-}, \CircleAction \PullBack \Ham^{\eqnt,-})
            \end{tikzcd}
        \end{equation}

        When $\CircleAction$ is linear of strictly positive slope $\CircleSlope$, this offers a way to compute equivariant symplectic cohomology (this construction is the equivariant version of \cite[Theorem~22]{ritter_floer_2014}).
        Let $\epsilon > 0$ be smaller than any positive Reeb period.
        For each $r \in \Integers_{\ge 0}$, choose equivariant Floer data $(\Ham^{\eqnt}_r, \EqntACS_r)$ for the action $\CircleAction^{-r} \rho$ of slope $\epsilon$.
        The pullback data $((\CircleAction^{-r}) \PullBack \Ham^{\eqnt}_r, (\CircleAction^{-r}) \PullBack \EqntACS_r)$ is equivariant for the action $\rho$ and has slope $\epsilon + r \CircleSlope$.
        Consider the following commutative diagram, where $\ContinuationMap_r$ are the continuation maps between the relevant Floer data.
        We have omitted $M$, degrees and the almost complex structure from the notation for clarity.
        \begin{equation}
            \label{eqn:commutative-diagram-for-computing-symplectic-cohomology-using-floer-seidel-map-limit}
            \begin{tikzcd}[cramped]
                \EFloerCohomology_{\rho}(\Ham^{\eqnt}_0) \arrow[r, "\ContinuationMap_0"]] \arrow[dr, dashed, to path=|- (\tikztotarget)]
                & \EFloerCohomology_{\rho}((\CircleAction^{-1}) \PullBack \Ham^{\eqnt}_1) \arrow[r, "(\CircleAction^{-1}) \PullBack \ContinuationMap_1"] \arrow[d, "\EFloerSeidel_{\Lifted{\CircleAction}}", "\cong"']
                & \EFloerCohomology_{\rho}((\CircleAction^{-2}) \PullBack \Ham^{\eqnt}_2) \arrow[r, "(\CircleAction^{-2}) \PullBack \ContinuationMap_2"] \arrow[d, "\EFloerSeidel_{\Lifted{\CircleAction}}", "\cong"'] & {\ldots} 
                \\
                & \EFloerCohomology_{\CircleAction^{-1}\rho}(\Ham^{\eqnt}_1) \arrow[r, "\ContinuationMap_1"] \arrow[dr, dashed, to path=|- (\tikztotarget)]
                & \EFloerCohomology_{\CircleAction^{-1}\rho}((\CircleAction^{-1}) \PullBack \Ham^{\eqnt}_2) \arrow[d, "\EFloerSeidel_{\Lifted{\CircleAction}}", "\cong"'] \arrow[r, "(\CircleAction^{-1}) \PullBack \ContinuationMap_2"]
                & {\ldots} 
                \\
                &
                & \EFloerCohomology_{\CircleAction^{-2}\rho}(\Ham^{\eqnt}_2) \arrow[r, "\ContinuationMap_2"] \arrow[dr, dashed, to path=|- (\tikztotarget)]
                & {\ldots}
                \\
                &
                &
                & \ldots
            \end{tikzcd}
        \end{equation}
        The direct limit of the top line of \eqref{eqn:commutative-diagram-for-computing-symplectic-cohomology-using-floer-seidel-map-limit} is, by definition, the equivariant symplectic cohomology of $\Manifold$ for the action $\rho$.
        The diagram shows this is isomorphic to the direct limit of the dashed maps $\EFloerSeidel_{\Lifted{\CircleAction}} \ComposedWith \ContinuationMap_r$.
        In \autoref{sec:equivariant-gluing}, we find another way to characterise these dashed maps using equivariant quantum cohomology.
        This diagram is crucial for our calculations in \autoref{sec:example-of-complex-plane} and \autoref{sec:example-of-tautological-line-bundle-on-projective-space}.
        
    \subsection{Equivariant Floer Seidel map properties}
    
        \label{sec:equivariant-floer-seidel-map-properties}
        
        The equivariant Floer Seidel map is a very natural operation, and as such it preserves many of the structures associated to equivariant Floer cohomology.
        It is a module map for both the geometric and algebraic module structures, and is compatible with continuation maps as described above.
        As such, after taking a direct limit, we get an isomorphism of equivariant symplectic cohomologies $\EFloerSeidel_{\Lifted{\CircleAction}} : \ESymplecticCohomology^\ArbitraryIndex_{\rho}(M) \overset{\cong}{\to} \ESymplecticCohomology^{\ArbitraryIndex + 2 \MaslovIndex(\Lifted{\CircleAction})}_{\CircleAction \PullBack \rho} (\Manifold)$.
        In the remainder of this section, we describe two further structures on equivariant symplectic cohomology which are compatible with the equivariant Floer Seidel map.
        
        \subsubsection{Gysin sequence}
        
            \label{sec:gysin-sequence-and-equivariant-floer-seidel-map}
            
            Analogously to \cite{bourgeois_s1-equivariant_2017}, there is a long exact sequence on equivariant symplectic cohomology
            \begin{equation}
                \label{eqn:gysin-sequence-on-equivariant-symplectic-cohomology}
                \cdots \to \ESymplecticCohomology^\ArbitraryIndex_\rho(\Manifold) \overset{\cdot \uformal}{\to} \ESymplecticCohomology^{\ArbitraryIndex + 2}_\rho(\Manifold) \to \SymplecticCohomology^{\ArbitraryIndex + 2}(\Manifold) \to \ESymplecticCohomology^{\ArbitraryIndex+1}_\rho(\Manifold) \to \cdots \EndFullStop
            \end{equation}
            This is an immediate algebraic consequence of our definitions because there is a short exact sequence of cochain complexes, where $\cdot \uformal$ denotes the algebraic $\Integers [\uformal]$-module operation.
            The second map $\ESymplecticCohomology^\ArbitraryIndex_\rho (\Manifold) \to \SymplecticCohomology^\ArbitraryIndex(\Manifold)$ is the map induced by $(\spherecriticalpoint_0, \WithFilling{x}) \mapsto \WithFilling{x}$ and $(\spherecriticalpoint_k, \WithFilling{x}) \mapsto 0$ for $k>0$ on equivariant Floer cohomology.
            This long exact sequence is the \define{Gysin exact sequence} of equivariant symplectic cohomology.
            
            The equivariant Floer Seidel maps are isomorphisms on the cochain complexes which are compatible with the maps in \eqref{eqn:gysin-sequence-on-equivariant-symplectic-cohomology}, and therefore fit into the following commutative diagram.
            \begin{equation}
                \label{eqn:commutative-diagram-for-gysin-sequence-under-floer-seidel-map}
                \begin{tikzcd}[column sep=small]
                    \cdots \arrow[r]
                    & \ESymplecticCohomology^\ArbitraryIndex_\rho(\Manifold) \arrow[r, "\cdot \uformal"] \arrow[d, "\EFloerSeidel_{\Lifted{\CircleAction}}"]
                    & \ESymplecticCohomology^{\ArbitraryIndex + 2}_\rho(\Manifold) \arrow[r] \arrow[d, "\EFloerSeidel_{\Lifted{\CircleAction}}"]
                    & \SymplecticCohomology^{\ArbitraryIndex + 2}(\Manifold) \arrow[r] \arrow[d, "\FloerSeidel_{\Lifted{\CircleAction}}"]
                    & \cdots
                    \\ \cdots \arrow[r]
                    & \ESymplecticCohomology^{\ArbitraryIndex + 2 \MaslovIndex(\Lifted{\CircleAction})}_{\CircleAction \PullBack \rho}(\Manifold) \arrow[r, "\cdot \uformal"] 
                    & \ESymplecticCohomology^{\ArbitraryIndex + 2 + 2 \MaslovIndex(\Lifted{\CircleAction})}_{\CircleAction \PullBack \rho}(\Manifold) \arrow[r]
                    & \SymplecticCohomology^{\ArbitraryIndex + 2 + 2 \MaslovIndex(\Lifted{\CircleAction})}(\Manifold) \arrow[r]
                    & \cdots
                \end{tikzcd}
            \end{equation}
        
        \subsubsection{Filtration and positive equivariant symplectic cohomology}
        
            \label{sec:filtration-and-positive-equivariant-symplectic-cohomology}
            
            For exact symplectic manifolds, the Floer cochain complexes may be equipped with an action filtration which distinguishes between different orbits \cite{viterbo_functors_1999}.
            Roughly, the orbits of a Hamiltonian with small\footnote{
                The slope $\slope$ of a Hamiltonian is \define{small} if it is smaller than any positive Reeb period, i.e. it satisfies $\slope < \min (\ReebPeriods \Intersection \IntervalOpen{0}{\Infinity})$.
            } positive slope correspond, via the PSS map of \autoref{sec:gluing-construction}, to the quantum cohomology of $\Manifold$.
            Conversely, orbits which occur only for Hamiltonians with larger slopes correspond to Reeb orbits on $\ContactManifold$, with the period of the Reeb orbit corresponding to the slope of the Hamiltonian at the radius of the orbit, as per \eqref{eqn:hamiltonian-vector-field-at-infinity-is-multiple-of-reeb-vector-field}.
            
            In our setup, however, the action functional \eqref{eqn:action-functional} fails to provide a filtration for two reasons.
            First, since we have not assumed $\Manifold$ is exact, the value of $\ActionFunctional_\Ham(x)$ depends on the lift of $x$ to $\ContractibleLoopSpaceWithFillings{\Manifold}$.
            Second, the action functional may not decrease along equivariant Floer trajectories that are not constant in $\InfiniteSphere$ (see \cite[page~3867]{bourgeois_s1-equivariant_2017}).
            These issues may be resolved by choosing special equivariant Floer data and adding a cut-off function to the action filtration, an approach taken in \cite[Appendix~D]{mclean_mckay_2018}.
            We briefly outline this construction.
            
            Fix two numbers\footnote{
                Choose $R_0$ large enough so that $\CircleAction$ is linear on $\Set{R > R_0}$.
            } $1 < R_0 < R_1 < \Infinity$.
            Our equivariant Hamiltonian $\EqntHam$ will have positive small slope at $R= R_0$, be quadratic and increasing on $R_0 < R < R_1$, and it will be linear on $R > R_1$; in particular, it will only depend on $R$ in the region $R > R_0$, modulo a small perturbation we will ignore.
            We write $\EqntHam_{\infsphereelement, t}(\ConvexCoordMap(y, R)) = h(R)$ to emphasize this.
            Fix a \define{cut-off function} $\beta : \IntervalClosedOpen{1}{\Infinity} \to \RealNumbers$; this is a smooth increasing function which is 0 on $\IntervalClosedOpen{1}{R_0}$ and has increasing gradient on the open interval $\IntervalOpen{R_0}{R_1}$.
            Define $f_h : \RealNumbers \to \RealNumbers$ to be
            \begin{equation}
                \label{eqn:small-f-for-filtration}
                f_h(R) = \int_0^R \beta \Prime (r) h\Prime (r) \wrt{r}
            \end{equation}
            and define the function $F_h : \ContractibleLoopSpace \Manifold \to \RealNumbers$ by
            \begin{equation}
                \label{eqn:definition-of-filtration}
                F_h(x) = - \int_{\Circle} x \PullBack (\beta(R) \ContactForm) + \int_{t=0}^1 f_h(R \ComposedWith x) \wrt{t} \EndFullStop
            \end{equation}
            Notice how \eqref{eqn:definition-of-filtration} resembles the action functional \eqref{eqn:action-functional} in the exact setup if we were to change $\beta$ to $\Identity_R$ (of course this choice doesn't satisfy the requirements for $\beta$).
            The function $F_h$ decreases along any equivariant Floer trajectory.
            Therefore the inclusion
            \begin{equation}
                \label{eqn:inclusion-of-cochain-complex-using-filtration}
                \EFloerC^\ArbitraryIndex_\rho(\Manifold; \EqntHam; F_h \ge 0) \Embedding \EFloerC^\ArbitraryIndex_\rho(\Manifold; \EqntHam) \EndComma
            \end{equation}
            of the subcomplex generated by the orbits which satisfy $F_h \ge 0$ is a chain map.
            
            The cohomology of the quotient cochain complex of \eqref{eqn:inclusion-of-cochain-complex-using-filtration} is the \define{positive equivariant Floer cohomology}, denoted $\EFloerCohomology^\ArbitraryIndex_{\rho,+}(\Manifold; \EqntHam)$.
            This construction is compatible with continuation maps, so we can take a direct limit of $\EFloerCohomology^\ArbitraryIndex_{\rho,+}(\Manifold; \EqntHam)$ under continuation maps with the slopes increasing.
            This direct limit is the \define{positive equivariant symplectic cohomology} of $\Manifold$, and is written $\ESymplecticCohomology^\ArbitraryIndex_{\rho,+}(\Manifold)$.
            The equivariant PSS maps of \autoref{sec:equivariant-gluing} give an isomorphism between the cohomology of the subcomplex with equivariant quantum cohomology $\EFloerCohomology^\ArbitraryIndex_\rho(\Manifold; \EqntHam; {F_h \ge 0}) \cong \EQuantumCohomology^\ArbitraryIndex_\rho(\Manifold)$.
            Associated to the short exact sequence induced by the inclusion \eqref{eqn:inclusion-of-cochain-complex-using-filtration}, there is a long exact sequence
            \begin{equation}
                \label{eqn:long-exact-sequence-equivariant-positive-symplectic-cohomology}
                \cdots \to \EQuantumCohomology^\ArbitraryIndex_\rho(\Manifold) \to \ESymplecticCohomology^\ArbitraryIndex_{\rho}(\Manifold) \to \ESymplecticCohomology^\ArbitraryIndex_{\rho,+}(\Manifold) \to \EQuantumCohomology^{\ArbitraryIndex + 1} _\rho (\Manifold) \to \cdots \EndFullStop
            \end{equation}
            
            We have the following compatibility result between the filtration and the equivariant Floer Seidel map.
            \begin{theorem}
                \label{thm:filtration-preserved-by-seidel-map}
                For any equivariant Hamiltonian orbit $(\spherecriticalpoint_k, x)$, the filtration satisfies
                \begin{equation}
                    \label{eqn:filtration-preserved-by-seidel-map}
                    F_{\CircleAction \PullBack h} (\CircleAction \PullBack x) = F_{h}(x) \EndFullStop
                \end{equation}
            \end{theorem}
            
            \begin{proof}
                By definition, we have $(\CircleAction \PullBack h)(R) = h(R) - \CircleSlope R$, so that $f_{\CircleAction \PullBack h} (R) = f_h (R) - \CircleSlope \beta(R)$, where $\CircleSlope$ is the slope of $\CircleAction$.
                Without loss of generality, let $x$ occur in the region $\Set{R_0 < R < R_1}$ since otherwise both sides of \eqref{eqn:filtration-preserved-by-seidel-map} are 0.
                Suppose that $x$ has period $l$.
                The period of $\CircleAction \PullBack x$ is $l - \CircleSlope$.
                We have
                \begin{equation*}
                    \begin{aligned}
                        F_{\CircleAction \PullBack h} (\CircleAction \PullBack x)
                        &= - \int_{\Circle} (\CircleAction \PullBack x) \PullBack (\beta(R) \ContactForm) + \int_{t=0}^1 f_{\CircleAction \PullBack h}(R \ComposedWith (\CircleAction \PullBack x)) \wrt{t} \\
                        &= - \beta(R \ComposedWith (\CircleAction \PullBack x)) \, (l - \CircleSlope) + \int_{t=0}^1 f_{\CircleAction \PullBack h}(R \ComposedWith (\CircleAction \PullBack x)) \wrt{t} \\
                        &= - \beta(R \ComposedWith x) \, (l - \CircleSlope) + \int_{t=0}^1 f_{\CircleAction \PullBack h}(R \ComposedWith x) \wrt{t} \\
                        &= - \beta(R \ComposedWith x) \, (l - \CircleSlope) + \int_{t=0}^1 f_{h}(R \ComposedWith x) - \CircleSlope \beta(R \ComposedWith x) \wrt{t} \\
                        &= - l \beta(R \ComposedWith x) + \int_{t=0}^1 f_{h}(R \ComposedWith x) \wrt{t} \\
                        &= - \int_{\Circle} x \PullBack (\beta(R) \ContactForm) + \int_{t=0}^1 f_{h}(R \ComposedWith x) \wrt{t} \\
                        &= F_h (x) \EndFullStop
                    \end{aligned}
                \end{equation*}
            \end{proof}
            
            Unfortunately, the pullback Hamiltonian $\CircleAction \PullBack \EqntHam$ does not satisfy the condition that its slope at $R_0$ is small; instead its slope at $R_0$ will be $\epsilon - \CircleSlope$ if $\EqntHam$ had slope $\epsilon$ at $R_0$.
            To rectify the situation for positive equivariant symplectic cohomology, we have to consider the subcomplex generated by orbits with $F_{\CircleAction \PullBack h} \ge f_{\CircleAction \PullBack h} (R_0 \Prime)$, where $R_0 \Prime$ is the radius at which $h\Prime - \CircleSlope = 0$.
            The \define{positive equivariant Floer Seidel map} is the composition of the equivariant Floer Seidel map with the map induced on the quotient complexes corresponding to the inclusions of these subcomplexes.
            More explicitly, abusing notation, we have the following diagram.
            \begin{equation}
                \begin{tikzcd}
                    \ESymplecticCohomology^\ArbitraryIndex_{\rho,+} \arrow[r, equal] \arrow[dd, "\EFloerSeidel_{\Lifted{\CircleAction}, +}"]
                    & \ESymplecticCohomology^\ArbitraryIndex_{\rho} (F_h < 0) \arrow[d, "\EFloerSeidel_{\Lifted{\CircleAction}}"]
                    \\ & \ESymplecticCohomology^{\ArbitraryIndex + 2 \MaslovIndex(\Lifted{\CircleAction})}_{\CircleAction \PullBack \rho} (F_{\CircleAction \PullBack h} < 0) \arrow[d]
                    \\ \ESymplecticCohomology^{\ArbitraryIndex + 2 \MaslovIndex(\Lifted{\CircleAction})}_{\CircleAction \PullBack \rho, +} \arrow[r, equal]
                    & \ESymplecticCohomology^{\ArbitraryIndex + 2 \MaslovIndex(\Lifted{\CircleAction})}_{\CircleAction \PullBack \rho} (F_{\CircleAction \PullBack h} < f_{\CircleAction \PullBack h} (R_0 \Prime))
                \end{tikzcd}
            \end{equation}
            This gives us the following morphism of long exact sequences.
            \begin{equation}
                \label{eqn:seidel-maps-interaction-with-long-exact-sequence-for-positive-equivariant-symplectic-cohomology}
                \begin{tikzcd}[cramped, column sep=tiny]
                    \cdots \arrow[r] & \EQuantumCohomology^\ArbitraryIndex_\rho \arrow[d, "\EQuantumSeidelMap_{\Lifted{\CircleAction}}"]\arrow[r] & \ESymplecticCohomology^\ArbitraryIndex_{\rho} \arrow[d, "\EFloerSeidel_{\Lifted{\CircleAction}}", "\cong"']\arrow[r] & \ESymplecticCohomology^\ArbitraryIndex_{\rho,+} \arrow[d, "\EFloerSeidel_{\Lifted{\CircleAction}, +}"]\arrow[r] & \EQuantumCohomology^{\ArbitraryIndex + 1} _\rho  \arrow[d, "\EQuantumSeidelMap_{\Lifted{\CircleAction}}"]\arrow[r] & \cdots
                    \\
                    \cdots \arrow[r] & \EQuantumCohomology^{\ArbitraryIndex + 2 \MaslovIndex(\Lifted{\CircleAction})}_{\CircleAction \PullBack\rho} \arrow[r] & \ESymplecticCohomology^{\ArbitraryIndex + 2 \MaslovIndex(\Lifted{\CircleAction})}_{\CircleAction \PullBack\rho} \arrow[r] & \ESymplecticCohomology^{\ArbitraryIndex + 2 \MaslovIndex(\Lifted{\CircleAction})}_{\CircleAction \PullBack\rho,+} \arrow[r] & \EQuantumCohomology^{\ArbitraryIndex + 2 \MaslovIndex(\Lifted{\CircleAction}) + 1} _{\CircleAction \PullBack\rho}  \arrow[r] & \cdots
                \end{tikzcd}
            \end{equation}
            The five lemma applied to the diagram \eqref{eqn:seidel-maps-interaction-with-long-exact-sequence-for-positive-equivariant-symplectic-cohomology} implies that $\EFloerSeidel_{\Lifted{\CircleAction}, +}$ is an isomorphism if and only if $\EQuantumSeidelMap_{\Lifted{\CircleAction}}$ is an isomorphism.
            
\section{Quantum theory}

    \label{sec:quantum-theory}
    
    \subsection{Quantum cohomology}
    
        \label{sec:quantum-cohomology}
    
        \subsubsection{Morse cohomology}
        
            \label{sec:morse-cohomology}
            
            Fix a Riemannian metric on $\Manifold$.
            Let $\MorseFunction : \Manifold \to \RealNumbers$ be a Morse-Smale function which increases in the radial coordinate direction at infinity.
            Thus the inequality $\partial_R (\MorseFunction(\ConvexCoordMap(y, R))) > 0$ holds at infinity.
            Denote by $\CriticalPoints(\MorseFunction)$ the finite set of critical points of $\MorseFunction$.
            The \define{Morse index $\MorseIndex(x)$} of a critical point $x$ is the dimension of the maximal subspace of the tangent space at $x$ on which the Hessian of $\MorseFunction$ is negative definite.
            The \define{Morse cohomology of $\Manifold$} is the cohomology of the cochain complex freely generated by $\CriticalPoints(\MorseFunction)$ whose differential counts negative gradient trajectories between critical points.
            It is isomorphic to the (singular) cohomology of $\Manifold$.
            Through the choice of an orientation of each unstable manifold, the count of trajectories is signed, so that Morse cohomology is a $\Integers$-graded Abelian group.
            
        \subsubsection{Quantum product}
            
            Let $\AlmostComplexStructure$ be a regular convex $\SymplecticForm$-compatible almost complex structure.
            Fix distinct points $p^-, p_1^+, p_2^+ \in \Projective^1$.
            The quantum product counts quadruples $(u, \gamma^-, \gamma_1^+, \gamma_2^+)$, where $u : \Projective^1 \to \Manifold$ is a simple (or constant) $\AlmostComplexStructure$-holomorphic sphere and $\gamma^- : \IntervalOpenClosed{-\Infinity}{0} \to \Manifold$ and $\gamma_i^+ : \IntervalClosedOpen{0}{\Infinity} \to \Manifold$ are negative gradient flowlines, with the intersection conditions $\gamma_i^+(0) = u(p_i^+)$ and $\gamma^-(0) = u(p^-)$.
            Denote by $\ModuliSpace(x^-, x_1^+, x_2^+; A)$ the space of such quadruples where $u$ represents $A \in \InterestingSpheres$ and the limits $\gamma_i^+(s) \to x_i^+$ and $\gamma^-(s) \to x^-$ hold as $s \to \pm \Infinity$.
            
            With a regular choice of three $s$-dependent perturbations $\MorseFunction_1^+, \MorseFunction_2^+, \MorseFunction^-$ of the Morse-Smale function $\MorseFunction$, the moduli spaces will all be smooth oriented manifolds with
            \begin{equation}
                \label{eqn:dimension-of-quantum-product-moduli-spaces}
                \Dimension\ModuliSpace(x^-, x_1^+, x_2^+; A) = \MorseIndex(x^-) -  \MorseIndex(x_1^+) - \MorseIndex(x_2^+) + 2 \FirstChernClass(A) \EndFullStop
            \end{equation}
            Via standard compactification and gluing arguments\footnote{
                In the region at infinity where $\AlmostComplexStructure$ is convex, any $\AlmostComplexStructure$-holomorphic sphere cannot achieve a maximal value of $R$ by a maximum principle.
                As such, all the holomorphic spheres lie in a compact region and standard compactification results apply.
            }, the 0-dimensional moduli space is compact and the 1-dimensional moduli space may be compactified to a manifold whose boundary is made up of the broken trajectories (the sphere will not bubble by regularity).
            As such, the map $\CochainComplex^\ArbitraryIndex(\Manifold; \MorseFunction; \NovikovRing) ^{\otimes 2} \to \CochainComplex^\ArbitraryIndex(\Manifold; \MorseFunction; \NovikovRing)$ given by
            \begin{equation}
                x_1^+ \QuantumProduct x_2^+ = 
                \sum_{
                    \substack{
                        A \in \InterestingSpheres \\
                        x^- \in \CriticalPoints(\MorseFunction) \\
                        \Dimension\ModuliSpace(x^-, x_1^+, x_2^+; A) = 0
                    }
                }
                \sum_{
                    (u, \gamma^\pm_\bullet) \in \ModuliSpace(x^-, x_1^+, x_2^+; A)
                } \Orientation(u, \gamma^\pm_\bullet) \ \NovVariable^A x^-
            \end{equation}
            is a chain map, and hence it induces a product structure on the Morse cohomology of $\Manifold$ with coefficients in the Novikov ring $\NovikovRing$.
            This product is unital, skew-commutative and associative, and induces the structure of a $\Integers$-graded $\NovikovRing$-algebra on $\Cohomology^\ArbitraryIndex (\Manifold; \MorseFunction; \NovikovRing)$.
            This is the \define{quantum cohomology $\QuantumCohomology^\ArbitraryIndex (\Manifold)$} of the manifold $\Manifold$, and the product is the \define{quantum product}.
            
            \remark{
                The quantum cohomology a priori depends on the Riemannian metric, the Morse-Smale function $\MorseFunction$ and its three perturbations, the chosen orientations of the unstable manifolds and the almost-complex structure.
                However the dependence on all of this data may be removed up to canonical isomorphism via standard homotopy arguments.
                Moreover, quantum cohomology is independent of the parameterisation of the convex end because this information is used only to constrain all flowlines and spheres to a compact region.
            }
            
    \subsection{Quantum Seidel map}
    
        \label{sec:quantum-seidel-map}
            
        Let $\Lifted{\CircleAction}$ be a lifted linear Hamiltonian circle action on $\Manifold$ with nonnegative slope.
        
        \subsubsection{Clutching construction}
        
            \label{sec:clutching-bundle-construction}
            
            In this section, we define a symplectic $\Manifold$-bundle $\ClutchingBundle$ over the sphere associated to the action $\CircleAction$.
            
            \paragraph{Base space}
            
                The sphere $\Sphere$ is the union of its upper hemisphere $\Hemisphere^-$ and lower hemisphere $\Hemisphere^+$.
                Each hemisphere is a copy of the closed unit disc in the complex plane.
                The equator of the sphere is the circle $\Circle \cong \RealNumbers / \Integers$.
                We identify the boundaries of the hemispheres with the equator via
                \begin{equation}
                    t \in \Circle \leftrightarrow \ExponentialNumber^{2 \PiNumber \ImaginaryNumber t} \in \Boundary \Hemisphere^- \leftrightarrow \ExponentialNumber^{- 2 \PiNumber \ImaginaryNumber t} \in \Boundary \Hemisphere^+ \EndFullStop
                \end{equation}
                The poles of the sphere are the points $\Pole^\pm = 0 \in \Hemisphere^\pm$.
                The complement of the poles in the sphere is isomorphic to a cylinder via the map
                \begin{equation}
                    \label{eqn:cylindrical-coordinates-on-the-sphere}
                    \RealNumbers \times \Circle \ni (s, t) \mapsto \left\{ 
                        \begin{array}{lll}
                            \ExponentialNumber^{2 \PiNumber (s + \ImaginaryNumber t)} &\in \Hemisphere^- & \text{if $s \le 0$,} \\
                            \ExponentialNumber^{- 2 \PiNumber (s + \ImaginaryNumber t)} &\in \Hemisphere^+ & \text{if $s \ge 0$.}
                        \end{array}
                    \right.
                \end{equation}
                
            \remark{
                Our notation is opposite to that of Seidel \cite{seidel_$_1997} and Ritter \cite{ritter_floer_2014}, so our $\Hemisphere^\pm$ correspond to their $D^\mp$.
            }
            
            \paragraph{Total space}
            
                The smooth manifold $\ClutchingBundle$ is the union of the manifolds $\Hemisphere^\pm \times \Manifold$ glued along the boundaries via
                \begin{equation}
                    \Boundary \Hemisphere^- \times \Manifold \ni (\ExponentialNumber^{2 \PiNumber \ImaginaryNumber t}, \ManifoldElement) \leftrightarrow (\ExponentialNumber^{- 2 \PiNumber \ImaginaryNumber t}, \CircleAction_t \ManifoldElement) \in \Boundary \Hemisphere^+ \times \Manifold \EndFullStop
                \end{equation}
                The projection map $\ClutchingProjection : \ClutchingBundle \to \Sphere$ is the union of the projection maps $\Hemisphere^\pm \times \Manifold \to \Hemisphere^\pm$.
                Let $\ClutchingInclusions^\pm : \Manifold \to \Set{\Pole^\pm} \times \Manifold \subset \ClutchingBundle$ be the fibre inclusion maps over the poles.
            
            \paragraph{Symplectic bilinear form}
            
                Denote by $\VerticalTangentSpace \ClutchingBundle$ the kernel of $\Derivative \ClutchingProjection$.
                With $\pi_\Manifold^\pm : \Hemisphere^\pm \times \Manifold \to \Manifold$ the projection map, the vector space $\VerticalTangentSpace_{(\infsphereelement, \ManifoldElement)} \ClutchingBundle$ is equipped with a symplectic bilinear form $\ClutchingSymplecticBilinearForm_{(\infsphereelement, \ManifoldElement)} = (\pi_\Manifold^\pm)_{(\infsphereelement, \ManifoldElement)}\PullBack \SymplecticForm_m$.
                Since the circle action $\CircleAction$ is symplectic, these symplectic bilinear forms agree along the equator, so that $\VerticalTangentSpace \ClutchingBundle \to \ClutchingBundle$ is a symplectic vector bundle with symplectic bilinear form $\ClutchingSymplecticBilinearForm$.
            
            \paragraph{Global 2-form}
            
                There is a closed 2-form $\ClutchingTwoForm$ on $\ClutchingBundle$ which restricts to $\ClutchingSymplecticBilinearForm$ on $\VerticalTangentSpace \ClutchingBundle$.
                The construction of $\ClutchingTwoForm$ for convex symplectic manifolds, due to Ritter \cite[Section~5]{ritter_floer_2014}, uses a special\footnote{
                    The functions $\ClutchingHamiltonian{}^{,\pm} : \Hemisphere^\pm \times \Manifold \to \RealNumbers$ must vanish near the poles, be independent of the $s$-coordinate near the equator, and glue according to $\ClutchingHamiltonian_t{}^{,+} = \CircleAction \PullBack \ClutchingHamiltonian_t{}^{,-}$.
                    The gluing condition ensures the Hamiltonian vector field in $\VerticalTangentSpace \ClutchingBundle$ is well-defined along the equator.
                    Moreover, $\ClutchingHamiltonian{}^{,\pm}$ must both be monotone, by which we mean that in a region at infinity, the functions are dependent only on the radial coordinate $R$ and the $s$-coordinate of \eqref{eqn:cylindrical-coordinates-on-the-sphere}, and satisfy $\partial_s \ClutchingHamiltonian{}^{,\pm} \le 0$.
                    We assume $\CircleAction$ is linear of nonnegative slope precisely so that such Hamiltonian functions exist.
                } pair of Hamiltonians $\ClutchingHamiltonian{}^{,\pm} : \Hemisphere^\pm \times \Manifold \to \RealNumbers$ to modify the fibrewise symplectic form $\SymplecticForm$ so that it becomes a well-defined closed 2-form.
            
            \paragraph{Almost complex structures}
            
                The sphere has an (almost) complex structure $\StandardACS$ given by $\mp \ImaginaryNumber$ on $\Hemisphere^\pm$.
                Denote by $\ACSSpace(\ClutchingBundle)$ the space of almost complex structures $\ClutchingBundleACS$ on $\ClutchingBundle$ which satisfy the following properties:
                \begin{itemize}
                    \item
                        $\Derivative \ClutchingProjection$ is $(\ClutchingBundleACS, \StandardACS)$-holomorphic.
                    \item
                        $\ClutchingBundleACS \RestrictedTo{\VerticalTangentSpace \ClutchingBundle}$ is a convex $\ClutchingSymplecticBilinearForm$-compatible almost complex structure on $\VerticalTangentSpace \ClutchingBundle$.
                    \item
                        At infinity, $\ClutchingBundleACS$ has the form
                        \begin{equation}
                            \label{eqn:clutching-acs-at-infinity}
                            \ClutchingBundleACS_{(z, m)} = \left(\begin{matrix}
                                \StandardACS & 0\\
                                \ExteriorDerivative s \otimes \HamiltonianVectorField{\ClutchingHamiltonian_z{}^{,\pm}} - \ExteriorDerivative t \otimes \FamilyACS_z^\pm \HamiltonianVectorField{\ClutchingHamiltonian_z{}^{,\pm}} & \FamilyACS_z^\pm
                            \end{matrix}\right)
                        \end{equation}
                        with respect to the decomposition $\TangentSpace_{(z, m)} \ClutchingBundle = \TangentSpace_z \Hemisphere^\pm \oplus \TangentSpace_m \Manifold$ and the coordinates $(s, t)$ on the sphere from \eqref{eqn:cylindrical-coordinates-on-the-sphere}, denoting by $\FamilyACS_z^\pm$ the fibrewise restriction $\ClutchingBundleACS \RestrictedTo{\VerticalTangentSpace \ClutchingBundle}$ on each hemisphere.
                \end{itemize}
                Given any $\ClutchingBundleACS \in \ACSSpace(\ClutchingBundle)$, the 2-form $\ClutchingTwoForm + c \ClutchingProjection \PullBack \SphereSymplecticForm$ is symplectic and $\ClutchingBundleACS$ is $(\ClutchingTwoForm + c \ClutchingProjection \PullBack \SphereSymplecticForm)$-compatible for large enough $c > 0$.
                Here, we denote by $\SphereSymplecticForm$ the standard symplectic form on $\Sphere$.
                
            \remark{
                The motivation for \eqref{eqn:clutching-acs-at-infinity} is that any $\ClutchingBundleACS$-holomorphic section is locally a Floer solution for $(\ClutchingHamiltonian, \FamilyACS)$ whenever \eqref{eqn:clutching-acs-at-infinity} applies, and hence a maximum principle forbids any (non-fixed) $\ClutchingBundleACS$-holomorphic sections outside a compact region.
            }
             
            \paragraph{Sections}
                
                Two sections $s_1, s_2 : \Sphere \to \ClutchingBundle$ are \define{$\InterestingSpheres$-equivalent} if the conditions $\ClutchingTwoForm(s_1) = \ClutchingTwoForm(s_2)$ and $\FirstChernClass(\VerticalTangentSpace \ClutchingBundle)(s_1) = \FirstChernClass(\VerticalTangentSpace \ClutchingBundle)(s_2)$ hold, where $\FirstChernClass(\VerticalTangentSpace \ClutchingBundle)$ is the first Chern class of the symplectic vector bundle $\VerticalTangentSpace \ClutchingBundle \to \ClutchingBundle$.
                The property of $\InterestingSpheres$-equivalence is independent of the choice of global 2-form $\ClutchingTwoForm$ which restricts to $\ClutchingSymplecticBilinearForm$.
                Moreover, the group $\InterestingSpheres$ acts freely and transitively on $\InterestingSpheres$-equivalence classes of sections.
                
                Given any lift $\Lifted{\CircleAction}$ and any $\WithFilling{x} \in \ContractibleLoopSpaceWithFillings{\Manifold}$, we can produce a section by setting\footnote{
                    More precisely, by $\WithFilling{x}(z)$, we mean $u(z)$ for a choice of filling $u$ of $x$, and similarly for $\Lifted{\CircleAction} \WithFilling{x} (\Conjugate{z})$.
                    The resulting $\InterestingSpheres$-equivalence class is independent of the choice.
                } $z \mapsto (z, \WithFilling{x}(z))$ on $\Hemisphere^-$ and $z \mapsto (z, \Lifted{\CircleAction}(\WithFilling{x})(\Conjugate{z}))$ on $\Hemisphere^+$.
                The $\InterestingSpheres$-equivalence class of this section is independent of the choice of $\WithFilling{x}$.
                We denote it by $\SectionClass_{\Lifted{\CircleAction}}$.
                It satisfies $\MaslovIndex(\WithFilling{\CircleAction}) = - \FirstChernClass(\VerticalTangentSpace \ClutchingBundle)(\SectionClass_{\Lifted{\CircleAction}})$.
                Every $\InterestingSpheres$-equivalence class is $\SectionClass_{\Lifted{\CircleAction}} + A$ for a unique $A \in \InterestingSpheres$.
                
            \paragraph{Fixed sections}
                
                For every fixed point $\ManifoldElement \in \Manifold$ of the circle action $\CircleAction$, there is a constant section $s_\ManifoldElement : z \mapsto (z, \ManifoldElement)$.
                The section $s_\ManifoldElement$ is the \define{fixed section at $\ManifoldElement$}.
                For any fixed section $s_\ManifoldElement$, we have that $- \FirstChernClass(\VerticalTangentSpace \ClutchingBundle)(s_\ManifoldElement)$ equals the sum of the weights of the action around $\ManifoldElement$ \cite[Lemma~2.2]{mcduff_topological_2006}.
            
            \begin{remark}
                [Minimal fixed sections]
                \label{rem:minimal-fixed-sections-of-clutching-bundle}
                A \define{minimal} fixed section is a fixed section $s_\ManifoldElement$ at a point $\ManifoldElement$ in the minimal locus of the Hamiltonian $\CircleHam$, i.e. $\CircleHam(\ManifoldElement) = \min(\CircleHam)$.
                Minimal fixed sections are $(\StandardACS, \ClutchingBundleACS)$-holomorphic for a restricted class of almost complex structures $\ClutchingBundleACS$ (see \cite[Definition~2.3]{mcduff_topological_2006} and the preceding text).
                In this setting, we moreover have $\ClutchingTwoForm(u) > \ClutchingTwoForm(s_\ManifoldElement)$ for any minimal fixed section $s_\ManifoldElement$ and any $(\StandardACS, \ClutchingBundleACS)$-holomorphic section $u$ which is not a minimal fixed section \cite[Lemma~3.1]{mcduff_topological_2006}.
            \end{remark}

        \subsubsection{Quantum Seidel map definition}
            \label{sec:quantum-seidel-map-definition}
            
            The objects of focus are $(\StandardACS, \ClutchingBundleACS)$-holomorphic sections of $\ClutchingBundle$, for a suitably regular $\ClutchingBundleACS \in \ACSSpace(\ClutchingBundle)$ which we use throughout this section.
            The moduli space $\ModuliSpace(\SectionClass)$ of $(\StandardACS, \ClutchingBundleACS)$-holomorphic sections which are in $\InterestingSpheres$-equivalence class $\SectionClass$ is a smooth manifold of dimension $2 \dimM + 2 \FirstChernClass(\VerticalTangentSpace \ClutchingBundle)(\SectionClass)$.
            A sequence of such sections $u_r \in \ModuliSpace(\SectionClass)$ with $\ClutchingTwoForm(u_r)$ bounded will have a subsequence which converges to a section with bubbles in the fibres.
            
            For critical points $x^\pm \in \CriticalPoints(\MorseFunction)$, denote by $\ModuliSpace(x^-, x^+; \SectionClass)$ the moduli space of triples $(\gamma^-, \gamma^+, u)$ where $\gamma^- : \IntervalOpenClosed{-\Infinity}{0} \to \Manifold$ and $\gamma^+ : \IntervalClosedOpen{0}{\Infinity} \to \Manifold$ are negative gradient trajectories of $\MorseFunction^\pm$ and $u \in \ModuliSpace(\SectionClass)$ is a section which satisfies $u(\Pole^\pm) = \gamma^\pm(0)$ at the poles.
            Here, the functions $\MorseFunction^\pm : \RealNumbers \times \Manifold \to \RealNumbers$ are $s$-dependent perturbations of $\MorseFunction$, chosen such that the data $(\MorseFunction^\pm, \ClutchingBundleACS)$ satisfies some regularity conditions.
            These regularity conditions ensure that $\ModuliSpace(x^-, x^+; \SectionClass)$ is a smooth manifold of dimension $\MorseIndex(x^-) - \MorseIndex(x^+) + 2 \FirstChernClass(\VerticalTangentSpace \ClutchingBundle)(\SectionClass)$.
            
            Define a degree-$2 \MaslovIndex(\Lifted{\CircleAction})$ chain map $\QuantumSeidelMap_{\Lifted{\CircleAction}} : \CochainComplex^\ArbitraryIndex (\Manifold; \MorseFunction; \NovikovRing) \to \CochainComplex^{\ArbitraryIndex + 2 \MaslovIndex(\Lifted{\CircleAction})} (\Manifold; \MorseFunction; \NovikovRing)$ by
            \begin{equation}
                x^+ \mapsto \sum_{\substack{
                    x^- \in \CriticalPoints(\MorseFunction) \\
                    A \in \InterestingSpheres \\
                    \Dimension \ModuliSpace(x^-, x^+; \SectionClass_{\Lifted{\CircleAction}} + A) = 0
                }} \sum_{(\gamma^\pm, u) \in \ModuliSpace(x^-, x^+; \SectionClass_{\Lifted{\CircleAction}} + A)} \Orientation((\gamma^\pm, u)) \ \NovVariable^A x^- \EndFullStop
            \end{equation}
            The regularity conditions we impose ensure that $\QuantumSeidelMap$ is a chain map so it induces a map on quantum cohomology.
            Moreover, $\QuantumSeidelMap$ intertwines quantum multiplication in $\QuantumCohomology^\ArbitraryIndex(\Manifold)$, giving
            \begin{equation}
                \label{eqn:quantum-intertwining-seidel}
                \QuantumSeidelMap_{\Lifted{\CircleAction}}(x \QuantumProduct y) = x \QuantumProduct \QuantumSeidelMap_{\Lifted{\CircleAction}}(y)
            \end{equation}
            for all $x, y \in \Cohomology^\ArbitraryIndex (\Manifold; \MorseFunction; \NovikovRing)$.
            
            \remark{
                There are two more-or-less equivalent methods to proving this intertwining relation \eqref{eqn:quantum-intertwining-seidel}.
                One approach is to prove the intertwining of the Floer Seidel map with the pair-of-pants product \cite[Proposition~6.3]{seidel_$_1997}, and apply the ring isomorphisms $\PSSmap^\pm$ from $\eqref{eqn:pss-map-domain-codomain}$ to deduce the desired result.
                This is the approach taken by Seidel.
                This method will not extend to the equivariant setup for two reasons: the pair-of-pants product has no equivariant version and the homotopy Seidel constructs to prove the intertwining with the pair-of-pants product involves reparameterising the path $\CircleAction : \Circle \to \HamiltonianSymplectomorphismGroup(\Manifold)$, which cannot be done while maintaining \eqref{eqn:circle-action-on-loop-space} (see \autoref{rem:requirement-for-circle-action-in-equivariant-case}).
                The second approach directly constructs a chain homotopy either side of \eqref{eqn:quantum-intertwining-seidel}.
                While a standard argument, I believe it has not appeared in the literature.
                It is this second approach we extend in \autoref{sec:intertwining-relation-equivariant-quantum-seidel-map}.
                The non-equivariant argument may be derived from the equivariant argument by removing the flowlines in $\InfiniteSphere$.
            }
            
        \subsubsection{Gluing construction}
        
            \label{sec:gluing-construction}
            
            The PSS isomorphism is a ring isomorphism between the Floer cohomology and the quantum cohomology of a weakly monotone closed symplectic manifold constructed in \cite{piunikhin_symplectic_1996}.
            As a map, the PSS isomorphism counts \define{spiked discs}.
            These are maps from the disc to $\Manifold$ which near the boundary act like Floer solutions and which near the centre act like a pseudoholomorphic sphere, together with half-flowlines\footnote{
            i.e. a flowline with domain either $\IntervalClosedOpen{0}{\Infinity}$ or $\IntervalOpenClosed{-\Infinity}{0}$.
        } between a critical point and the centre of the disc, the \define{spikes}.
            To extend the definition of these maps to convex symplectic manifolds, we make the following definition.
            
            \begin{definition}
                The (time-dependent) Hamiltonian $\Ham^0 : \Circle \to \SmoothFunctions(\Manifold)$ has \define{slope zero} if $\Ham^0$ is $\DifferentiableMaps^2$-bounded and, at infinity, $\Ham^0_t(\ConvexCoordMap(y, R)) = h(R)$ for a smooth function $h : \IntervalOpen{R_0}{\Infinity} \to \RealNumbers$ which satisfies $0 < h^\prime < T_\text{min}$, $h^{\prime \prime} < 0$ and $h^\prime \to 0$, where $T_\text{min} \in \IntervalOpenClosed{0}{\Infinity}$ is the minimal Reeb period.
            \end{definition}
            
            While such a Hamiltonian is not linear, it still satisfies \eqref{eqn:hamiltonian-vector-field-at-infinity-is-multiple-of-reeb-vector-field} and a maximum principle at infinity.
            Thus we can define Floer cohomology for a regular choice of Floer data $(\Ham^0, \FamilyACS)$.
            The PSS construction yields a pair of ring isomorphisms between Floer cohomology and quantum cohomology
            \begin{align}
                \label{eqn:pss-map-domain-codomain}
                \PSSmap^- : \FloerCohomology^\ArbitraryIndex(\Manifold; \Ham^0) &\to \QuantumCohomology^\ArbitraryIndex(\Manifold) \EndComma &
                \PSSmap^+ :  \QuantumCohomology^\ArbitraryIndex(\Manifold) &\to \FloerCohomology^\ArbitraryIndex(\Manifold; \Ham^0) 
            \end{align}
            which are mutual inverses \cite[Theorem~37]{ritter_floer_2014}.
            
            Seidel's gluing argument \cite[Section~8]{seidel_$_1997} proves the following diagram is commutative.
            \begin{equation}
                \label{eqn:diagram-relating-quantum-and-floer-seidel-maps}
                \begin{tikzcd}[column sep=0, row sep=large]
                    \QuantumCohomology^\ArbitraryIndex (\Manifold)
                        \arrow[rr, "\QuantumSeidelMap_{\Lifted{\CircleAction}}"]
                        \arrow[d, "\PSSmap^+"', "\cong"] 
                        &
                        & \QuantumCohomology^{\ArbitraryIndex + 2\MaslovIndex(\Lifted{\CircleAction})} (\Manifold)
                    \\
                    \FloerCohomology^\ArbitraryIndex(\Manifold; \Ham^0)
                        \arrow[rd, "\FloerSeidel_{\Lifted{\CircleAction}}"', "\cong"]
                        &
                        & \FloerCohomology^{\ArbitraryIndex + 2\MaslovIndex(\Lifted{\CircleAction})} (\Manifold; \Ham^0)
                        \arrow[u, "\PSSmap^-"', "\cong"]
                    \\
                    & \FloerCohomology^{\ArbitraryIndex + 2\MaslovIndex(\Lifted{\CircleAction})} (\Manifold; \CircleAction\PullBack\Ham^0)
                        \arrow[ru, "\text{continuation map}"']
                        &                  
                \end{tikzcd}
            \end{equation}
            In \cite{ritter_floer_2014}, Ritter used a non-equivariant version of \eqref{eqn:commutative-diagram-for-computing-symplectic-cohomology-using-floer-seidel-map-limit} together with \eqref{eqn:diagram-relating-quantum-and-floer-seidel-maps} to show that if $\CircleAction$ has positive slope, then the direct limit of the direct system
            \begin{equation}
                \QuantumCohomology^\ArbitraryIndex (\Manifold) \xrightarrow{\QuantumSeidelMap_{\Lifted{\CircleAction}}}
                \QuantumCohomology^{\ArbitraryIndex + 2\MaslovIndex(\Lifted{\CircleAction})}  (\Manifold) \xrightarrow{\QuantumSeidelMap_{\Lifted{\CircleAction}}}
                \QuantumCohomology^{\ArbitraryIndex + 4\MaslovIndex(\Lifted{\CircleAction})}  (\Manifold) \xrightarrow{\QuantumSeidelMap_{\Lifted{\CircleAction}}}\cdots
            \end{equation}
            is isomorphic to symplectic cohomology.
            This offers a method to calculate symplectic cohomology because the quantum Seidel map is quantum multiplication by the element $\QuantumSeidelMap_{\Lifted{\CircleAction}} (1)$.
            
    \subsection{Equivariant Quantum cohomology}
    
        \label{sec:equivariant-quantum-cohomology}
    
        \subsubsection{Equivariant Morse cohomology}
        
            \label{sec:equivariant-morse-cohomology}
        
            Let $\rho$ be a smooth circle action on $\Manifold$, and fix a $\rho$-invariant Riemannian metric on $\Manifold$.
            An \define{equivariant Morse function} is a function $\EqntMorseFunction : \InfiniteSphere \times \Manifold \to \RealNumbers$ which is invariant under the free action \eqref{eqn:circle-action-for-borel-homotopy-quotient}, and which extends Morse functions $\MorseFunction^k : \Manifold \to \RealNumbers$ analogously to \autoref{def:extending-floer-data-equivariantly}.
            We assume that, at infinity, the function $\MorseFunction\Eqnt_\infsphereelement (\Argument)$ is increasing in the radial coordinate direction for all $\infsphereelement \in \InfiniteSphere$.
            
            Analogously to the equivariant Floer cohomology construction of \autoref{sec:equivariant-floer-theory}, an \define{equivariant critical point} is an equivalence class $\EquivalenceClass{\infsphereelement, x} \in \InfiniteSphere \times_\Circle \Manifold$ such that $\infsphereelement$ is a critical point of $\spheremorsefunction$ and $x$ is a critical point of $\MorseFunction\Eqnt_\infsphereelement (\Argument)$.
            The index of such a critical point is $\MorseIndex(\infsphereelement; \spheremorsefunction) + \MorseIndex(x; \MorseFunction\Eqnt)$.
            We use the notation $(\spherecriticalpoint_k, x) \in \CriticalPoints(\MorseFunction\Eqnt)$ for equivariant critical points and $\Modulus{c_k, x}$ for their indices.
            
            For a suitably regular equivariant Morse function, the moduli spaces of \define{equivariant negative gradient trajectories} are smooth oriented manifolds.
            \define{Equivariant (Morse) cohomology $\ECohomology^\ArbitraryIndex_\rho (\Manifold)$} is the cohomology of the cochain complex generated over $\Integers$ by the equivariant critical points whose differential counts the equivariant negative gradient trajectories.
            As in \autoref{sec:morse-cohomology}, an orientation of the unstable manifolds must be chosen in order that the count is signed.
            We omit the details.
            The $\Integers [\uformal]$-module actions of \autoref{sec:u-multiplication-shift-invariance} and \autoref{sec:y-shape-module-structure-on-equivariant-floer} may be defined on equivariant Morse cohomology, and we opt for the geometric action $\uformal \CupProduct$ which counts `Y'-shaped graphs.
            
        \subsubsection{Equivariant quantum product}
        
            \label{sec:equivariant-quantum-product}
        
            Let $\rho$ be a symplectic circle action which is linear at infinity.
            We use a `Y'-shaped flowline in $\InfiniteSphere$ for the equivariant quantum product so that it resembles a deformed equivariant cup product.
            In order to construct the moduli space, we need three $s$-dependent perturbations $\spheremorsefunction^-, \spheremorsefunction_1^+, \spheremorsefunction_2^+$ of the standard function on $\InfiniteSphere$ and three $s$-dependent perturbations ${\EqntMorseFunction}^{, -}, {\EqntMorseFunction_1}^{, +}, {\EqntMorseFunction_2}^{, +}$ of the equivariant Morse function on $\Manifold$.
            We also need a regular convex $\InfiniteSphere$-dependent $\SymplecticForm$-compatible almost complex structure $\AlmostComplexStructure\Eqnt$, which is equivariant in the sense that $\AlmostComplexStructure\Eqnt_\infsphereelement = \left(\Derivative \CircleAction_\CircleElement \right) \Inverse \AlmostComplexStructure\Eqnt_{\CircleElement \Inverse \cdot \infsphereelement} \ \Derivative \CircleAction_\CircleElement$ for all $\CircleElement \in \Circle$ and $\infsphereelement \in \InfiniteSphere$.
            
            We consider septuples $(v^-, v_1^+, v_2^+, \gamma^-, \gamma_1^+, \gamma_2^+, u)$, where $v_\ast^\pm$ are negative gradient flowlines of $\spheremorsefunction_\ast^\pm$ satisfying $v_1^+(0) = v_2^+(0) = v^-(0)$ and $\gamma_\ast^\pm$ are negative gradient flowlines of $({\EqntMorseFunction_\ast}^{,\pm})_{v_\ast^\pm} (\Argument)$ which intersect at 0 with a simple (or constant) $\AlmostComplexStructure\Eqnt_{v^-(0)}$-holomorphic sphere $u$ at the points $p_\ast^\pm \in \Projective^1$ (see \autoref{fig:equivariant-quantum-product}).
            When the data is sufficiently regular, the moduli space of ($\Circle$-equivalence classes of) such septuples with limits $(\spherecriticalpoint_{k_\ast^\pm}, x_\ast^\pm) \in \CriticalPoints(\EqntMorseFunction)$ and with $u$ representing $A \in \InterestingSpheres$ is a smooth oriented manifold of dimension
            \begin{equation}
                \Modulus{\spherecriticalpoint_{k^-}, x^-} + 2 \FirstChernClass(A) - \Modulus{\spherecriticalpoint_{k_1^+}, x_1^+} - \Modulus{\spherecriticalpoint_{k_2^+}, x_2^+} \EndFullStop
            \end{equation}
            The \define{equivariant quantum product $\QuantumProduct_\rho$} counts the 0-dimensional moduli spaces.
            
\begin{figure}
	\begin{center}
		\begin{tikzpicture}
			\node[inner sep=0] at (0,0) {\includegraphics[width=10 cm]{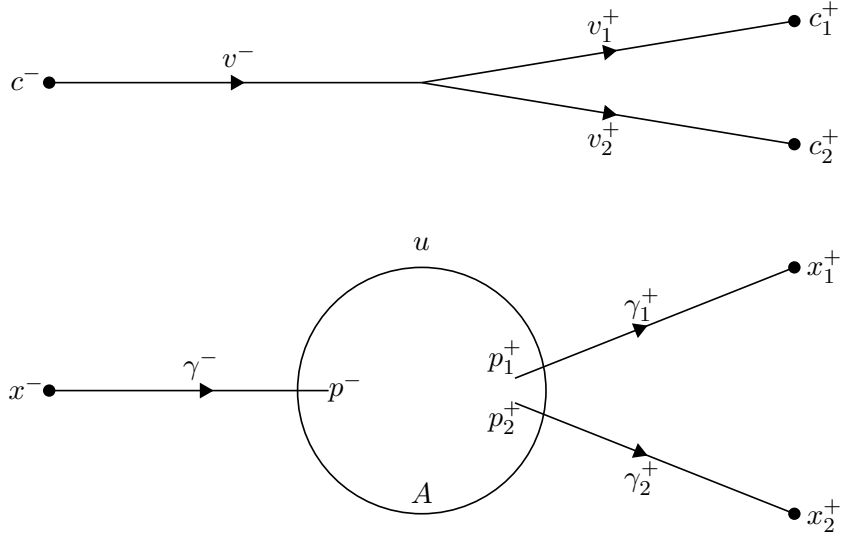}};
			\node at (-2.4,2.8){$v^-$};
			\node at (2.4,3.2){$v_1^+$};
			\node at (2.4,1.7){$v_2^+$};
			\node at (-2.9,-1.3){$\gamma^-$};
			\node at (2.9,-0.5){$\gamma_1^+$};
			\node at (2.9,-2.8){$\gamma_2^+$};
			\node at (-1,-1.6){$p^-$};
			\node at (1.1,-1.2){$p_1^+$};
			\node at (1.1,-2){$p_2^+$};
			\node at (-5.2,2.5){$\spherecriticalpoint^-$};
			\node at (5.3,3.3){$\spherecriticalpoint_1^+$};
			\node at (5.3,1.6){$\spherecriticalpoint_2^+$};
			\node at (-5.2,-1.6){$x^-$};
			\node at (5.3,0){$x_1^+$};
			\node at (5.3,-3.3){$x_2^+$};
			\node at (0,0.3){$u$};
			\node at (0,-3){$A$};
		\end{tikzpicture}
	\end{center}
	\caption{
	    The equivariant quantum product counts equivalence classes of septuples $(v^-, v_1^+, v_2^+, \gamma^-, \gamma_1^+, \gamma_2^+, u)$.
	    The `Y'-shaped graph above maps to $\InfiniteSphere$ while the configuration below the graph maps to $\Manifold$.
    }
	\label{fig:equivariant-quantum-product}
\end{figure}
            
            The \define{equivariant quantum cohomology $\EQuantumCohomology ^\ArbitraryIndex _\rho (\Manifold)$} is the cohomology of the cochain complex
            \begin{equation}
                \label{eqn:equivariant-quantum-cochain-complex}
                \EQuantumCochainComplex^l(\Manifold) = \prod _{k=0} ^\Infinity \ \bigoplus_{                    (\spherecriticalpoint_k, x) \in \CriticalPoints(\MorseFunction\Eqnt)} \NovikovRing^{l - |\spherecriticalpoint_k, x|} \langle (\spherecriticalpoint_k, x) \rangle
            \end{equation}
            with the equivariant Morse differential.
            It is a graded $\NovikovRing [\uformal]$-module for both the algebraic and the geometric $\uformal$-actions.
            The product $\QuantumProduct_\rho$ is a unital, graded-commutative, associative product on $\EQuantumCohomology^\ArbitraryIndex _\rho (\Manifold)$ compatible with the $\NovikovRing$-module structure and the geometric\footnote{
                It is difficult to see any relationship between the algebraic $\Integers [\uformal]$-module structure and the equivariant product, hence our decision to use the geometric module structure in this paper.
            } $\Integers [\uformal]$-module structure.
            All of these properties follow from standard homotopy proofs, as does the independence from all the data chosen.
            
            Via the algebraic $\Integers [\uformal]$-module structure, the equivariant quantum cochain complex \eqref{eqn:equivariant-quantum-cochain-complex} is the graded completed tensor product $\QuantumCochainComplex^\ArbitraryIndex(\Manifold) \GradedCompletedTensorProduct \Integers [\uformal]$.
            By the \define{graded completed tensor product} $A \GradedCompletedTensorProduct \Integers [\uformal]$, where $A$ is a graded $\Integers$-module, we mean the graded $\Integers$-module whose grading-$l$ subgroup is
            \begin{equation}
                \label{eqn:homogeneous-elements-of-graded-completed-tensor-product}
                \left(A \GradedCompletedTensorProduct \Integers [\uformal] \right) ^l = \prod _{k=0} ^\Infinity A^{l - 2k} \Tensor \Integers \cdot \uformal^k \EndFullStop
            \end{equation}
            This is a variant of the completed tensor product used by Zhao \cite[Section~2]{zhao_periodic_2014}, but which is a graded module in the conventional sense.
            
            Since $\Manifold$ is nonnegatively monotone, the equivariant quantum product $\QuantumProduct_\rho$ would be well-defined with the graded completed tensor product replaced by an ordinary tensor product, however the resulting equivariant quantum cohomology would fail to be isomorphic with our equivariant Floer cohomology (c.f. \eqref{eqn:equivariant-floer-complex}).
            
\section{Equivariant quantum Seidel map}

    \label{sec:equivariant-quantum-seidel-map}
        
    Let $\Lifted{\CircleAction}$ be a lift of a linear Hamiltonian circle action of nonnegative slope and $\rho$ a symplectic circle action linear at infinity.
    Assume $\CircleAction$ and $\rho$ commute.
            
    \subsection{Clutching bundle action}
    
        \label{sec:clutching-bundle-circle-action}
        
        The sphere $\Sphere$ has a natural rotation action given by $\CircleElement \cdot (s, t) = (s, t + \CircleElement)$ away from the poles, using the parameterisation \eqref{eqn:cylindrical-coordinates-on-the-sphere}.
        
        The clutching bundle $\ClutchingBundle$ from \autoref{sec:clutching-bundle-construction} admits a smooth circle action given by
        \begin{equation}
            \label{eqn:circle-action-on-clutching-bundle-total-space}
            \left\{
            \begin{array}{lllllll}
                \Hemisphere^- \times \Manifold &\ni & (z, \ManifoldElement) &\mapsto & (\ExponentialNumber^{2 \PiNumber \ImaginaryNumber \CircleElement} z, (\CircleAction \PullBack \rho)_\CircleElement (\ManifoldElement)) &\in &\Hemisphere^- \times \Manifold
                \\
                \Hemisphere^+ \times \Manifold &\ni & (z, \ManifoldElement) &\mapsto & (\ExponentialNumber^{- 2 \PiNumber \ImaginaryNumber \CircleElement} z, \rho_\CircleElement(\ManifoldElement)) &\in &\Hemisphere^+ \times \Manifold
            \end{array}
            \right. \EndComma
        \end{equation}
        which glues correctly along the equator.
        We denote this action by $\ClutchingAction$.
        The projection map is equivariant with respect to the action $\ClutchingAction$ on $\ClutchingBundle$ and to the rotation action on $\Sphere$.
        The poles $\Pole^\pm$ are fixed points of the rotation of $\Sphere$, so $\ClutchingAction$ restricts to a circle action in each of the fibres $\ClutchingBundle_{\Pole^\pm}$: the action on $\ClutchingBundle_{\Pole^+}$ is the action $\rho$ and the action on $\ClutchingBundle_{\Pole^-}$ is the action $\CircleAction \PullBack \rho$.
        Here, we have identified these fibres $\ClutchingBundle_{\Pole^\pm}$ with $\Manifold$ via the inclusion maps $\ClutchingInclusions^\pm$.
            
    \subsection{Equivariant quantum Seidel map definition}
    
        \label{sec:equivariant-quantum-seidel-map-definition}
        
\begin{figure}
	\begin{center}
		\begin{tikzpicture}
			\node[inner sep=0] at (0,0) {\includegraphics[width=10 cm]{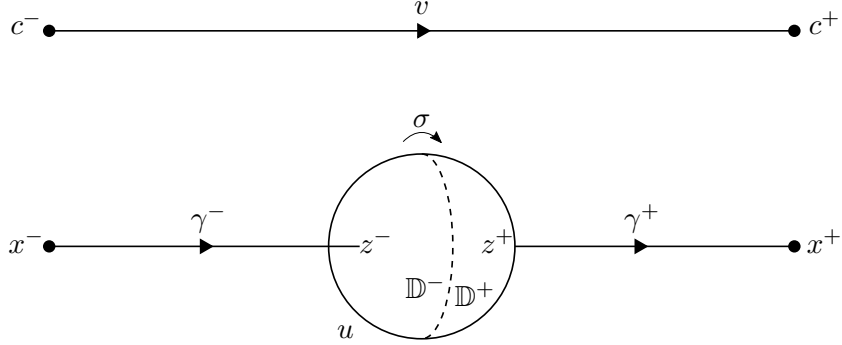}};
			\node at (0,2.3){$v$};
			\node at (-1,-2){$u$};
			\node at (-2.8,-0.5){$\gamma^-$};
			\node at (2.9,-0.5){$\gamma^+$};
			\node at (0,0.8){$\CircleAction$};
			\node at (0.05,-1.4){$\Hemisphere^-$};
			\node at (0.7,-1.5){$\Hemisphere^+$};
			\node at (1,-0.8){$\Pole^+$};
			\node at (-0.6,-0.8){$\Pole^-$};
			\node at (-5.2,2.1){$\spherecriticalpoint^-$};
			\node at (5.3,2.1){$\spherecriticalpoint^+$};
			\node at (-5.2,-0.8){$x^-$};
			\node at (5.3,-0.8){$x^+$};
		\end{tikzpicture}
	\end{center}
	\caption{
	    The equivariant quantum Seidel map counts equivalence classes of quadruples $(v, \gamma^-, \gamma^+, u)$.
	    The map $u$ is a section of the clutching bundle which twists the fibres by $\CircleAction$ when passing from the upper hemisphere $\Hemisphere^-$ to the lower hemisphere $\Hemisphere^+$.
	    The flowlines $\gamma^\pm$ map to the manifold $\Manifold$, which is identified with the fibres of the clutching bundle over the poles $\Pole^\pm$.
	}
	\label{fig:equivariant-quantum-seidel-map}
\end{figure}
        
        The equivariant quantum Seidel map is a version of the quantum Seidel map $\QuantumSeidelMap$ which is equivariant with respect to the circle action $\ClutchingAction$ of \autoref{sec:clutching-bundle-circle-action}.
        Since the action $\ClutchingAction$ restricts to different actions on the two fibres over the poles, the equivariant quantum Seidel map will map between the equivariant quantum cohomology for these two different circle actions.

        Let $\EqntMorseFunction{}^{,\pm}$ be equivariant Morse-Smale functions for the circle actions in the corresponding to the fibre $\ClutchingBundle_{\Pole^\pm}$, so $\EqntMorseFunction{}^{,+}$ is equivariant with respect to the action $\rho$ and $\EqntMorseFunction{}^{,-}$ is equivariant with respect to the action $\CircleAction \PullBack \rho$.
        
        We require $s$-dependent perturbations $\EqntMorseFunction_s{}^{,\pm}$ of the equivariant Morse data and an $\InfiniteSphere$-dependent almost complex structure $\ClutchingBundleACS\Eqnt$ which is equivariant in the sense of \autoref{sec:equivariant-quantum-product} with respect to $\ClutchingAction$.
        This almost complex structure should have similar properties to the non-equivariant $\ClutchingBundleACS$, so that $\Derivative \ClutchingProjection$ is holomorphic and the fibrewise restriction $\ClutchingBundleACS\Eqnt \RestrictedTo{\VerticalTangentSpace \ClutchingBundle}$ is a $\ClutchingSymplecticBilinearForm$-compatible almost complex structure for all $\infsphereelement \in \InfiniteSphere$, and $\ClutchingBundleACS\Eqnt \RestrictedTo{\VerticalTangentSpace \ClutchingBundle}$ is convex and $\ClutchingBundleACS\Eqnt$ has the form \eqref{eqn:clutching-acs-at-infinity} at all points in a region at infinity.
        
        The regularity conditions will guarantee the moduli spaces below are smooth manifolds which in dimensions 0 and 1 compactify without bubbling.
        
        The equivariant quantum Seidel map counts ($\Circle$-equivalence classes of) quadruples $(v, \gamma^-, \gamma^+, u)$, where $v$ is a flowline in $\InfiniteSphere$, the curves $\gamma^+ : \IntervalClosedOpen{0}{\Infinity} \to \Manifold$ and $\gamma^- : \IntervalOpenClosed{-\Infinity}{0} \to \Manifold$ are equivariant $\MorseFunction_{ s}^{\eqnt,\pm}(v(s),\Argument)$-flowlines and $u$ is a $\ClutchingBundleACS\Eqnt_{v(0)}$-holomorphic section satisfying $u(\Pole^\pm) = \gamma^\pm (0)$ (see \autoref{fig:equivariant-quantum-seidel-map}).
        It is a degree-$2 \MaslovIndex(\Lifted{\CircleAction})$ $\NovikovRing$-module homomorphism $\EQuantumSeidelMap_{\Lifted{\CircleAction}} : \EQuantumCohomology_\rho^\ArbitraryIndex (\Manifold) \to \EQuantumCohomology_{\CircleAction \PullBack \rho}^{\ArbitraryIndex + 2 \MaslovIndex(\Lifted{\CircleAction})} (\Manifold)$.
        A standard homotopy argument shows that $\EQuantumSeidelMap_{\Lifted{\CircleAction}}$ commutes with the geometric $\Integers [\uformal]$-module structure.
            
    \subsection{Equivariant gluing}
    
        \label{sec:equivariant-gluing}
        
        The results of \autoref{sec:gluing-construction} extend to the equivariant setup.
        The equivariant PSS maps are the $\NovikovRing [\uformal]$-module isomorphisms
        \begin{equation}
            \begin{aligned}
                \EPSSmap^- &: \EFloerCohomology^\ArbitraryIndex_\rho(\Manifold; \Ham^{\eqnt,0}) \to \EQuantumCohomology^\ArbitraryIndex_\rho (\Manifold) \\
                \EPSSmap^+  &: \EQuantumCohomology^\ArbitraryIndex_\rho (\Manifold) \to \EFloerCohomology^\ArbitraryIndex_\rho(\Manifold; \Ham^{\eqnt,0})
            \end{aligned}
        \end{equation}
        which count equivariant spiked discs.
        Here, $\Ham^{\eqnt,0}$ is an equivariant Hamiltonian of slope zero with respect to the action $\rho$.
        The equivariant version of \eqref{eqn:diagram-relating-quantum-and-floer-seidel-maps} is the following commutative diagram.
        \begin{equation}
            \label{eqn:diagram-relating-quantum-and-floer-seidel-maps-equivariant-case}
            \begin{tikzcd}[row sep=2em, column sep=-1em]
                \EQuantumCohomology^\ArbitraryIndex_\rho (\Manifold)
                    \arrow[rr, "\EQuantumSeidelMap_{\Lifted{\CircleAction}}"]
                    \arrow[d, "\EPSSmap^+"', "\cong"] 
                    &
                    & \EQuantumCohomology^{\ArbitraryIndex + 2\MaslovIndex(\Lifted{\CircleAction})}_{\CircleAction \PullBack \rho} (\Manifold)
                \\
                \EFloerCohomology^\ArbitraryIndex_\rho(\Manifold; \EqntHam{}^{,0})
                    \arrow[rd, "\EFloerSeidel_{\Lifted{\CircleAction}}"', "\cong"]
                    &
                    & \EFloerCohomology ^{\ArbitraryIndex + 2\MaslovIndex(\Lifted{\CircleAction})} _{\CircleAction \PullBack \rho} (\Manifold; \EqntHam{}^{,0})
                    \arrow[u, "\EPSSmap^-"', "\cong"]
                \\
                & \EFloerCohomology ^{\ArbitraryIndex + 2\MaslovIndex(\Lifted{\CircleAction})} _{\CircleAction \PullBack \rho} (\Manifold; \CircleAction\PullBack\EqntHam{}^{,0})
                    \arrow[ru, "\text{continuation map}"']
                    &                  
            \end{tikzcd}
        \end{equation}
        If $\CircleAction$ has positive slope, then Ritter's argument \eqref{eqn:commutative-diagram-for-computing-symplectic-cohomology-using-floer-seidel-map-limit} combined with \eqref{eqn:diagram-relating-quantum-and-floer-seidel-maps-equivariant-case} implies that the direct limit of the direct system
        \begin{equation}
            \EQuantumCohomology^\ArbitraryIndex_{\rho} (\Manifold) \xrightarrow{\EQuantumSeidelMap_{\Lifted{\CircleAction}}}
            \EQuantumCohomology^{\ArbitraryIndex + 2\MaslovIndex(\Lifted{\CircleAction})} _{\CircleAction \PullBack \rho} (\Manifold) \xrightarrow{\EQuantumSeidelMap_{\Lifted{\CircleAction}}}
            \EQuantumCohomology^{\ArbitraryIndex + 4\MaslovIndex(\Lifted{\CircleAction})} _{(\CircleAction^2) \PullBack \rho} (\Manifold) \xrightarrow{\EQuantumSeidelMap_{\Lifted{\CircleAction}}}\cdots
        \end{equation}
        is isomorphic to equivariant symplectic cohomology $\ESymplecticCohomology_\rho^\ArbitraryIndex(\Manifold)$.
            
    \subsection{Intertwining relation}
    
        \label{sec:intertwining-relation-equivariant-quantum-seidel-map}
        
        The intertwining result \eqref{eqn:quantum-intertwining-seidel} does not hold in our equivariant setup.
        Instead, we have the formula 
        \begin{equation}
            \label{eqn:quantum-intertwining-seidel-equivariant}
            \EQuantumSeidelMap_{\Lifted{\CircleAction}} (\bm{x} \underset{\rho}{\QuantumProduct} ((\ClutchingInclusions^+)\PullBack \bm{\alpha}))
            -
            \EQuantumSeidelMap_{\Lifted{\CircleAction}} (\bm{x}) \underset{\CircleAction \PullBack \rho}{\QuantumProduct} ((\ClutchingInclusions^-)\PullBack \bm{\alpha})
            =
            \uformal \CupProduct \EQuantumSeidelMap_{\Lifted{\CircleAction}, \bm{\alpha}}(\bm{x}) \EndFullStop
        \end{equation}
        In this equation, $\bm{\alpha} \in \ECohomology^\ArbitraryIndex_{\ClutchingAction}(\ClutchingBundle)$ is an equivariant cohomology class of the clutching bundle, and the map $\EQuantumSeidelMap_{\Lifted{\CircleAction}, \bm{\alpha}}$ is the $\bm{\alpha}$-weighted equivariant quantum Seidel map defined in the following section.
        The proof of \eqref{eqn:quantum-intertwining-seidel-equivariant} takes up the rest of this section.
        
        \subsubsection{Weighted equivariant quantum Seidel map}
        
            \label{sec:error-term-definition}
            
\begin{figure}
	\begin{center}
		\begin{tikzpicture}
			\node[inner sep=0] at (0,0) {\includegraphics[width=10 cm]{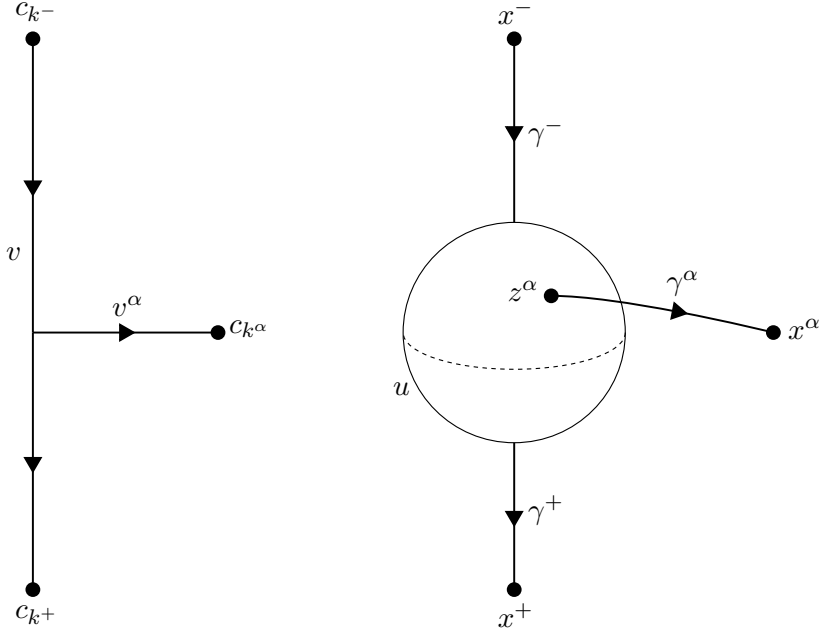}};
			\node at (-5.1,0.8){$v$};
			\node at (-3.6,0.1){$v^\alpha$};
			\node at (1.9,-2.6){$\gamma^+$};
			\node at (1.9,2.4){$\gamma^-$};
			\node at (3.7,0.4){$\gamma^\alpha$};
			\node at (1.6,0.3){$z^\alpha$};
			\node at (0,-1){$u$};
			\node at (-2,-0.2){$\spherecriticalpoint_{k^\alpha}$};
			\node at (5.3,-0.2){$x^\alpha$};
			\node at (-4.8,-4){$\spherecriticalpoint_{k^+}$};
			\node at (1.5,-4){$x^+$};
			\node at (-4.8,4){$\spherecriticalpoint_{k^-}$};
			\node at (1.5,4){$x^-$};
		\end{tikzpicture}
	\end{center}
	\caption{
	    The weighted equivariant quantum Seidel map counts equivalence classes of septuples $(v, \gamma^+, \gamma^-, u; v^\alpha, \gamma^\alpha, z^\alpha)$.
	    The section $u$ of the clutching bundle intersects the flowline $\gamma^\alpha$ over the point $z^\alpha \in \Sphere$.
	}
	\label{fig:weighted-equivariant-quantum-seidel-map}
\end{figure}

            Fix an invariant Riemannian metric on the clutching bundle $\ClutchingBundle$ for the action $\ClutchingAction$, and let $\MorseFunction\Eqnt_\ClutchingBundle$ be an equivariant Morse-Smale function for this action.
            
            Take a regular choice of the data $\MorseFunction^{\eqnt , \pm}_s$ and $\ClutchingBundleACS\Eqnt$ from \autoref{sec:equivariant-quantum-seidel-map-definition}.
            We use an $s$-dependent perturbation $\MorseFunction^{\eqnt, \alpha}_{\ClutchingBundle, s}$ of $\MorseFunction\Eqnt_\ClutchingBundle$ on $\IntervalClosedOpen{0}{\Infinity}$.
            We will consider septuples $(v,\allowbreak \gamma^+,\allowbreak \gamma^-,\allowbreak u;\allowbreak v^\alpha,\allowbreak \gamma^\alpha,\allowbreak z^\alpha)$ where $(v, \gamma^+, \gamma^-, u)$ is a quadruple from \autoref{sec:equivariant-quantum-seidel-map-definition}, $z^\alpha \in \Sphere$ is a point in the sphere and $(v^\alpha, \gamma^\alpha) : \IntervalClosedOpen{0}{\Infinity} \to \InfiniteSphere \times \ClutchingBundle$ is an equivariant $\MorseFunction^{\eqnt, \alpha}_{\ClutchingBundle, s}$ flowline satisfying $v^\alpha(0) = v(0)$ and $\gamma^\alpha(0) = u(z^\alpha)$ (see \autoref{fig:weighted-equivariant-quantum-seidel-map}).
            Given any equivariant critical point $\bm{\alpha} = (\spherecriticalpoint_{k^\alpha}, x^\alpha)$ of $\MorseFunction\Eqnt_\ClutchingBundle$, together with equivariant critical points $\bm{x}^\pm = (\spherecriticalpoint_{k^\pm}, x^\pm) \in \CriticalPoints (\MorseFunction\Eqnt)$ and a class $A \in \InterestingSpheres$, denote by $\ModuliSpace(\bm{x}^-, \bm{x}^+, A; \bm{\alpha})$ the moduli space of $\Circle$-equivalence classes of septuples as above with the obvious limits and with $u$ of class $\SectionClass_{\Lifted{\CircleAction}} + A$.
            For regular perturbations, these moduli spaces are smooth oriented manifolds of dimension
            \begin{equation}
                \Dimension \ModuliSpace(\bm{x}^-, \bm{x}^+, A; \bm{\alpha}) = \Modulus{\bm{x}^-} - \Modulus{\bm{x}^+} - \Modulus{\bm{\alpha}} + 2 \FirstChernClass(A) + 2 \EndFullStop
            \end{equation}
            The $+2$ comes from the $2$-dimensional freedom of the point $z^\alpha \in \Sphere$.
            Moreover, we can assume $z^\alpha \notin \Set{\Pole^\pm}$ for moduli spaces of dimension 0 and 1 by imposing further regularity conditions on $\MorseFunction^{\eqnt, \alpha}_{\ClutchingBundle, s}$.
            
            Counting this moduli space yields a map \begin{equation}
                \EQuantumSeidelMap_{\Lifted{\CircleAction}, \bm{\alpha}} : \QuantumCochainComplex _\rho ^\ArbitraryIndex (\Manifold) \to \QuantumCochainComplex _{\CircleAction \PullBack \rho} ^{\ArbitraryIndex + 2 \MaslovIndex (\Lifted{\CircleAction}) + \Modulus{\bm{\alpha}} - 2} (\Manifold) \EndFullStop
            \end{equation}
            Our definition may be immediately extended linearly to any $\bm{\alpha} \in \ECochainComplex^\ArbitraryIndex_{\ClutchingAction} (\ClutchingBundle)$.
            The 1-dimensional moduli spaces are compactified by any of the flowlines breaking since there is no bubbling by regularity.
            This yields the equation
            \begin{equation}
                d \  \EQuantumSeidelMap_{\Lifted{\CircleAction}, \bm{\alpha}} (\bm{x}) = \EQuantumSeidelMap_{\Lifted{\CircleAction}, \bm{\alpha}} \  d(\bm{x}) + (-1)^{\Modulus{\bm{x}}} \EQuantumSeidelMap_{\Lifted{\CircleAction}, d \bm{\alpha}}(\bm{x}) \EndFullStop
            \end{equation}
            Thus for closed Morse cochains $\bm{\alpha}$, the map $\EQuantumSeidelMap_{\Lifted{\CircleAction}, \bm{\alpha}}$ is a chain map.
            
            \begin{remark}
                [Interpretation]
                \label{remark:degree-two-weighted-map}
                Let $k^\alpha = 0$ and $|x^\alpha| = 2$.
                For any quadruple $(v,\allowbreak \gamma^+,\allowbreak \gamma^-,\allowbreak u)$ from \autoref{sec:equivariant-quantum-seidel-map-definition}, the flowline $v^\alpha$ generically flows to the minimum $\spherecriticalpoint_0 = \spherecriticalpoint_{k^\alpha}$.
                Therefore, the count of flowlines $\gamma^\alpha$ with $\gamma^\alpha(0) = u(z^\alpha)$ for some $z^\alpha \in \Sphere$ and $\gamma^\alpha(\Infinity) = x^\alpha$ recovers the number $\EquivalenceClass{x^\alpha} (u \PushForward \FundamentalClass{\Sphere})$.
                This is the evaluation of the degree-2 cohomology class $\EquivalenceClass{x^\alpha}$ on the degree-2 homology class $(u \PushForward \FundamentalClass{\Sphere})$.
                As such, the map $\EQuantumSeidelMap_{\Lifted{\CircleAction}, \bm{\alpha}}$ is a weighted version of $\EQuantumSeidelMap_{\Lifted{\CircleAction}}$ under which any section $u$ has weight $\EquivalenceClass{x^\alpha} (u \PushForward \FundamentalClass{\Sphere})$.
            \end{remark}
            
        \subsubsection{1-dimensional moduli space}
        \label{sec:one-dimensional-moduli-space-for-intertwining-relation}
        
\begin{figure}
	\begin{center}
		\begin{tikzpicture}
			\node[inner sep=0] at (0,0) {\includegraphics[width=10 cm]{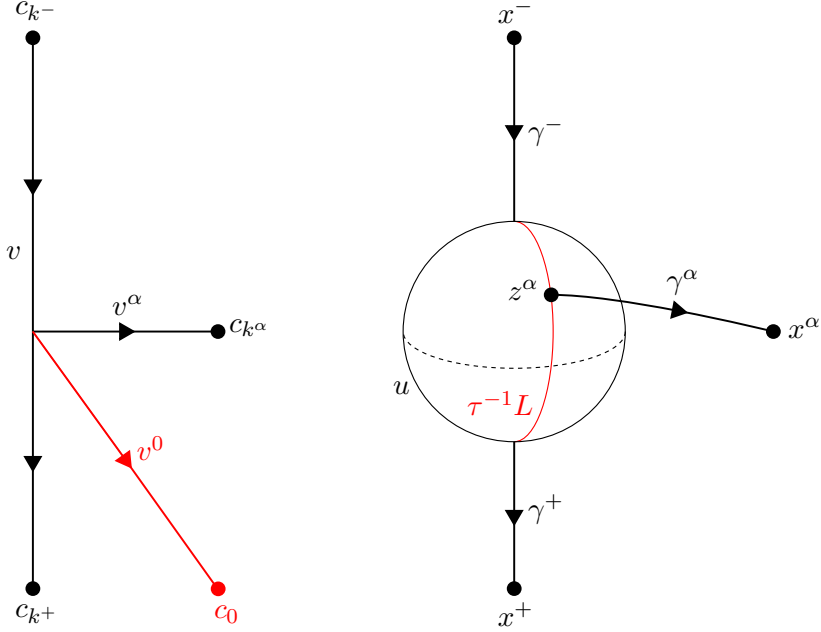}};
			\node at (-5.1,0.8){$v$};
			\node at (-3.6,0.1){$v^\alpha$};
			\node at (1.9,-2.6){$\gamma^+$};
			\node at (1.9,2.4){$\gamma^-$};
			\node at (3.7,0.4){$\gamma^\alpha$};
			\node at (1.6,0.3){$z^\alpha$};
			\node at (0,-1){$u$};
			\node at (-2,-0.2){$\spherecriticalpoint_{k^\alpha}$};
			\node at (5.3,-0.2){$x^\alpha$};
			\node at (-4.8,-4){$\spherecriticalpoint_{k^+}$};
			\node at (1.5,-4){$x^+$};
			\node at (-4.8,4){$\spherecriticalpoint_{k^-}$};
			\node at (1.5,4){$x^-$};
			\node at (-2.3,-4){$\textcolor{red}{\spherecriticalpoint_0}$};
			\node at (1.3,-1.2){$\textcolor{red}{\tau \Inverse \LongitudeLine}$};
			\node at (-3.3,-1.8){$\textcolor{red}{v^0}$};
		\end{tikzpicture}
	\end{center}
	\caption{
	    The map $K_{\bm{\alpha}}$ counts equivalence classes of tuples like the weighted equivariant quantum Seidel map in \autoref{fig:weighted-equivariant-quantum-seidel-map}, however with an additional flowline $v^0$ in $\InfiniteSphere$.
	    Moreover, we restrict to tuples which satisfy $z^\alpha \in \textcolor{red}{\tau \Inverse \LongitudeLine}$, where $\textcolor{red}{\tau = \spheretrivialise_0 (v^0(\Infinity))}$ is determined by the additional flowline in $\InfiniteSphere$.
	}
	\label{fig:one-dimensional-moduli-space-for-equivariant-quantum-seidel-map-intertwining-formula}
\end{figure}
            
            Fix the line of longitude $\LongitudeLine = \RealNumbers_{> 0} \Intersection \Hemisphere^\pm \subset \Sphere$, which does not include the poles.
            One way to derive the commutativity of the quantum product $\QuantumProduct$ and the quantum Seidel map $\QuantumSeidelMap_{\Lifted{\CircleAction}}$ is to take a non-equivariant version of the moduli space from \autoref{sec:error-term-definition} in which $z^\alpha \in \LongitudeLine$.
            The 1-dimensional moduli space is compactified by breaking one of the flowlines, or allowing a bubble over the pole when $z^\alpha \to \Pole^\pm$.
            If $x^\alpha$ is closed, this yields a chain homotopy between $(\ClutchingInclusions^-)\PullBack x^\alpha \QuantumProduct \QuantumSeidelMap _{\Lifted{\CircleAction}} (\Argument)$ and $\QuantumSeidelMap_{\Lifted{\CircleAction}}  ((\ClutchingInclusions^+)\PullBack x^\alpha \QuantumProduct \Argument)$.
            
            The intersection condition $z^\alpha \in \LongitudeLine$ does not transform correctly under the $\Circle$-action on the equivariant moduli space, however.
            To rectify this, take a further $s$-dependent perturbation $\spheremorsefunction^0_s$ of the Morse function $\spheremorsefunction$ on $\InfiniteSphere$.
            We consider octuples $(v, \gamma^+, \gamma^-, u; v^\alpha, \gamma^\alpha, z^\alpha; v^0)$ which extend the septuples in the construction of \autoref{sec:error-term-definition} (see \autoref{fig:one-dimensional-moduli-space-for-equivariant-quantum-seidel-map-intertwining-formula}).
            The map $v^0 : \IntervalClosedOpen{0}{\Infinity} \to \InfiniteSphere$ is a $\spheremorsefunction^0_s$-flowline which satisfies the intersection $v^0(0) = v(0)$ and the limit $v^0(\Infinity) \in \spherecriticalpoint_0$.
            We moreover impose $z^\alpha \in \spheretrivialise_0 (v^0(\Infinity)) \Inverse \cdot \LongitudeLine$, where the action $\cdot$ on $\Sphere$ is the rotation action defined in \autoref{sec:clutching-bundle-circle-action}.
            Notice that this condition is preserved for all $(v^0(\Infinity), z^\alpha) \in \spherecriticalpoint_0 \times \Sphere$ by the natural circle action on $\spherecriticalpoint_0 \times \Sphere$ because $\spherecriticalpoint_0 \subset \InfiniteSphere$ has the inverse action in accordance with \eqref{eqn:circle-action-for-borel-homotopy-quotient}.
            The effect of this construction is that the 2-dimensional freedom of the point $z^\alpha$ has been reduced to a 1-dimensional freedom `along $\LongitudeLine$'; the second dimension of freedom has been absorbed into the $\Circle$-action.
            
            \begin{remark}
                [Regularity]
                We impose regularity conditions on the data so that the moduli spaces of the above octuples are smooth manifolds, as well as the moduli spaces of the above octuples without the condition $z^\alpha \in \spheretrivialise_0 (v^0(\Infinity)) \Inverse \cdot \LongitudeLine$ and with $v^0(\Infinity) \in \spherecriticalpoint_k$ for any $k$.
                Moreover, we use regularity conditions to avoid unnecessary intersections over the poles by asking that the projection $[v(0), z^\alpha] : \ModuliSpace \to \InfiniteSphere \times_\Circle \Sphere$ intersects $\InfiniteComplexProjectiveSpace \times \Set{\Pole^\pm}$ transversally for all the above moduli spaces.
                We impose further regularity conditions to ensure that we control the behaviour of configurations with bubbles just as in Seidel's argument \cite[Section~7]{seidel_$_1997}.
            \end{remark}
            
            We quotient by the free $\Circle$-action to get the moduli space of $\Circle$-equivalence classes of above octuples with the obvious constraints, which we denote by $\ModuliSpace_\tau (\bm{x}^-, \bm{x}^+, A; \bm{\alpha})$.
            It is a smooth oriented manifold of dimension
            \begin{equation}
                \Dimension \ModuliSpace_\tau(\bm{x}^-, \bm{x}^+, A; \bm{\alpha}) = \Modulus{\bm{x}^-} - \Modulus{\bm{x}^+} - \Modulus{\bm{\alpha}} + 2 \FirstChernClass(A) + 1 \EndFullStop
            \end{equation}
            Denote by $K_{\bm{\alpha}}$ the map $\EQuantumCochainComplex _\rho ^\ArbitraryIndex (\Manifold) \to \EQuantumCochainComplex _{\CircleAction \PullBack \rho} ^{\ArbitraryIndex + 2 \MaslovIndex (\Lifted{\CircleAction}) + \Modulus{\bm{\alpha}} - 1} (\Manifold)$ which counts these moduli spaces.
            
        \subsubsection{Boundary of the moduli space}
            
            By standard compactification and gluing arguments, the 1-dimensional moduli spaces $\ModuliSpace_\tau (\bm{x}^-, \bm{x}^+, A; \bm{\alpha})$ will have a boundary composed of broken flowlines and bubbled spheres.
            The sum of these boundary components will be zero.
            Subject to further homotopies, this sum yields the desired relation \eqref{eqn:quantum-intertwining-seidel-equivariant}.
            We list the various components of the boundary below, and detail how they contribute to the sum.
            
            \begin{description}
                \item
                    [$(v\RestrictedTo{\IntervalOpenClosed{-\Infinity}{0}}, \gamma^-)$ breaking]
                        We get a contribution $d \ K_{\bm{\alpha}}(\bm{x}^+)$.

                \item  
                    [$(v\RestrictedTo{\IntervalClosedOpen{0}{\Infinity}}, \gamma^+)$ breaking]
                        We get a contribution $K_{\bm{\alpha}} (d\bm{x}^+)$.
                        
                \item
                    [$(v^\alpha, \gamma^\alpha)$ breaking]
                        We get a contribution $(-1)^{|\bm{x}^+|} K_{d \bm{\alpha}} (\bm{x}^+)$.
                        If $\bm{\alpha}$ is a closed cochain, then this contribution vanishes.
                        
                \item
                    [$v^0$ breaking]
                        As a consequence of the regularity conditions, the only possible breaking of the flowline $v^0$ will be to the critical point $\spherecriticalpoint_1 \in \CriticalPoints(\spheremorsefunction)$.
                        Consider such a breaking, but before we have taken the quotient by the $\Circle$-action.
                        The broken flowline consists of a flowline $v^0$ whose limit is $v^0(\Infinity) = y \in \spherecriticalpoint_1$ and a second unparameterised $\spheremorsefunction$-flowline from $y$ to $\tau \in \Circle \cong \spherecriticalpoint_0$, the identification $\Circle \cong \spherecriticalpoint_0$ given by the map $\spheretrivialise_0$.
                        In such a configuration, the point $z^\alpha$ cannot be either of the poles because of the regularity conditions.
                        As a consequence, the point $z^\alpha$ uniquely determines $\tau \in \Circle$ via $z^\alpha \in \tau \Inverse \cdot \LongitudeLine$.
                        It may be explicitly shown that there is a unique $\spheremorsefunction$-flowline whose limits are any specified points of $\spherecriticalpoint_0$ and $\spherecriticalpoint_1$.
                        It follows that the flowline from $y$ to $\tau$ may be omitted without loss of generality.
                        
                        A standard homotopy argument will separate the flowlines $v^0$ and $v^\alpha$ in $\InfiniteSphere$ (see \autoref{fig:flowline-breaking-to-c1}).
                        Therefore when $v^0$ breaks, we get a contribution of $- \uformal \CupProduct \EQuantumSeidelMap_{\Lifted{\CircleAction}, \bm{\alpha}} (\bm{x}^+)$ to the sum, up to chain homotopy.
                        
\begin{figure}
	\begin{center}
		\begin{tikzpicture}
			\node[inner sep=0] at (0,0) {\includegraphics[width=10 cm]{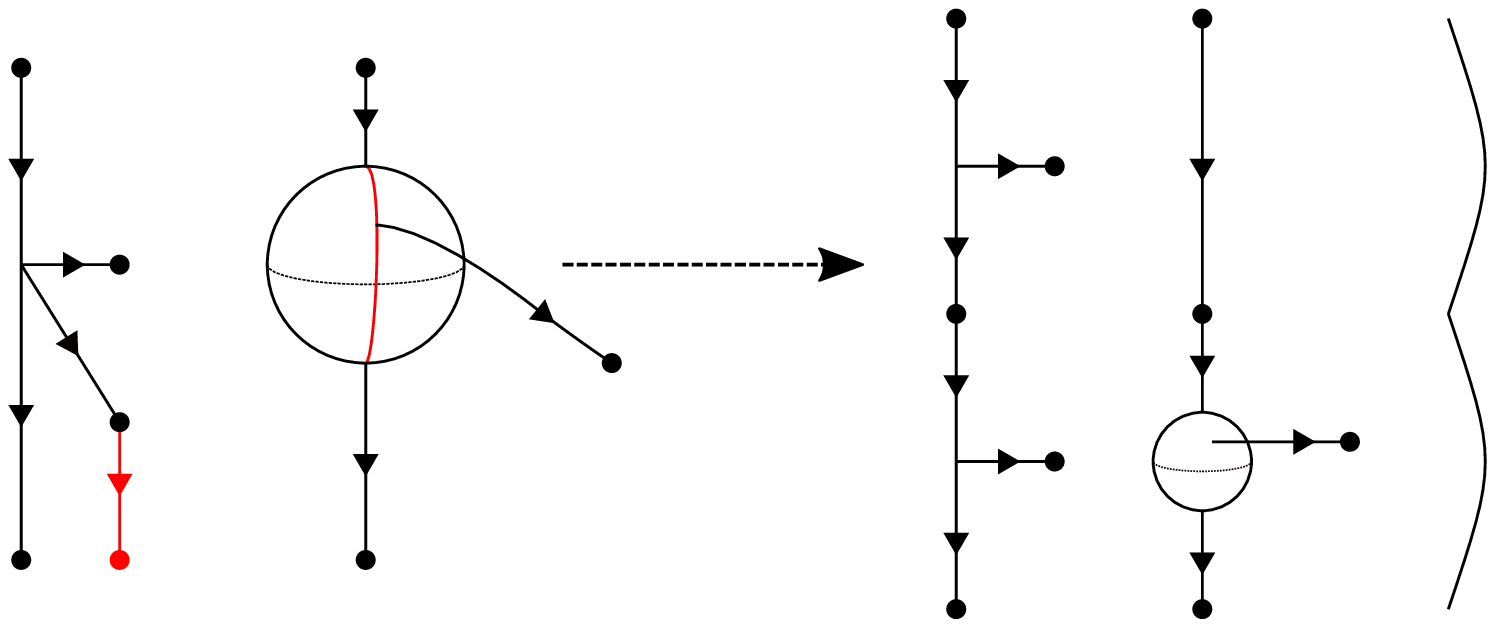}};
			\node at (-3.8,0.3){$\spherecriticalpoint_{k^\alpha}$};
			\node at (-0.6,-0.3){$x^\alpha$};
			\node at (-4.9,-1.9){$\spherecriticalpoint_{k^+}$};
			\node at (-2.6,1.9){$x^-$};
			\node at (-4.9,1.9){$\spherecriticalpoint_{k^-}$};
			\node at (-2.6,-1.9){$x^+$};
			\node at (-4.2,-1.9){$\textcolor{red}{\spherecriticalpoint_0}$};
			\node at (-3.9,-0.7){$\spherecriticalpoint_1$};
			
			\node at (-0.3,0.6){\tiny homotopy argument};
			
			\node at (2.4,-1){$\spherecriticalpoint_{k^\alpha}$};
			\node at (4.4,-0.8){$x^\alpha$};
			\node at (1.4,-2.2){$\spherecriticalpoint_{k^+}$};
			\node at (3.1,-2.2){$x^+$};
			\node at (1.4,2.2){$\spherecriticalpoint_{k^-}$};
			\node at (3.1,2.2){$x^-$};
			\node at (2.3,1){$\spherecriticalpoint_1$};
			
			\node at (5.4,1){$\uformal \CupProduct$};
			\node at (5.7,-1.1){$\EQuantumSeidelMap_{\Lifted{\CircleAction}, \bm{\alpha}}$};
		\end{tikzpicture}
	\end{center}
	\caption{
	    When the flowline $v^0$ breaks, we can use a homotopy argument to split the moduli space into multiplication by $\uformal$ and the weighted equivariant Seidel map.
	    The \textcolor{red}{flowline from $\spherecriticalpoint_1$ to $\spherecriticalpoint_0$} contains redundant information so we omit it.
	}
	\label{fig:flowline-breaking-to-c1}
\end{figure}
                        
                \item
                    [$z^\alpha \to \Pole^\pm$ with bubbling]
                        Due to the regularity conditions, the only possible bubbling configuration is a single bubble in the fibre over one of the two poles $\Pole^\pm$ with the flowline to $x^\alpha$ starting within the fibre.
                        In this configuration, we get $z^\alpha \in \Set{ \Pole^\pm }$, so the condition $z^\alpha \in \spheretrivialise_0 (v^0(\Infinity)) \Inverse \cdot \LongitudeLine$ is automatically satisfied (for any value of $v^0(\Infinity))$.
                        As such, the flowline $v^0$ may be omitted without losing any information because it no longer constrains the point $z^\alpha$.
                        
                        We will treat the case of a bubble over the pole $\Pole^+$.
                        If the bubble is of class $B \in \InterestingSpheres$, then the section $u$ is of class $\SectionClass_{\Lifted{\CircleAction}} + A - B$.
                        We use a standard homotopy argument to insert a broken flowline between the section and the bubble, and also to extend the flowline to $\bm{\alpha}$ so it breaks into a half-flowline to $(\ClutchingInclusions^+)\PullBack \bm{\alpha}$ and an equivariant functorial flowline\footnote{
                            Our construction of the pullback maps is a variant of \cite[Section~1.3]{rot_functoriality_2014}, in which we allow $s$-dependent perturbations of the Morse data instead of perturbing the function.
                            Let $\Manifold^\pm$ be manifolds with Morse-Smale functions $\MorseFunction^\pm : \Manifold^\pm \to \RealNumbers$ (and metrics and orientation data).
                            Given any $\phi : \Manifold^- \to \Manifold^+$, a \define{functorial flowline} is a pair of half-flowlines $(\gamma^- : \IntervalOpenClosed{-\Infinity}{0} \to \Manifold^-, \allowbreak \gamma^+ : \IntervalClosedOpen{0}{\Infinity} \to \Manifold^+)$ of $s$-dependent perturbations of the functions $\MorseFunction^\pm$ which satisfy $\phi (\gamma^-(0)) = \gamma^+(0)$.
                            When $\phi$ is an immersion or a submersion, it follows from standard arguments that the moduli spaces of functorial flowlines between critical points $x^\pm$ is a smooth manifold of dimension $\MorseIndex(x^-; \MorseFunction^-) - \MorseIndex(x^+; \MorseFunction^+)$.
                            The map which counts these moduli spaces is a chain map and is denoted $\phi \PullBack$.
                            The equivariant pullback maps are defined analogously.
                        } from there to $\bm{\alpha}$ (see \autoref{fig:bubbling-over-south-pole}).
                        The result is a contribution of $\EQuantumSeidelMap_{\Lifted{\CircleAction}} (\bm{x}^+ \QuantumProduct_\rho ((\ClutchingInclusions^+)\PullBack \bm{\alpha}))$ to the sum, up to chain homotopy.
                        
                        Analogously, a bubble over the pole $\Pole^-$ yields a contribution
                        \begin{equation}
                           -\EQuantumSeidelMap_{\Lifted{\CircleAction}} (\bm{x}^+) \QuantumProduct_{\CircleAction \PullBack \rho} ((\ClutchingInclusions^-)\PullBack \bm{\alpha})
                        \end{equation}
                        to the sum, up to chain homotopy.
                        
\begin{figure}
	\begin{center}
		\begin{tikzpicture}
			\node[inner sep=0] at (0,0) {\includegraphics[width=10 cm]{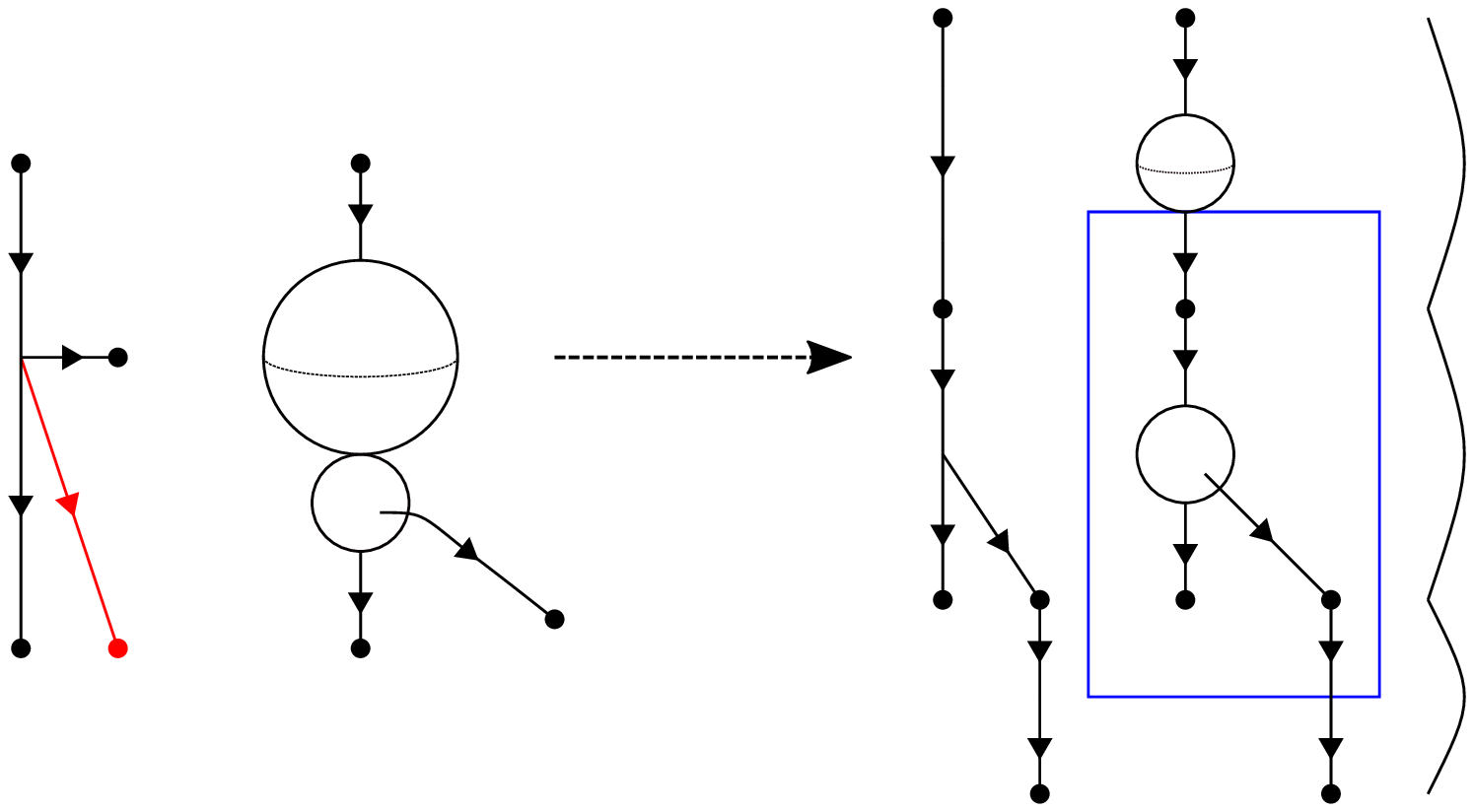}};
			\node at (-3.8,0.3){$\spherecriticalpoint_{k^\alpha}$};
			\node at (-0.9,-1.4){$x^\alpha$};
			\node at (-4.9,-1.9){$\spherecriticalpoint_{k^+}$};
			\node at (-2.6,1.9){$x^-$};
			\node at (-4.9,1.9){$\spherecriticalpoint_{k^-}$};
			\node at (-2.6,-1.9){$x^+$};
			\node at (-4.2,-1.9){$\textcolor{red}{\spherecriticalpoint_0}$};
			
			\node at (-0.3,0.6){\tiny homotopy argument};
			
			\node at (2.1,-2.9){$\spherecriticalpoint_{k^\alpha}$};
			\node at (4.1,-2.9){$x^\alpha$};
			\node at (1.5,-1.6){$\spherecriticalpoint_{k^+}$};
			\node at (3.1,-1.6){$x^+$};
			\node at (1.5,2.9){$\spherecriticalpoint_{k^-}$};
			\node at (3.1,2.9){$x^-$};
			
			\node at (4.1,1.2){\tiny \textcolor{blue}{fibre}};
			
			\node at (5.5,1.7){$\EQuantumSeidelMap_{\Lifted{\CircleAction}}$};
			\node at (5.2,-0.5){$\underset{\rho}{\QuantumProduct}$};
			\node at (5.6,-2){$(\ClutchingInclusions^+) \PullBack \bm{\alpha}$};
		\end{tikzpicture}
	\end{center}
	\caption{
	    When the section bubbles over the south pole, we can use a homotopy argument to get an equivariant quantum product in the \textcolor{blue}{fibre above the south pole} followed by the map $\EQuantumSeidelMap_{\Lifted{\CircleAction}}$.
	    The \textcolor{red}{flowline to $\spherecriticalpoint_0$} is redundant, so we remove it.
	}
	\label{fig:bubbling-over-south-pole}
\end{figure}
            \end{description}
            
            Summing these contributions, we get the equation
            \begin{equation}
                \label{eqn:intertwining-formula-seen-as-from-homotopy-directly}
                -
                \uformal \CupProduct \EQuantumSeidelMap_{\Lifted{\CircleAction}, \bm{\alpha}}(\bm{x})
                +
                \EQuantumSeidelMap_{\Lifted{\CircleAction}} (\bm{x} \underset{\rho}{\QuantumProduct} ((\ClutchingInclusions^+)\PullBack \bm{\alpha}))
                -
                \EQuantumSeidelMap_{\Lifted{\CircleAction}} (\bm{x}) \underset{\CircleAction \PullBack \rho}{\QuantumProduct} ((\ClutchingInclusions^-)\PullBack \bm{\alpha})
                = 0
            \end{equation}
            for $\bm{x} \in \EQuantumCohomology^\ArbitraryIndex_\rho(\Manifold)$, which rearranges to give \eqref{eqn:quantum-intertwining-seidel-equivariant} as desired.
            
\section{Examples}

    \label{sec:example-computations}

    \subsection{Complex plane}
    
        \label{sec:example-of-complex-plane}
    
        The complex plane $\ComplexNumbers$ is an exact open symplectic manifold whose symplectic form is given by $\SymplecticForm = \ExteriorDerivative x \wedge \ExteriorDerivative y$ at points $z = x + \ImaginaryNumber y \in \ComplexNumbers$ .
        The contact form $\ContactForm = \PiNumber \ExteriorDerivative t$ on $\Circle$ and the isomorphism $\ConvexCoordMap : \Circle \times \IntervalClosedOpen{1}{\Infinity} \to \ComplexNumbers$ given by $(t, R) \mapsto \sqrt{R} \ExponentialNumber^{2 \PiNumber \ImaginaryNumber t}$ gives $\ComplexNumbers$ the structure of a convex symplectic manifold.
        The set of Reeb periods is $\ReebPeriods = \PiNumber \Integers$.
        The smooth circle action $\CircleAction : \Circle \times \ComplexNumbers \to \ComplexNumbers$ given by $\CircleElement \cdot z = \ExponentialNumber^{2 \PiNumber \ImaginaryNumber \CircleElement} z$ has Hamiltonian $\CircleHam (z) = \PiNumber \Modulus{z}^2$ and is linear of slope $\PiNumber$.
        The action has a unique fixed point $0_\ComplexNumbers \in \ComplexNumbers$.
        The lift of $\CircleAction$ to $\ContractibleLoopSpaceWithFillings{\ComplexNumbers}$ is unique because $\ComplexNumbers$ is contractible.
        This lift $\Lifted{\CircleAction}$ fixes the point $(\Circle \mapsto 0_\ComplexNumbers, \Disc \mapsto 0_\ComplexNumbers)$ and has Maslov index $\MaslovIndex(\Lifted{\CircleAction}) = 2$.
        
        \newcommand{\LinearHam}[1]{\Ham_{\text{linear}}^{#1}}
        Let $\LinearHam{\lambda} : \ComplexNumbers \to \RealNumbers$ be the autonomous linear Hamiltonian $z \mapsto \lambda \Modulus{z}^2$.
        If $\lambda \notin \ReebPeriods$, then the unique Hamiltonian orbit of $\LinearHam{\lambda}$ is the constant loop at $0_\ComplexNumbers$ and has Conley-Zehnder index $-2 \Floor{\frac{\lambda}{\PiNumber}}$ \cite[Section~3.2]{oancea_survey_2004}.
        Therefore, for $\lambda \notin \ReebPeriods$, the Floer cochain complex is
        \begin{equation}
            \FloerC^\ArbitraryIndex(\ComplexNumbers; \LinearHam{\lambda}) \cong \left\{
                \begin{array}{ll}
                    \Integers   & \text{if $\ArbitraryIndex = -2 \Floor{\frac{\slope}{\PiNumber}}$,} \\
                    0           & \text{otherwise.}
                \end{array}
            \right.
        \end{equation}
        The symplectic cohomology of $\ComplexNumbers$ thus vanishes.
        In contrast, the equivariant cohomology $\ESymplecticCohomology^\ArbitraryIndex_{\Identity_\ComplexNumbers}(\ComplexNumbers)$ is isomorphic to $\RationalNumbers [\uformal, \uformal\Inverse]$ as a $\Integers [\uformal]$-module for the geometric and the algebraic module structures.
        We derive this isomorphism in the course of our discussion below (see \eqref{eqn:equivariant-symplectic-cohomology-of-complex-plane}).
        With the equivariant Seidel map, which is an isomorphism, we get further $\Integers [\uformal]$-module isomorphisms $\ESymplecticCohomology^\ArbitraryIndex_{\CircleAction^r}(\ComplexNumbers) \cong \RationalNumbers [\uformal, \uformal\Inverse]$ for $r \in \Integers$.
        
        Give $\ComplexNumbers$ the standard Riemannian metric, which is $\CircleAction$-invariant.
        The Hamiltonian function $\CircleHam$ is Morse-Smale with respect to this metric and has exactly one critical point: the fixed point\footnote{Give $\UnstableManifold(0_\ComplexNumbers) = \Set{0_\ComplexNumbers}$ the orientation $+1$.} $0_\ComplexNumbers$ of Morse index 0.
        Since $\CircleHam$ is invariant, the function $\CircleHam\Eqnt : \InfiniteSphere \times \ComplexNumbers \to \RealNumbers$ given by $(\infsphereelement, z) \mapsto \CircleHam(z)$ is a regular equivariant Morse function on $\ComplexNumbers$.
        Thus the Morse cohomology of $\ComplexNumbers$ is $\Cohomology^\ArbitraryIndex(\ComplexNumbers; \CircleHam) \cong \Integers \langle 0_\ComplexNumbers \rangle$ and the equivariant Morse cohomology is $\ECohomology^\ArbitraryIndex_{\CircleAction^r}(\ComplexNumbers; \CircleHam\Eqnt) \cong \Integers [\uformal] \langle 0_\ComplexNumbers \rangle$ for any $r \in \Integers$.
        Since the Novikov ring is $\NovikovRing = \Integers$, the quantum cohomology and the equivariant quantum cohomology are simply the Morse cohomology and the equivariant Morse cohomology respectively.
        
        We compute the equivariant quantum Seidel map for the lifted action $\Lifted{\CircleAction}$.
        We can do this for the underlying actions $\rho = (\CircleAction^r) \PullBack \Identity_\ComplexNumbers = \CircleAction^{-r}$, where $r$ is nonnegative and  $\Identity_\ComplexNumbers$ is the identity action on $\ComplexNumbers$.
        
        \begin{theorem}
            \label{thm:equivariant-quantum-seidel-map-for-complex-plane}
            Let $r \ge 0$.
            The equivariant quantum Seidel map 
            \begin{equation}
                \EQuantumSeidelMap_{\Lifted{\CircleAction}} : \ECohomology^\ArbitraryIndex_{\CircleAction^{-r}}(\ComplexNumbers; \CircleHam\Eqnt) \to \ECohomology^{\ArbitraryIndex + 2}_{\CircleAction^{-(r+1)}}(\ComplexNumbers; \CircleHam\Eqnt)
            \end{equation}
            is the $\Integers [\uformal]$-linear extension of $0_\ComplexNumbers \mapsto (r + 1) \uformal \ 0_\ComplexNumbers$.
        \end{theorem}
        
        Rather than finding the equivariant quantum Seidel map directly, we will appeal to \eqref{eqn:diagram-relating-quantum-and-floer-seidel-maps-equivariant-case} and opt to find the continuation map instead.
        We will use the sequence of Hamiltonians defined by Zhao in \cite[Section~8.1]{zhao_periodic_2014}, which we outline below.
        
        \newcommand{\QuadraticHam}[1]{H_{\operatorname{quadratic}}^{#1}}
        For each nonnegative integer $s \in \Integers_{\ge 0}$, let $\QuadraticHam{s \PiNumber + 1} : \Circle \times \ComplexNumbers \to \RealNumbers$ be a time-dependent Hamiltonian function such that:
        \begin{itemize}
            \item On $\Modulus{z}^2 < 1$, the function is negative, achieves its minimum at $0_\ComplexNumbers$ and is Morse with exactly one critical point at $0_\ComplexNumbers$;
            \item On $1 < \Modulus{z}^2 < s \PiNumber + 2$, the function $\QuadraticHam{s \PiNumber + 1}$ equals $\frac{1}{2}(\Modulus{z}^2 - 1)^2$ plus a small time-dependent perturbation around $\Modulus{z}^2 = j \PiNumber + 1$ for $j = 1, \ldots, s$; and
            \item On $\Modulus{z}^2 > s \PiNumber + 2$, the function $\QuadraticHam{s \PiNumber + 1}$ is linear\footnote{
                We slightly change $\QuadraticHam{s \PiNumber + 1}$ near $\Modulus{z}^2 = s \PiNumber + 2$ so that the function is smooth.
            } of slope $s \PiNumber + 1$.
        \end{itemize}
        The Hamiltonian orbits of $\QuadraticHam{s \PiNumber + 1}$ occur only at $0_\ComplexNumbers$ and at $R = \Modulus{z}^2 = j \PiNumber + 1$.
        Modulo the perturbations, the slope of $\QuadraticHam{s \PiNumber + 1}$ when $R = j \PiNumber + 1$ is
        \begin{equation}
            \LiebnitzDerivative{\left(\QuadraticHam{s \PiNumber + 1}\right)}{R} = \LiebnitzDerivative{}{R} {\frac{1}{2}(\Modulus{z}^2 - 1)^2} =  R - 1 = j \PiNumber \EndFullStop
        \end{equation}
        Thus the slope at $R = j \PiNumber + 1$ is the Reeb period $j \PiNumber \in \ReebPeriods$.
        If we hadn't perturbed $\QuadraticHam{s \PiNumber + 1}$ in this region, we would get a $\Circle$-family of Hamiltonian orbits corresponding to the Reeb orbit of period $j \PiNumber$, as per \eqref{eqn:hamiltonian-vector-field-at-infinity-is-multiple-of-reeb-vector-field}.
        Instead, as a result of Zhao's perturbation, we get two Hamiltonian orbits, which we denote by $x_{2j-1}$ and $x_{2j}$.
        Denote by $x_0$ the constant Hamiltonian orbit at $0_\ComplexNumbers$.
        The orbit $x_l$ has degree $-l$ for all $0 \le l \le 2s$.
        
        With this choice of Hamiltonian, the equivariant Floer cochain complex (with respect to the trivial action $\Identity_\ComplexNumbers$) is given by
        \begin{equation}
            \label{eqn:equivariant-floer-cochain-complex-for-complex-plane-zhao-hamiltonians}
            \EFloerC^\ArbitraryIndex_{\Identity_\ComplexNumbers}(\ComplexNumbers; \QuadraticHam{s \PiNumber + 1}) = \bigoplus_{k \ge 0} \bigoplus_{j = 0}^{2s} \Integers \langle (\spherecriticalpoint_k, x_j) \rangle \EndFullStop
        \end{equation}
        
        \begin{lemma}
            [{\cite[Section~8.1]{zhao_periodic_2014}}]
            \label{lemma:zhao-chain-complex-maps-for-complex-plane}
            There exist choices of all remaining data such that the differential of \eqref{eqn:equivariant-floer-cochain-complex-for-complex-plane-zhao-hamiltonians} is given by\footnote{
                Zhao derived the equation $d (\spherecriticalpoint_k, x_{2j-1}) = (\spherecriticalpoint_k, x_{2j-2}) + j (\spherecriticalpoint_{k+1}, x_{2j})$, which has different signs to \eqref{eqn:differential-on-equivariant-floer-cohomology-complex-plane}.
                Her result changes to \eqref{eqn:differential-on-equivariant-floer-cohomology-complex-plane} once we apply the rule $\uformal \mapsto - \uformal$ to account for our different conventions, as in \autoref{rem:diagonal-action-in-literature-for-equivariant-floer-theory}.
            }
            \begin{equation}
                \label{eqn:differential-on-equivariant-floer-cohomology-complex-plane}
                d (\spherecriticalpoint_k, x_{2j-1}) = (\spherecriticalpoint_k, x_{2j-2}) - j (\spherecriticalpoint_{k+1}, x_{2j}), \qquad d (\spherecriticalpoint_k, x_{2j}) = 0
            \end{equation}
            and the continuation map $\kappa_s : \EFloerC^\ArbitraryIndex_{\Identity_\ComplexNumbers}(\ComplexNumbers; \QuadraticHam{s \PiNumber + 1}) \to \EFloerC^\ArbitraryIndex_{\Identity_\ComplexNumbers}(\ComplexNumbers; \QuadraticHam{(s+1) \PiNumber + 1})$ is the inclusion map on the cochain complex.
        \end{lemma}
        
        Using this explicit cochain complex, we deduce that the inclusion\footnote{
            The equation $\uformal \CupProduct (\spherecriticalpoint_k, x_{2s}) = (\spherecriticalpoint_{k + 1}, x_{2s})$ holds in $\EFloerC^\ArbitraryIndex _{\Identity_\ComplexNumbers} (\ComplexNumbers; \QuadraticHam{s \PiNumber + 1})$ for degree reasons, so the inclusion is a $\Integers [\uformal]$-module map for the geometric module structure.
        } map $\Integers [\uformal] \langle x_{2s} \rangle \cong \bigoplus_k \Integers \langle (\spherecriticalpoint_k, x_{2s}) \rangle \Inclusion \EFloerC^\ArbitraryIndex_{\Identity_\ComplexNumbers}(\ComplexNumbers; \QuadraticHam{s \PiNumber + 1})$ induces an isomorphism on cohomology.
        Moreover, with respect to this isomorphism, the continuation map $\kappa_s$ is the map $x_{2s} \mapsto (s+1)\uformal \ x_{2(s+1)}$.
        Explicitly, we have
        \begin{equation}
            \begin{tikzcd}
                \EFloerCohomology^\ArbitraryIndex_{\Identity_\ComplexNumbers}(\ComplexNumbers; \QuadraticHam{s \PiNumber + 1})
                \arrow[r, "\kappa_s"]
                & \EFloerCohomology^\ArbitraryIndex_{\Identity_\ComplexNumbers}(\ComplexNumbers; \QuadraticHam{(s+1) \PiNumber + 1})
                \\
                \Integers [\uformal] \langle x_{2s} \rangle
                \arrow[u, "\cong"]
                \arrow[r, "x_{2s} \mapsto (s+1)\uformal \ x_{2(s+1)}"']
                & \Integers [\uformal] \langle x_{2(s+1)} \rangle
                \arrow[u, "\cong"]
            \end{tikzcd}
        \end{equation}
        so that the map $\kappa_s$ is really multiplication by $(s+1)\uformal$.
        
        To compute the equivariant symplectic cohomology of $\ComplexNumbers$, it is enough to consider the direct limits of the continuation maps $\kappa_s$ because the slopes of $\QuadraticHam{s \PiNumber + 1}$ are arbitrarily large.
        The map $\kappa_s$ is multiplication by $(s+1)\uformal$, so it contributes to ``allowing division'' by $(s+1)\uformal$ in the direct limit.
        Thus we get the isomorphism
        \begin{equation}
            \label{eqn:equivariant-symplectic-cohomology-of-complex-plane}
            \ESymplecticCohomology^\ArbitraryIndex_{\Identity_\ComplexNumbers}(\ComplexNumbers) \cong \RationalNumbers [\uformal, \uformal\Inverse] \EndFullStop
        \end{equation}
        
        \begin{proof}[Proof of \autoref{thm:equivariant-quantum-seidel-map-for-complex-plane}]
            Fix $r \ge 0$.
            Consider the following diagram, which is a combination of \eqref{eqn:equivariant-floer-seidel-map-commutes-with-equivariant-continuation-map} and \eqref{eqn:diagram-relating-quantum-and-floer-seidel-maps-equivariant-case}.
            The top square commutes on the cochain complexes whereas the bottom square commutes on cohomology.
            \begin{equation}
                \label{eqn:commutative-diagram-for-computation-of-equivariant-quantum-seidel-map-for-complex-plane}
                \begin{tikzcd}
                    \EFloerC^{\ArbitraryIndex - 2r}_{\CircleAction^0}(\ComplexNumbers; \QuadraticHam{r \PiNumber + 1}) \arrow[d, "\cong"', "\EFloerSeidel_{\Lifted{\CircleAction}^{r+1}}"] \arrow[r, "\kappa_r"] & \EFloerC^{\ArbitraryIndex - 2r}_{\CircleAction^0}(\ComplexNumbers; \QuadraticHam{(r + 1) \PiNumber + 1}) \arrow[d, "\cong"', "\EFloerSeidel_{\Lifted{\CircleAction}^{r+1}}"]  \\
                    \EFloerC^{\ArbitraryIndex + 2}_{ \CircleAction^{-(r+1)}} (\ComplexNumbers;(\CircleAction^{r+1} )\PullBack\QuadraticHam{r \PiNumber + 1}) \arrow[r, "(\CircleAction^{r+1})\PullBack \kappa_r"]                & \EFloerC^{\ArbitraryIndex + 2} _{ \CircleAction^{-(r+1)}} (\ComplexNumbers;(\CircleAction^{r+1} )\PullBack\QuadraticHam{(r+1) \PiNumber + 1}) \arrow[dd, "\EPSSmap^-_{\CircleAction^{-(r+1)}}"] \\
                    \EFloerC^{\ArbitraryIndex} _{\CircleAction^{-r}} (\ComplexNumbers; (\CircleAction^r )\PullBack\QuadraticHam{r \PiNumber + 1}) \arrow[u, "\EFloerSeidel_{\Lifted{\CircleAction}}"', "\cong"]                  &                     \\
                    \EQuantumCochainComplex^\ArbitraryIndex_{\CircleAction^{-r}}(\ComplexNumbers) \arrow[r, "\EQuantumSeidelMap_{\Lifted{\CircleAction}}"] \arrow[u, "\EPSSmap^+_{\CircleAction^{-r}}"'] & \EQuantumCochainComplex^{\ArbitraryIndex+2}_{\CircleAction^{-(r+1)}}(\ComplexNumbers)                 
                \end{tikzcd}
            \end{equation}
            Here, we use the subscript on $\EPSSmap^\pm$ to record the underlying circle action so we can distinguish the different maps.
            Moreover, we use the identity $(\CircleAction^r) \PullBack \Identity_\ComplexNumbers = \CircleAction^{-r}$ to simplify notation.
            
            An inspection of Zhao's explicit perturbation finds that there is\footnote{This $\epsilon$ will depend on the choice of coherent orientation that was made in \autoref{lemma:zhao-chain-complex-maps-for-complex-plane}.} $\epsilon \in \Set{\pm 1}$ such that $\EPSSmap^+_{\CircleAction^{-s}}(0_\ComplexNumbers) = \epsilon \ (\CircleAction^s) \PullBack x_{2s}$ for all $s \ge 0$ because the asymptotic behaviour of a spiked disc is the same for all these maps (see \cite[Appendix~B]{ritter_topological_2013} for a characterisation of the coherent orientation).
            
            Thus, in cohomology, the element $0_\ComplexNumbers$ is mapped in \eqref{eqn:commutative-diagram-for-computation-of-equivariant-quantum-seidel-map-for-complex-plane} as below.
            \begin{equation}
                \label{eqn:mapping-element-according-to-commutative-diagram-for-complex-plane-computation}
                \begin{tikzcd}
                \epsilon \ x_{2r} \arrow[r, maps to]                     & \epsilon \ (r+1) \uformal \ x_{2(r+1)} \arrow[d, maps to]         \\
                \epsilon \ (\CircleAction^{r+1})\PullBack x_{2r} \arrow[r, maps to] \arrow[u, maps to] & \epsilon \ (r+1) \uformal \ (\CircleAction^{r+1})\PullBack x_{2(r+1)} \arrow[dd, maps to] \\
                \epsilon \ (\CircleAction^r)\PullBack x_{2r} \arrow[u, maps to]                    &                              \\
                0_\ComplexNumbers \arrow[u, maps to] \arrow[r, maps to]  & \epsilon^2 \ (r+1) \uformal \ 0_\ComplexNumbers                          
                \end{tikzcd}
            \end{equation}
            Thus we have $0_\ComplexNumbers \mapsto (r+1) \uformal \ 0_\ComplexNumbers$ as desired.
        \end{proof}
        
        This result generalises to $\ComplexNumbers^\dimM$ and the action $\CircleElement \cdot (z_1, \ldots, z_n) = (\ExponentialNumber^{2 \PiNumber \ImaginaryNumber \CircleElement} z_1, \ldots, \ExponentialNumber^{2 \PiNumber \ImaginaryNumber \CircleElement} z_n)$.
        This action $\CircleAction$ also has exactly one fixed point, $0_{\ComplexNumbers^\dimM}$.
        There exist analogous data to describe the equivariant cohomology with $0_{\ComplexNumbers^\dimM}$ the unique minimal critical point.
        Since $\ComplexNumbers^\dimM$ is exact and equivariantly contractible, the equivariant quantum product is trivial.
        
        \begin{theorem}
            \label{thm:equivariant-quantum-seidel-map-complex-vector-space}
            Let $r \ge 0$.
            The equivariant quantum Seidel map $\EQuantumSeidelMap_{\Lifted{\CircleAction}} : \ECohomology^\ArbitraryIndex_{\CircleAction^{-r}}(\ComplexNumbers^\dimM) \to \ECohomology^{\ArbitraryIndex + 2n}_{\CircleAction^{-(r+1)}}(\ComplexNumbers^\dimM)$ is the $\Integers [\uformal]$-linear extension of the assignment $0_{\ComplexNumbers^\dimM} \mapsto ((r + 1) \uformal)^\dimM \ 0_{\ComplexNumbers^\dimM}$.
        \end{theorem}
        
        \begin{proof}
            A similar approach to the calculation for $\ComplexNumbers$, using the spectral sequences from \cite[Corollary~7.2]{mclean_mckay_2018} to deduce the differential, yields the desired formula, however only up to sign.
            Instead, we will derive the $\ComplexNumbers^\dimM$ case directly from \autoref{thm:equivariant-quantum-seidel-map-for-complex-plane}.
            
            By \autoref{rem:minimal-fixed-sections-of-clutching-bundle}, the only holomorphic section of the clutching bundle is the minimal fixed section at $0_{\ComplexNumbers^\dimM}$.
            It follows that $\EQuantumSeidelMap_{\Lifted{\CircleAction}} (0_{\ComplexNumbers^\dimM})$ may be characterised by intersecting equivariant flowlines with the fixed point $0_{\ComplexNumbers^\dimM}$.
            Thus $\EQuantumSeidelMap_{\Lifted{\CircleAction}} (0_{\ComplexNumbers^\dimM}) = A^{0_{\ComplexNumbers^\dimM}}_{\CircleAction^{-(r+1)}}(0_{\ComplexNumbers^\dimM})$, where we define the map $A$ below.
            The theorem immediately follows from \eqref{eqn:intersection-with-fixed-point-map-compatible-with-product}, which allows us to express $A^{0_{\ComplexNumbers^\dimM}}_{\CircleAction^{-(r+1)}}(0_{\ComplexNumbers^\dimM})$ as the $\dimM$-th power of $A^{0_{\ComplexNumbers}}_{\CircleAction^{-(r+1)}}(0_{\ComplexNumbers})$, and the computation in $\ComplexNumbers$ from \autoref{thm:equivariant-quantum-seidel-map-for-complex-plane}.
            
            Let $X$ be a manifold with a smooth circle action $\phi$ and fixed point $p \in X$.
            Equip $X$ with an equivariant Morse function.
            The map $A^p_\phi : \ECohomology^\ArbitraryIndex_\phi(X) \to \ECohomology^{\ArbitraryIndex + \Dimension(X)}_\phi(X)$ counts equivariant $s$-dependent negative gradient trajectories that intersect $p$ at $s=0$.
            Notice that, since $p$ is a fixed point, this intersection condition is independent of the representative of the equivariant trajectory.
            
            Given two such manifolds $X$ and $Y$ with circle actions $\phi_X$ and $\phi_Y$ and fixed points $p_X$ and $p_Y$ respectively, consider the following diagram.
            \begin{equation}
                \label{eqn:diagram-for-fibre-product-of-product}
                \begin{tikzcd}
                    \InfiniteSphere \times_\Circle (X \times Y) \arrow[d, "\pi_Y"] \arrow[r, "\pi_X"] & \InfiniteSphere \times_\Circle X \arrow[d] \\
                    \InfiniteSphere \times_\Circle Y \arrow[r]                     & \InfiniteSphere / \Circle          
                \end{tikzcd}
            \end{equation}
            Standard homotopies yield the equation
            \begin{equation}
                \label{eqn:intersection-with-fixed-point-map-compatible-with-product}
                A^{(p_X, p_Y)}_{(\phi_X, \phi_Y)} \ComposedWith (\pi_X \PullBack \CupProduct \pi_Y \PullBack) = (\pi_X \PullBack A^{p_X}_{\phi_X}) \CupProduct (\pi_Y \PullBack A^{p_Y}_{\phi_Y}) \EndFullStop
            \end{equation}
        \end{proof}
    
    \subsection{Projective space}
    
        \label{sec:example-of-projective-space}
    
        The complex projective space $\Projective^\dimM$ is a closed monotone K\"{a}hler manifold with the Fubini-Study symplectic form $\SymplecticForm\FubiniStudy$ and the Fubini-Study metric.
        Its Novikov ring is $\NovikovRing = \Integers [\NovVariable, \NovVariable\Inverse]$, where $\NovVariable$ is a formal variable of degree $2(\dimM + 1)$.
        The standard Morse function on $\Projective^\dimM$ is the function $\MorseFunction_{\Projective^\dimM} ([z_0:\cdots:z_\dimM]) = \sum_{k=0}^\dimM k \Modulus{z_k}^2$.
        The critical points of $\MorseFunction_{\Projective^\dimM}$ are the unit vectors\footnote{
            Give $\UnstableManifold(e_k)$ the orientation that comes naturally from the complex structure.
        } $e_k$, each with Morse index $\MorseIndex(e_k; \MorseFunction_{\Projective^\dimM}) = 2k$.
        The Hamiltonian circle action $\CircleAction$ given by $\CircleElement \cdot [z_0: z_1:\cdots:z_\dimM] = [z_0: \ExponentialNumber^{2 \PiNumber \ImaginaryNumber \CircleElement} z_1 : \cdots : \ExponentialNumber^{2 \PiNumber \ImaginaryNumber \CircleElement} z_\dimM]$ preserves the metric and the function $\MorseFunction_{\Projective^\dimM}$.
        As per \autoref{sec:example-of-complex-plane}, this means we can form a canonical equivariant Morse function $\MorseFunction_{\Projective^\dimM}\Eqnt$ from $\MorseFunction_{\Projective^\dimM}$.
        
        We use the Morse functions $\MorseFunction_{\Projective^\dimM}$ and $\MorseFunction_{\Projective_\dimM}\Eqnt$ to describe a basis for the various cohomologies below.
        In \eqref{eqn:projective-space-all-cohomologies-and-products}, we give module isomorphisms\footnote{
            As for $\ComplexNumbers$, we have $\uformal \CupProduct (\spherecriticalpoint_l, e_k) = (\spherecriticalpoint_{l+1}, e_k)$, so we use the shorthand $\uformal^l e_k$ for the equivariant critical point $(\spherecriticalpoint_l, e_k)$.
        } to each of the cohomologies, and describe the corresponding products\footnote{
            Projective space is \define{equivariantly formal}, which means there is a $\Integers [\uformal]$-module isomorphism $\ECohomology_{\CircleAction^{-r}}^\ArbitraryIndex(\Projective^\dimM) \cong \Integers [\uformal] \Tensor \Cohomology^\ArbitraryIndex(\Projective^\dimM)$.
            The decomposition is not natural, however.
            As we see in \eqref{eqn:projective-space-all-cohomologies-and-products}, even the equivariant cup product does not respect this decomposition.
        } (which are determined by the given information).
        The global minimum $e_0$ is the unit for all products.
        In the following, $r$ is any nonnegative integer.
        \begin{equation}
            \label{eqn:projective-space-all-cohomologies-and-products}
            \begin{aligned}
                \Cohomology^\ArbitraryIndex(\Projective^\dimM) \cong \Integers \langle e_k \rangle_{k=0}^\dimM,&
                    & e_1 \CupProduct e_k &= \left\{ \begin{array}{l}
                        e_{k+1}  \\
                        0
                    \end{array}\right. &
                    \begin{array}{l}
                        0 < k < \dimM,  \\
                        k = \dimM. 
                    \end{array}\\
                \ECohomology_{\CircleAction^{-r}}^\ArbitraryIndex(\Projective^\dimM) \cong \Integers [\uformal] \langle e_k \rangle_{k=0}^\dimM,&
                    & e_1 \underset{-r}{\CupProduct} e_k &= \left\{ \begin{array}{l}
                        e_{k+1} - r \uformal e_k \\
                        - r \uformal e_\dimM
                    \end{array}\right. &
                    \begin{array}{l}
                        0 < k < \dimM,  \\
                        k = \dimM. 
                    \end{array}\\
                \QuantumCohomology^\ArbitraryIndex(\Projective^\dimM) \cong \NovikovRing \langle e_k \rangle_{k=0}^\dimM,&
                    & e_1 \QuantumProduct e_k &= \left\{ \begin{array}{l}
                        e_{k+1} \\
                        \NovVariable e_0
                    \end{array}\right. &
                    \begin{array}{l}
                        0 < k < \dimM,  \\
                        k = \dimM. 
                    \end{array}\\
                \EQuantumCohomology_{\CircleAction^{-r}}^\ArbitraryIndex(\Projective^\dimM) \cong \NovikovRing \GradedCompletedTensorProduct \Integers [\uformal] \langle e_k \rangle_{k=0}^\dimM,&
                    & e_1 \underset{-r}{\QuantumProduct} e_k &= \left\{ \begin{array}{l}
                        e_{k+1} - r \uformal e_k \\
                        \NovVariable e_0 - r \uformal e_\dimM 
                    \end{array}\right. &
                    \begin{array}{l}
                        0 < k < \dimM,  \\
                        k = \dimM. 
                    \end{array}
            \end{aligned}
        \end{equation}
        
        Let $\Lifted{\CircleAction}$ be the unique lift of $\CircleAction$ to $\ContractibleLoopSpaceWithFillings{\Projective^\dimM}$ which fixes the point $(\Circle \mapsto e_0, \Disc \mapsto e_0)$.
        It has Maslov index $2\dimM$.
        For $r \ge 0$, the equivariant quantum Seidel map $\EQuantumSeidelMap_{\Lifted{\CircleAction}} : \EQuantumCohomology_{\CircleAction^{-r}}^\ArbitraryIndex(\Projective^\dimM) \to \EQuantumCohomology_{\CircleAction^{-(r+1)}}^{\ArbitraryIndex + 2 \dimM}(\Projective^\dimM)$ is given by
        \begin{equation}
            \label{eqn:equivariant-quantum-seidel-map-for-projective-space}
            \left\{ 
                \begin{aligned}
                    e_0 & \mapsto \sum_{l=0}^\dimM (r+1)^{\dimM - l} \uformal^{\dimM - l} \ e_l \\
                    e_k & \mapsto \sum_{l=0}^{k - 1} (r+1)^{k-1-l} \NovVariable \uformal^{k-1-l} \ e_l \EndFullStop
                \end{aligned}
            \right.
        \end{equation}
        
        \begin{proof}
            Our method of determining the equivariant coefficients in \eqref{eqn:projective-space-all-cohomologies-and-products} and \eqref{eqn:equivariant-quantum-seidel-map-for-projective-space} is to find certain coefficients directly and to deduce the remaining coefficients by repeated application of the intertwining formula \eqref{eqn:quantum-intertwining-seidel-equivariant}.
            We demonstrate our method for $\Projective^2$.
            
            The non-equivariant products and non-equivariant quantum Seidel maps are known for $\Projective^2$ (for example, by \cite[(5.13)]{mcduff_topological_2006}), hence we can write the equivariant quantum products, using unknown integer coefficients, as
            \definecolor{colorForVanishingCoefficients}{gray}{0.5}
            \begin{align}
                e_1 \underset{-r}{\QuantumProduct} e_1 &= e_2 + \alpha_r \uformal e_1 \mathbin{\textcolor{colorForVanishingCoefficients}{+}} \textcolor{colorForVanishingCoefficients}{ \beta_r \uformal^2 e_0} \\
                e_1 \underset{-r}{\QuantumProduct} e_2 &= q e_0 + \gamma_r \uformal e_2  \mathbin{\textcolor{colorForVanishingCoefficients}{+}} \textcolor{colorForVanishingCoefficients}{ \delta_r \uformal^2 e_1 + \epsilon_r \uformal^3 e_0}
            \end{align}
            and the equivariant quantum Seidel map $\EQuantumSeidelMap_{\Lifted{\CircleAction}} : \EQuantumCohomology_{\CircleAction^{-r}}^\ArbitraryIndex(\Projective^2) \to \EQuantumCohomology_{\CircleAction^{-(r+1)}}^{\ArbitraryIndex + 4}(\Projective^2)$ as
            \begin{equation}
                \label{eqn:equivariant-quantum-seidel-map-for-projective-space-with-unknown-coefficients}
                \left\{ 
                    \begin{aligned}
                        e_0 &\mapsto e_2 + A_r \uformal e_1 + B_r \uformal^2 e_0 \\
                        e_1 &\mapsto q e_0  \mathbin{\textcolor{colorForVanishingCoefficients}{+}} \textcolor{colorForVanishingCoefficients}{ C_r \uformal e_2 + D_r \uformal^2 e_1 + E_r \uformal^3 e_0} \\
                        e_2 &\mapsto q e_1 + F_r q \uformal e_0  \mathbin{\textcolor{colorForVanishingCoefficients}{+}} \textcolor{colorForVanishingCoefficients}{ G_r \uformal^2 e_2 + H_r \uformal^3 e_1 + I_r \uformal^4 e_0} \EndFullStop
                    \end{aligned}
                \right.
            \end{equation}
            
            By \autoref{rem:minimal-fixed-sections-of-clutching-bundle}, the only holomorphic section which contributes a $q^0$ term in \eqref{eqn:equivariant-quantum-seidel-map-for-projective-space-with-unknown-coefficients} is the minimal fixed section at $e_0$.
            By using only small perturbations of the invariant Morse function $\MorseFunction_{\Projective^2}$, we immediately deduce $\beta_r ,\delta_r ,\epsilon_r = 0$ and $C_r , D_r , E_r , G_r , H_r , I_r = 0$ (hence fading these terms above) because otherwise the function would increase along its negative gradient trajectories.
            Moreover, the coefficient $B_r$ is a local count of equivariant trajectories that intersect the fixed point $e_0$, and is therefore $(r+1)^2$ since this is the same count as in \autoref{thm:equivariant-quantum-seidel-map-complex-vector-space}.
            Finally, since the Borel space $\InfiniteSphere \times_\Circle \Projective^2$ decomposes as $\InfiniteComplexProjectiveSpace \times \Projective^2$ in the $r=0$ case, we have $\alpha_0, \gamma_0 = 0$.
            
            We apply the intertwining relation \eqref{sec:intertwining-relation-equivariant-quantum-seidel-map} with\footnote{
                There is a degree-2 equivariant cohomology class $\bm{\alpha}$ in the clutching bundle which restricts to $e_1$ at both poles.
            } $\bm{\alpha} = e_1$ and $\bm{x} = e_k$.
            By \autoref{remark:degree-two-weighted-map} and an algebraic topology calculation, the $\bm{\alpha}$-weighted equivariant quantum Seidel map counts the fixed section at $e_0$ with weight 0 and any section corresponding to $q^1$ with weight 1.
            With $\bm{x} = e_0$, we get
            \begin{align}
                \label{intertwining-formula-application-e0}
                0 &= \EQuantumSeidelMap_{\Lifted{\CircleAction}} (e_0 \underset{-r}{\QuantumProduct} e_1) - \EQuantumSeidelMap_{\Lifted{\CircleAction}} (e_0) \underset{-(r+1)}{\QuantumProduct} e_1 - \uformal \EQuantumSeidelMap_{\Lifted{\CircleAction}, e_1} (e_0) \nonumber \\
                &=  \EQuantumSeidelMap_{\Lifted{\CircleAction}} (e_1) - (e_2 + A_r \uformal e_1 + (r+1)^2 \uformal^2 e_0) \underset{-(r+1)}{\QuantumProduct} e_1 - 0 \nonumber \\
                &= q e_0 - (q e_0 + \gamma_{r+1} \uformal e_2) - A_r \uformal(e_2 + \alpha_{r+1} \uformal e_1) - (r+1)^2 \uformal^2 e_1 \nonumber \\
                &= -(\gamma_{r+1} + A_r) \uformal e_2 - (A_r \alpha_{r+1} + (r+1)^2) \uformal^2 e_1 \EndComma
            \intertext{with $\bm{x} = e_1$, we get}
                \label{intertwining-formula-application-e1}
                0 &= \EQuantumSeidelMap_{\Lifted{\CircleAction}} (e_1 \underset{-r}{\QuantumProduct} e_1) - \EQuantumSeidelMap_{\Lifted{\CircleAction}} (e_1) \underset{-(r+1)}{\QuantumProduct} e_1 - \uformal \EQuantumSeidelMap_{\Lifted{\CircleAction}, e_1} (e_1) \nonumber \\
                &= \EQuantumSeidelMap_{\Lifted{\CircleAction}} (e_2 + \alpha_r \uformal e_1) - q e_0 \underset{-(r+1)}{\QuantumProduct} e_1 - \uformal (q e_0) \nonumber \\
                &= (q e_1 + F_r q \uformal e_0) + \alpha_r \uformal (q e_0) - q e_1 - q \uformal e_0 \nonumber \\
                &= (F_r + \alpha_r - 1) q \uformal e_0 \EndComma
            \intertext{and with $\bm{x} = e_2$, we get}
                \label{intertwining-formula-application-e2}
                0 &= \EQuantumSeidelMap_{\Lifted{\CircleAction}} (e_2 \underset{-r}{\QuantumProduct} e_1) - \EQuantumSeidelMap_{\Lifted{\CircleAction}} (e_2) \underset{-(r+1)}{\QuantumProduct} e_1 - \uformal \EQuantumSeidelMap_{\Lifted{\CircleAction}, e_1} (e_2) \nonumber \\
                &= \EQuantumSeidelMap_{\Lifted{\CircleAction}} (q e_0 + \gamma_r \uformal e_2) - (q e_1 + F_r q \uformal e_0) \underset{-(r+1)}{\QuantumProduct} e_1 - \uformal (q e_1 + F_r q \uformal e_0) \nonumber \\
                &= \begin{multlined}[t]
                    q(e_2 + A_r \uformal e_1 + (r+1)^2 \uformal^2 e_0) + \gamma_r \uformal (q e_1 + F_r \uformal q e_0) \\
                    - q (e_2 + \alpha_{r+1} \uformal e_1) - F_r q \uformal e_1 - q \uformal e_1 - F_r q \uformal^2 e_0
                \end{multlined}
                \nonumber \\
                &= (A_r + \gamma_r - \alpha_{r+1} - F_r - 1) q \uformal e_1 + ((r+1)^2 + \gamma_r F_r - F_r) q \uformal^2 e_0 \EndFullStop
            \end{align}
            The coefficients of \eqref{intertwining-formula-application-e0}, \eqref{intertwining-formula-application-e1} and  \eqref{intertwining-formula-application-e2} yield the following silmultaneous equations.
            \begin{align}
                \label{silmultaneous-equation-gamma-A}
                    \gamma_{r+1} &= -A_r \\
                \label{silmultaneous-equation-A-alpha}
                    A_r \alpha_{r+1} &= - (r+1)^2 \\
                \label{silmultaneous-equation-F}
                    F_r &= - \alpha_r + 1 \\
                \label{silmultaneous-equation-A-alpha-F}
                    \alpha_{r+1} - A_r &= \gamma_r - 1 - F_r \\
                \label{silmultaneous-equation-F-2}
                    (\gamma_r - 1) F_r &= - (r+1)^2
            \end{align}
            By induction, set $\alpha_r, \gamma_r = -r$.
            Either \eqref{silmultaneous-equation-F} or \eqref{silmultaneous-equation-F-2} yields $F_r = r+1$.
            The unique solution to \eqref{silmultaneous-equation-A-alpha} and \eqref{silmultaneous-equation-A-alpha-F} is $\alpha_{r+1} = -(r+1)$ and $A_r = r+1$.
            Finally, \eqref{silmultaneous-equation-gamma-A} yields $\gamma_{r+1} = -(r+1)$.
            This proves the induction hypothesis $\alpha_{r+1}, \gamma_{r+1} = -(r+1)$ and hence completes the proof.
        \end{proof}
        
        \begin{remark}
            [Inverse action]
            \label{rem:example-projective-space-inverse-of-circle-action-equivariant-quantum-seidel-map}
            The inverse circle action $\CircleAction \Inverse$ is another Hamiltonian circle action for which we can compute the equivariant quantum Seidel map.
            We have $\EQuantumSeidelMap_{\Lifted{\CircleAction}} \EQuantumSeidelMap_{\Lifted{\CircleAction}\Inverse} = \EQuantumSeidelMap_{\Lifted{\CircleAction}\Lifted{\CircleAction} \Inverse} = \Identity$, so we can deduce the map $\EQuantumSeidelMap_{\Lifted{\CircleAction}\Inverse}$ by inverting \eqref{eqn:equivariant-quantum-seidel-map-for-projective-space}.
            Thus the map $\EQuantumSeidelMap_{\Lifted{\CircleAction}\Inverse} :\EQuantumCohomology_{\CircleAction^{-(r+1)}}^{\ArbitraryIndex}(\Projective^\dimM) \to \EQuantumCohomology_{\CircleAction^{-r}}^{\ArbitraryIndex - 2 \dimM}(\Projective^\dimM)$ is given by
            \begin{equation}
                \left\{ 
                    \begin{aligned}
                        e_0 & \mapsto \NovVariable \Inverse \ e_1, \\
                        e_k & \mapsto -(r+1) \uformal \NovVariable \Inverse \ e_k + \NovVariable \Inverse \ e_{k+1}, \qquad 1 \le k < \dimM \\
                        e_\dimM & \mapsto -(r+1) \uformal \NovVariable \Inverse \ e_\dimM + e_0.
                    \end{aligned}
                \right.
            \end{equation}
            The assignment $e_0 \mapsto \NovVariable \Inverse \ e_1$ may be deduced directly via minimal fixed sections using \autoref{rem:minimal-fixed-sections-of-clutching-bundle}.
            Here, $e_1$ is the Poincar\'{e} dual of the subset $\Set{z_0 = 0}$, which is the minimal locus of $-\CircleHam$.
            No sections other than these minimal fixed sections can contribute to $\EQuantumSeidelMap_{\Lifted{\CircleAction}\Inverse}(e_0)$ for degree reasons.
        \end{remark}
        
    \subsection{Tautological line bundle on projective space}
    
        \label{sec:example-of-tautological-line-bundle-on-projective-space}
        
        The total space $\tautLB{n}$ of the tautological line bundle over projective space $\Projective^n$ is one of the negative line bundles studied by Ritter \cite{ritter_floer_2014}.
        The elements of $\tautLB{n}$ are of the form $([\mathbf{z}], \mathbf{v}) = ([z_0:\ldots:z_n], (v_0, \ldots, v_n)) \in \Projective^n \times \ComplexNumbers^{n+1}$, where $\mathbf{z}$ and $\mathbf{v}$ are linearly dependent.
        Denote by $\ZeroSection$ the image of the zero section, giving $Z = \Set{\mathbf{v} = 0} \cong \Projective^n$.
        
        There is a symplectic form $\SymplecticForm$ on $\tautLB{n}$ such that $(\tautLB{n}, \SymplecticForm)$ is a monotone convex symplectic manifold whose fibres and whose submanifold $Z$ are symplectic submanifolds \cite[Section~7]{ritter_floer_2014}.
        Its Novikov ring is $\NovikovRing = \Integers [\NovVariable, \NovVariable\Inverse]$, where $\NovVariable$ is a formal variable of degree $2n$.
        
        Let $\CircleAction$ be the linear Hamiltonian circle action given by $\CircleElement \cdot ([\mathbf{z}], \mathbf{v}) = ([\mathbf{z}], \ExponentialNumber^{2 \PiNumber \ImaginaryNumber \CircleElement}\mathbf{v})$.
        It rotates the fibres and its fixed point set is $\ZeroSection$.
        Set $\Lifted{\CircleAction}$ to be the unique lift of $\CircleAction$ which fixes the points $(\Circle \mapsto ([\mathbf{z}], 0), \Disc \mapsto ([\mathbf{z}], 0))$.
        The Maslov index of $\Lifted{\CircleAction}$ is 2.
        
        The invariant Morse function given by $\MorseFunction (([\mathbf{z}], \mathbf{v})) = \sum_{k=0}^\dimM k \Modulus{z_k}^2 + \Modulus{v_k}^2$ has critical points at each of the unit vectors in $\ZeroSection$.
        Denote the $k$-th such critical point\footnote{
            Orient $\UnstableManifold(e_k)$ according to the natural complex structure.
        } by $e_k$.
        It has Morse index $\MorseIndex(e_k; \MorseFunction) = 2k$.
        Give $\tautLB{n}$ the metric which is the restriction of the standard metric on $\ComplexNumbers^{n+1} \times \ComplexNumbers^{n+1}$, so that we recover the Fubini-Study metric along the zero section.
        Thus the Morse cohomology of $\tautLB{n}$ is
        \begin{align}
            \Cohomology^\ArbitraryIndex(\tautLB{n}) \cong \Integers \langle e_k \rangle_{k=0}^n,&
            & e_1 \CupProduct e_k &= \left\{ \begin{array}{l}
                e_{k+1}  \\
                0
            \end{array}\right. &
            \begin{array}{l}
                k < n,  \\
                k = n. 
            \end{array}
        \end{align}
        Since $\tautLB{n}$ equivariantly retracts onto $\ZeroSection$ (with the trivial circle action), its equivariant cohomology $\ECohomology_{\CircleAction^r}^\ArbitraryIndex(\tautLB{n})$ is ring isomorphic to $\Integers [\uformal] \Tensor_\Integers \Cohomology^\ArbitraryIndex (\tautLB{n})$.
        
        The quantum cohomology is
        \begin{align}
            \QuantumCohomology^\ArbitraryIndex(\tautLB{n}) \cong \NovikovRing \langle e_k \rangle_{k=0}^n,&
                    & e_1 \QuantumProduct e_k &= \left\{ \begin{array}{l}
                        e_{k+1} \\
                        - \NovVariable e_1
                    \end{array}\right. &
                    \begin{array}{l}
                        k < n,  \\
                        k = n.
                    \end{array}
        \end{align}
        The quantum product is deduced via \eqref{eqn:quantum-intertwining-seidel} from the the quantum Seidel map which Ritter computed \cite[Lemma~60]{ritter_floer_2014}:
        \begin{equation}
            \label{eqn:quantum-seidel-map-tautological-line-bundle}
            \QuantumSeidelMap_{\Lifted{\CircleAction}}(e_k) = \begin{cases}
                -e_{k+1} & k < n , \\
                \NovVariable e_1 & k = n.
            \end{cases}
        \end{equation}
        
        The equivariant quantum cohomology is $\NovikovRing \GradedCompletedTensorProduct \Integers [\uformal] \langle e_k \rangle_{k=0}^n$.
        The equivariant quantum product and equivariant quantum Seidel maps are given by the following theorems.
        
        \begin{theorem}
            \label{thm:tautological-line-bundle-equivariant-quantum-product}
            For $r \ge 0$, the equivariant quantum product on $\EQuantumCohomology^\ArbitraryIndex_{\CircleAction^{-r}}(\tautLB{n})$ is given by
            \begin{equation}
                \label{eqn:equivariant-quantum-product-on-tautological-line-bundle}
                e_1 \underset{-r}{\QuantumProduct} e_k = \begin{cases}
                    e_{k+1} & k < n \EndComma \\
                    - \NovVariable e_1 + r \uformal \NovVariable e_0 & k = n \EndFullStop
                \end{cases}
            \end{equation}
        \end{theorem}
        
        \begin{theorem}
            \label{thm:tautological-line-bundle-equivariant-seidel-map}
            For $r \ge 0$, the equivariant quantum Seidel map
            \begin{equation}
                \label{eqn:equivariant-seidel-map-with-domain-tautological-line-bundle}
                \EQuantumSeidelMap_{\Lifted{\CircleAction}} : \EQuantumCohomology^\ArbitraryIndex_{\CircleAction^{-r}}(\tautLB{n}) \to \EQuantumCohomology^{\ArbitraryIndex + 2}_{\CircleAction^{-(r+1)}}(\tautLB{n})
            \end{equation}
            is given by
            \begin{equation}
                \label{eqn:equivariant-quantum-seidel-map-on-tautological-line-bundle}
                e_k \mapsto \begin{cases}
                    -e_{k+1} + (r+1) \uformal e_k & k < n \EndComma \\
                    \NovVariable e_1 + (r+1) \uformal e_n - (r+1) \uformal \NovVariable e_0 & k = n \EndFullStop
                \end{cases}
            \end{equation}
        \end{theorem}
        
        \begin{proof}
            [Proof of \autoref{thm:tautological-line-bundle-equivariant-quantum-product} and \autoref{thm:tautological-line-bundle-equivariant-seidel-map}]
            We use exactly the same method as that we used in \autoref{sec:example-of-projective-space}.
            Write the equivariant quantum products and equivariant quantum Seidel maps using variables for the unknown coefficients.
            By using only a small perturbation of the Morse function, deduce that all coefficients are zero apart from those which occur above.
            The coefficients of the $\uformal e_k$ terms and the $\uformal e_n$ term in \eqref{eqn:equivariant-quantum-seidel-map-on-tautological-line-bundle} are all $(r+1)$ because the fibre locally resembles $\ComplexNumbers$.
            Apply the intertwining formula \eqref{eqn:quantum-intertwining-seidel-equivariant} to the elements $e_{n-1}$ and $e_n$ to deduce the two remaining coefficients.
        \end{proof}
        
        \subsubsection{Deducing equivariant symplectic cohomology}
        
            \label{sec:deducing-equivariant-symplectic-cohomology-tautological-line-bundle}
        
            The symplectic cohomology is the limit of the quantum Seidel map \eqref{eqn:quantum-seidel-map-tautological-line-bundle}.
            We have $\QuantumSeidelMap_{\Lifted{\CircleAction}} (e_n + \NovVariable e_0) = 0$, and $\QuantumSeidelMap_{\Lifted{\CircleAction}}$ is an isomorphism after quotienting by $e_n + \NovVariable e_0$.
            Thus we get $\NovikovRing$-algebra and $\NovikovRing$-module isomorphisms
            \begin{equation}
                \label{eqn:symplectic-cohomology-on-tautological-line-bundle}
                \SymplecticCohomology^\ArbitraryIndex(\tautLB{n}) \underset{\text{alg.}}{\cong} \frac{\Integers [\NovVariable, \NovVariable\Inverse] \langle e_k \rangle_{k=0}^n}{(e_n + \NovVariable e_0)} \underset{\text{mod.}}{\cong} \bigoplus_{k=0}^{n-1} \Integers [\NovVariable, \NovVariable\Inverse] \langle e_k \rangle \EndFullStop
            \end{equation}
            In particular, in every even degree it has one copy of $\Integers$.
            
            The equivariant quantum Seidel map $\EQuantumSeidelMap_{\Lifted{\CircleAction}}$ from \eqref{eqn:equivariant-seidel-map-with-domain-tautological-line-bundle} is injective, and its determinant is $\Determinant \EQuantumSeidelMap_{\Lifted{\CircleAction}} = (r+1)^{n+1} \uformal^{n+1}$.
            To find the direct limit of the maps $\EQuantumSeidelMap_{\Lifted{\CircleAction}}$, we must use a different strategy.
            
            \newcommand{\Localise}{\text{local}}
            We localise the ring $\NovikovRing \GradedCompletedTensorProduct \Integers [\uformal]$ so the determinants are invertible.
            We denote this localisation by $\NovikovRing \GradedCompletedTensorProduct \Integers [\uformal] _\Localise$.
            The degree-$l$ elements in $\NovikovRing \GradedCompletedTensorProduct \Integers [\uformal] _\Localise$ are given as per the graded completed tensor product \eqref{eqn:homogeneous-elements-of-graded-completed-tensor-product}, except finitely-many negative powers of $\uformal$ are permitted, and we tensor with $\RationalNumbers$.
            Denote by $\EQuantumCohomology^\ArbitraryIndex _{\CircleAction^0}(\tautLB{n}) _{\Localise}$ the tensor product $\EQuantumCohomology^\ArbitraryIndex _{\CircleAction^0} (\tautLB{n}) \Tensor \left( \NovikovRing \GradedCompletedTensorProduct \Integers [\uformal] _\Localise \right)$.
            
            Consider the following commutative diagram.
            \begin{equation}
                \begin{tikzcd}[row sep=large]
                    \EQuantumCohomology^\ArbitraryIndex _{\CircleAction^0} (\tautLB{n}) \arrow[d, hook] \arrow[r, "\EQuantumSeidelMap_{\Lifted{\CircleAction}}"] 
                    & \EQuantumCohomology^{\ArbitraryIndex +2} _{\CircleAction^{-1}} (\tautLB{n}) \arrow[ld, dashed, "\EQuantumSeidelMap_{\Lifted{\CircleAction}} \Inverse"] \arrow[r, "\EQuantumSeidelMap_{\Lifted{\CircleAction}}"] 
                    & \EQuantumCohomology^{\ArbitraryIndex +2} _{\CircleAction^{-2}} (\tautLB{n}) \arrow[lld, dashed, bend left, "\EQuantumSeidelMap_{\Lifted{\CircleAction}} \Inverse \ComposedWith \EQuantumSeidelMap_{\Lifted{\CircleAction}} \Inverse"]\\
                    \EQuantumCohomology^\ArbitraryIndex _{\CircleAction^0}(\tautLB{n}) _{\Localise} & & 
                \end{tikzcd}
            \end{equation}
            The direct limit of the maps $\EQuantumSeidelMap_{\Lifted{\CircleAction}}$ is the image of the (injective) dashed maps above.
            This gives
            \begin{equation}
                    \ESymplecticCohomology^\ArbitraryIndex _{\CircleAction^0} (\tautLB{n}) \cong \bigcup_{p=0}^\Infinity \Set{\operatorname{image}\left(
                    (\EQuantumSeidelMap_{\Lifted{\CircleAction}} \Inverse)^p : \EQuantumCohomology^{\ArbitraryIndex + 2p}_{\CircleAction^{-p}} \to \EQuantumCohomology^\ArbitraryIndex _{\CircleAction^0\CommaSpace \Localise}
                    \right)} \EndFullStop
            \end{equation}
            We immediately deduce that $\ESymplecticCohomology^\ArbitraryIndex _{\CircleAction^0} (\tautLB{n})$ is not a finitely-generated $\NovikovRing \GradedCompletedTensorProduct \Integers [\uformal]$-module because this is a strictly increasing chain of submodules.
            Moreover, we can follow the element $e_n$ under the maps $\EQuantumSeidelMap_{\Lifted{\CircleAction}}$ to get
            \begin{equation}
            \begin{multlined}
                e_n \mapsto \NovVariable e_1 + \SmallO(\uformal) \mapsto \pm \NovVariable e_2 + \SmallO(\uformal) \mapsto \cdots \mapsto \pm \NovVariable e_n + \SmallO(\uformal) \\ \mapsto \pm \NovVariable^2 e_1 + \SmallO(\uformal) \mapsto \cdots \EndComma
            \end{multlined}
            \end{equation}
            which implies that $e_n$ is not divisible by $\uformal$ in the direct limit (for none of the images are divisible by $\uformal$).
            Therefore $\ESymplecticCohomology^\ArbitraryIndex _{\CircleAction^0} (\tautLB{n})$ is a proper submodule of the localisation.
        
        \subsubsection{Finding generators}
        
            Recall the \define{adjugate} (or \define{adjoint}) of a nonsingular matrix $A$ is the unique matrix $A^\ast$ such that $A ^\ast A = A A^\ast = \Determinant(A) \Identity$, so that if the inverse of $A$ exists, it is $A \Inverse = \frac{1}{\Determinant(A)} A^\ast$.
            Therefore, to find the image of the inverses $(\EQuantumSeidelMap_{\Lifted{\CircleAction}} \Inverse)^p$, we find the image of the adjugates (which is a submodule of $\EQuantumCohomology^\ArbitraryIndex _{\CircleAction^0} (\tautLB{n})$ without localisation), and then divide these elements by the determinants.
            
            We have $\Determinant(\EQuantumSeidelMap_{\Lifted{\CircleAction}} ^p) = \prod_{r=0}^{p-1} ((r+1)^{n+1} \uformal^{n+1})$.
            We characterise the image of the adjugate below.
            
            Using \eqref{eqn:equivariant-quantum-seidel-map-on-tautological-line-bundle}, we have $\EQuantumSeidelMap_{\Lifted{\CircleAction}} (e_n + \NovVariable e_0) = (r+1) \uformal e_n$.
            Thus, with respect to the ordered basis $\langle e_n + \NovVariable e_0, e_1, e_2, \ldots, e_{n-1}, e_n \rangle$, the map $\EQuantumSeidelMap_{\Lifted{\CircleAction}}$ is given by the following matrix.
            \begin{equation}
                \label{eqn:matrix-of-equivariant-quantum-seidel-map-tautological-line-bundle}
                A_r =
                \begin{pmatrix}
                    0 & 0 & 0 & \cdots & 0 & -(r+1) \uformal \\
                    0 & (r+1) \uformal & 0 & \cdots & 0 &\NovVariable \\
                    0 & -1 & (r+1) \uformal & & 0 & 0 \\
                    \vdots & \ddots & \ddots & \ddots & & \vdots \\
                    0 & 0 & 0 & \ddots & (r+1) \uformal & 0 \\
                    (r+1) \uformal & 0 & 0 & & -1 & 2(r+1) \uformal
                \end{pmatrix}
            \end{equation}
            Note that if we permute the first and last column, the matrix \eqref{eqn:matrix-of-equivariant-quantum-seidel-map-tautological-line-bundle} becomes lower triangular.
            Set $x_k^p = A_0^\ast \cdots A_{p-1}^\ast e_k$ for $k = 1, \ldots, n$, and set $x_0^p = A_0^\ast \cdots A_{p-1}^\ast (e_n + \NovVariable e_0)$.
            Thus the image of $(\EQuantumSeidelMap_{\Lifted{\CircleAction}} ^\ast)^p$ above is the span of $\Set{x_0^p, \ldots, x_n^p}$.
            Using \eqref{eqn:matrix-of-equivariant-quantum-seidel-map-tautological-line-bundle}, we can set up a recursive formula for the $x_k^p$, which starts like
            \begin{equation}
                \label{eqn:system-of-recursive-relations-between-generators-tautological-line-bundle}
                \left\{
                \begin{array}{l}
                    x_n^{p+1} = (p+1)^{n} \uformal^{n} x_0^p \\
                    x_{n-1}^{p+1} = (p+1)^{n} \uformal^{n} x_{n-1}^p + (p+1)^{n-1} \uformal^{n-1} x_0^p \\
                    \ldots
                \end{array}
                \right.
            \end{equation}
            
            \begin{example}
                [$n=1$]
                For $\tautLB{1}$, the system \eqref{eqn:system-of-recursive-relations-between-generators-tautological-line-bundle} becomes
                \begin{equation}
                    \label{eqn:system-of-recursive-relations-between-generators-tautological-line-bundle-n-1}
                    \left\{
                \begin{array}{l}
                    x_1^{p+1} = (p+1) \uformal x_0^p \\
                    x_0^{p+1} = (\NovVariable + 2(p+1) \uformal) x_0^p - (p+1) \uformal x_1^p
                \end{array}
                \right.
                \end{equation}
                and has solution $x_0^p = \NovVariable^{p-1}((\NovVariable + p (p+1) \uformal)e_0 - \uformal e_1) + \SmallO(\uformal^2)$.
                Since $x_1^{p+1} = (p+1) \uformal x_0^p$, one copy of $(p+1)\uformal$ cancels upon division by the determinant.
                This yields the presentation
                \begin{equation}
                    \ESymplecticCohomology^\ArbitraryIndex _{\CircleAction^0} (\tautLB{1}) \cong \NovikovRing \GradedCompletedTensorProduct \Integers [\uformal] \left\langle \frac{x_0^p}{1^2 \cdots p^2 (p+1) \uformal^{2p + 1}} \right\rangle_{p=1}^\Infinity \EndFullStop
                \end{equation}
            \end{example}
            
            \begin{remark}
                [Possible nicer presentation]
                The author was unable to establish whether there is an element $X \in \ESymplecticCohomology^\ArbitraryIndex _{\CircleAction^0} (\tautLB{n})$ which is divisible by every power of $\uformal$, or indeed by every determinant.
                If so, this would produce an isomorphism
                \begin{equation}
                    \ESymplecticCohomology^\ArbitraryIndex _{\CircleAction^0} (\tautLB{1}) \cong \NovikovRing \GradedCompletedTensorProduct \Integers [\uformal] \oplus \NovikovRing \GradedCompletedTensorProduct \Integers [\uformal] _\Localise
                \end{equation}
                in the $n=1$ case, and similar isomorphisms for $n > 1$.
            \end{remark}

\renewcommand{\doitext}{}
\bibliography{main}
\bibliographystyle{halpha}

\end{document}